\documentclass[11pt]{amsart}
\usepackage{geometry}                % See geometry.pdf to learn the layout options. There are lots.
\usepackage{graphicx}
\usepackage{amssymb}
\usepackage{epstopdf}
\usepackage{amsmath}
\usepackage{color}
\usepackage{hyperref}
\usepackage{pdfsync}
\usepackage{tikz}
\usepackage{soul}
 
\usepackage[all]{xy}
\xyoption{matrix}
\xyoption{arrow}

\usepackage{tikz}

\usetikzlibrary{arrows,decorations.pathmorphing,backgrounds,positioning,fit,petri}

\newtheorem{thm}{Theorem}[subsection]
\newtheorem{lem}[thm]{Lemma}
\newtheorem{cor}[thm]{Corollary}
\newtheorem{prop}[thm]{Proposition}

\newtheorem{innercustomthm}{Theorem}
\newenvironment{customthm}[1]{\renewcommand\theinnercustomthm{#1}\innercustomthm}{\endinnercustomthm}

\newtheorem{innercustomlem}{Lemma}
\newenvironment{customlem}[1]
{\renewcommand\theinnercustomlem{#1}\innercustomlem}{\endinnercustomlem}

\newtheorem{innercustomcor}{Corollary}
\newenvironment{customcor}[1]{\renewcommand
\theinnercustomcor{#1}\innercustomcor}{\endinnercustomcor}

\theoremstyle{definition}
\newtheorem{defn}[thm]{Definition}
\newtheorem{eg}[thm]{Example}

 \newtheorem{notation}[thm]{Notation}

\newtheorem{rem}[thm]{Remark}
 
\numberwithin{equation}{section}
 
 \tikzset{help lines/.style={step=#1cm,very thin, color=gray},
help lines/.default=.5} % draws a grid spaced #1 cm
\tikzset{thick grid/.style={step=#1cm,thick, color=gray},
thick grid/.default=1} % draws a grid spaced #1 cm

\newcommand{\vs}[1]{\vskip .#1 cm} %enter amount of skip wanted at #1
\newcommand{\noi}{\noindent}

\newcommand{\xrarrow}{\xrightarrow} %right arrow {label on top}

\newcommand{\then}{\Rightarrow}

\newcommand{\into}{\hookrightarrow}
 \newcommand{\onto}{\twoheadrightarrow}
 \newcommand{\cof}{\rightarrowtail}

 \newcommand{\brk}[1]{\left<#1\right>}

\DeclareMathOperator{\Hom}{Hom}%
\DeclareMathOperator{\Ext}{Ext}%
\DeclareMathOperator{\End}{End}%
\DeclareMathOperator{\interior}{int}  % interior
\newcommand\undim{\underline\dim\,}
 \newcommand\und{\underline}

\newcommand{\field}[1]{\mathbb{#1}}
\newcommand{\ZZ}{\ensuremath{{\field{Z}}}}
\newcommand{\CC}{\ensuremath{{\field{C}}}}
\newcommand{\RR}{\ensuremath{{\field{R}}}}
\newcommand{\QQ}{\ensuremath{{\field{Q}}}}

\newcommand{\commentout}[1]{}

\newcommand{\cA}{\ensuremath{{\mathcal{A}}}}

\newcommand{\cB}{\ensuremath{{\mathcal{B}}}}

\newcommand{\cP}{\ensuremath{{\mathcal{P}}}}
\newcommand{\cQ}{\ensuremath{{\mathcal{Q}}}}
\newcommand{\cR}{\ensuremath{{\mathcal{R}}}}
\newcommand{\cS}{\ensuremath{{\mathcal{S}}}}

\newcommand{\cU}{\ensuremath{{\mathcal{U}}}}
\newcommand{\cV}{\ensuremath{{\mathcal{V}}}}
\newcommand{\cW}{\ensuremath{{\mathcal{W}}}}

\newcommand{\cX}{\ensuremath{{\mathcal{X}}}}
\newcommand{\cY}{\ensuremath{{\mathcal{Y}}}}

\title{Picture groups and maximal green sequences}
\author{Kiyoshi Igusa}
\address{Department of Mathematics, Brandeis University, Waltham, MA 02454}\email{igusa@brandeis.edu}
\author{Gordana Todorov}
\address{Department of Mathematics, Northeastern University, Boston, MA 02115}\email{g.todorov@northeastern.edu}


\begin{document}

\begin{abstract}
We show that picture groups are directly related to maximal green sequences for valued Dynkin quivers of finite type. Namely, there is a bijection between maximal green sequences and positive expressions (words in the generators without inverses) for the Coxeter element of the picture group. We actually prove the theorem for the more general set up of finite ``vertically and laterally ordered'' sets of positive real Schur roots for any hereditary algebra (not necessarily of finite type).

Furthermore, we show that every picture for such a set of positive roots is a linear combination of ``atoms'' and we give a precise description of atoms as special semi-invariant pictures.
\end{abstract}

\maketitle

\tableofcontents

%\listoffigures

 %\newpage
 
The theory of pictures and picture groups comes from topology and goes back to the first author's PhD thesis \cite{I:thesis} where pictures were used to represent elements in the higher algebraic $K$-theory group $K_3\ZZ[\pi]$ and these were used to obtain obstructions to the 1-parameter ``pseudoisotopy implies isotopy'' problem, extending work of Allen Hatcher and John Wagoner who showed the relation between $K_2\ZZ[\pi]$ and pseudoisotopy \cite{HW}. The salient fact is that pictures for a group $G$ represent elements of $H_3(G)$ and, by an observation of S.M. Gersten \cite{Gersten}, $K_3R=H_3(St(R))$ for any ring $R$ (and $K_2R$ is the center of the Steinberg group $St(R)$). Pictures are also known as ``spherical diagrams'' in some text books on combinatorial group theory \cite{LS}. They are also called ``Peiffer diagrams'' and generally attributed to Ren\'ee Peiffer who gave the algebraic definition of pictures in \cite{Peiffer}. (See Theorem \ref{eq: relation for Q(G)}.) ``Partial pictures'' were used by John Wagoner \cite{JW} to describe the boundary map in the long exact $K$-theory sequence of an ideal. Later, in \cite{IK}, John Klein and the first author used Morse theory to construct a picture representing nontrivial elements of $K_3\ZZ[\zeta]$ where $\zeta^n=1$. In the case $n=2, \zeta=-1$, this picture was shown in \cite{I:dilog1} to give a generator of $K_3\ZZ=\ZZ/48$. Later, in joint work with Kent Orr \cite{IOr}, the first author used pictures to obtain new results on Milnor's $\overline\mu$ link invariants.

Pictures also appeared in representation theory and combinatorics. Harm Derksen and Jerzy Weyman used what we now call ``semi-invariant pictures'' for acyclic quivers to give a new proof of the saturation conjecture for Littlewood-Richardson coefficients \cite{DW}. This was based on the seminal work of Alastair King \cite{King} using geometric invariant theory to study representations of quivers. Around this time, Sergey Fomin and Andrei Zelevinsky invented cluster algebras \cite{FZ1}, \cite{FZ2}. Soon after that, Aslak~Bakke Buan, Robert Marsh, Markus Reineke, Idun Reiten and the second author in \cite{BMRRT} associated to any acyclic cluster algebra a ``cluster category'' whose rigid indecomposable objects correspond to the cluster variables. Following this breakthrough, a large body of knowledge has been accumulated \cite{BMR}, \cite{DWZ1}, \cite{DWZ2}, \cite{A}. This includes work of the authors, together with Kent Orr and Jerzy Weyman giving the connection between cluster theory of acyclic quivers and semi-invariant pictures in \cite{IOTW}, \cite{Modulated}, \cite{IOTW4}. 

Later, Takahide Adachi, Osamu Iyama, and Idun Reiten \cite{AIR} extended cluster theory to arbitrary finite dimensional algebras using $\tau$-rigid objects in place of cluster variables. In \cite{Modulated} the real Schur roots of an acyclic valued quiver are shown to be $c$-vectors for the associated cluster algebra and labels for the ``walls'' $D(\beta)$ for the ``wall-and-chamber'' structure for hereditary algebras (given by the semi-invariant pictures) in \cite{Modulated}, \cite{IOTW4}. In \cite{BST}, Thomas Br\"ustle, David Smith and Hipolito Treffinger extend the ``wall-and-chamber'' structure to arbitrary finite dimensional algebras using $\tau$-tilting theory and the space of stability conditions following Tom Bridgeland \cite{Br}.

Maximal green sequences were invented by Bernhard Keller to study Kontsevich-Soebelman's version \cite{KS1}, \cite{KS2} of the Donaldson-Thomas invariants. What we call ``linear'' maximal green sequences in \cite{Linearity1} were earlier used by Markus Reineke for similar formulas in \cite{Reineke}. Originally, a maximal green sequence was defined to be a sequence of ``green'' mutations of an initial ``seed'' for a cluster algebra, a mutation in the $k$-direction being ``green'' if the $k$-th $c$-vector is positive. However, here we use the representation theoretic version where $c$-vectors $\beta_i$ are replaced with ``wall-crossing'' through the walls $D(\beta_i)$. See \cite{Linearity1}, \cite{Linearily2} for a detailed explanation of how these and other versions of the definition of maximal green sequences are related. 

One of the big questions which we would like to understand is the conjecture that there are only finitely many maximal green sequences (possibly none) for any exchange matrix. In particular, this is still not known for arbitrary acyclic exchange matrices. These correspond to maximal wall crossing sequences in the wall-and-chamber structure for a hereditary algebra. This version of the conjecture has been verified for tame hereditary algebras by Thomas Br\"ustle, Gregoire Dupont and Matthieu P\'erotin \cite{BDP} and extended to cluster-tilted algebras of tame type by Thomas Br\"ustle, Stephen Hermes and the authors in \cite{BHIT}.

In the present paper we have two objectives. The first is to devise a method for attacking hereditary algebras of infinite type by looking at a finite ``admissible'' set of exceptional modules $M_\beta$ where $\beta\in\cS$, an admissible sequence of roots. (We recall in the Appendix the definition of exceptional modules and the fact that they are uniquely determined by their dimension vectors which are real Schur roots $\beta$.) The idea is to study maximal green sequences, which are given by sequences of wall crossing, by looking only at the subsequence of those walls $D(\beta)$ where $\beta\in \cS$. Such subsequences are examples of ``maximal $\cS$-green sequences'' (Definition \ref{def: MGS}) that we consider in greater generality in this paper. Since a maximal green sequence cannot cross the same wall twice \cite{BHIT}, these subsequences will be bounded in length by the size of $\cS$. The same holds for more general maximal $\cS$-green sequences by Remark \ref{rem: only finitely many S-MGS}.

The second objective is to determine exactly which such sequences will occur using the ``picture group.'' The main theorem of this paper, Theorem \ref{thm A} is that, for admissible $\cS$, maximal $\cS$-green sequences are in bijection with positive expressions for the ``Coxeter element'' $c_\cS$ in the ``picture group'' $G(\cS)$ (Definition \ref{def: picture group G(S)}). We observe that there is at least one maximal $\cS$-green sequence given by the ``Coxeter green sequence'' (Remark \ref{rem: Coxeter path}). Also, there are only finitely many maximal $\cS$-green sequences by Remark \ref{rem: only finitely many S-MGS}.

Section 1 has the definitions of ``admissible'' and ``weakly admissible'' sets of real Schur roots which are ``laterally'' and ``(weakly) vertically ordered'' sets of roots. These are concepts introduced in this paper for the purpose of using finite methods to study infinite sets of representations. 

Subsection \ref{ss1.2: statement of results} has a statement of the main results Theorem \ref{thm A} and Corollary \ref{cor B}. In subsection \ref{ss1.3: outline of proof of main theorem} an outline of the proof of Theorem \ref{thm A} is given using three lemmas \ref{lem C: beta m occurs a fixed number of times}, \ref{lem D: properties of compartments used in proof}, \ref{lem E: beta m only commutes with hom orthogonal roots}.

Section \ref{sec2: Compartments} contains a discussion of ``compartments''. Whereas the union of all the walls $D(\beta)$ divides Euclidean space $\RR^n$ into possibly infinitely many ``chambers'', since we consider  only finitely many walls, $D(\beta)$ for $\beta\in \cS$, the complement of these walls has only finitely many regions which we call ``compartments''. In subsection \ref{ss2.3: proof of lemma D} we prove Lemma \ref{lem D: properties of compartments used in proof} which describes the sequence of compartments in a maximal $\cS$-green sequence.

The theory of pictures and picture groups is explained in Section \ref{sec3: Pictures}. The main theorem for pictures is the ``Atomic Deformation Theorem'' (Theorem \ref{Atomic Deformation theorem}) which says that every picture for the picture group $G(\cS)$ has an ``atomic deformation'' to the empty picture, equivalently, any picture is equivalent to a disjoint union of ``atoms''. This idea comes from \cite{IOr} where a similar theorem is proved leading to results in topology. The Atomic Deformation Theorem is used to prove the last two lemmas \ref{lem C: beta m occurs a fixed number of times} and \ref{lem E: beta m only commutes with hom orthogonal roots} completing the proof of the main result.

Section \ref{ss4} is the Appendix which reviews the well-known properties of exceptional sequences, real Schur roots and wide subcategories used in this paper.
 
\section{Introduction}\label{sec: Introduction}
 
 In this section we give the basic definitions used in this paper which are the concepts of ``vertical'' (Def. \ref{def: vertical ordering}) and ``lateral'' orderings (Def. \ref{def: lateral ordering}) of real Schur roots and ``admissible'' sequences of such roots (Def. \ref{def: admissible sequence}). We define ``pictures'' (Fig. \ref{Figure99}) for the ``picture group'' $G(\cS)$ for $\cS$ (Def. \ref{def: picture group G(S)}) and we define ``maximal $\cS$-green sequences'' (Def. \ref{def: MGS}). We also give the statements of the main theorems \ref{thm A}, \ref{cor B} and an outline of the proofs using three lemmas about pictures \ref{lem C: beta m occurs a fixed number of times}, \ref{lem D: properties of compartments used in proof}, \ref{lem E: beta m only commutes with hom orthogonal roots}. Corollary \ref{cor B} is the special case of Theorem \ref{thm A} for a hereditary algebra of finite representation type with the admissible set of roots being all positive roots of the root system. In this special case, we obtain a group theoretic description of all maximal green sequences. We believe that, using $\tau$-tilting theory, analogous statements can be obtained for any finite dimensional algebra in particular those which are $\tau$-tilting finite. However, in this paper, all quivers will be valued quivers without oriented cycles, possibly of infinite type. See the Appendix for basic background material on exceptional modules and real Schur roots.
 
\subsection{Basic definitions}\label{ss1.1: basic definitions} We assume that $Q$ is a valued acyclic quiver and we always consider subsets $\cS$ of the set of positive real Schur roots of $Q$. 
The positive real Schur roots are precisely the dimension vectors of the exceptional modules over any modulated quiver with underlying valued quiver $Q$ \cite{Modulated}. We need to order the roots, in these subsets $\cS$, in two different ways which we call ``lateral'' and ``vertical'' ordering. 

\begin{notation}
Let $Q$ be a modulated quiver and $\beta$ a positive real Schur root. 
\begin{enumerate}
\item 
We will denote by $M_{\beta}$ the unique exceptional module with dimension vector equal to $\beta$. Exceptional modules are always indecomposable.
\item A positive real Schur root $\beta'$ will be called  a \und{subroot} of $\beta$ if the exceptional module $M_{\beta'}$ is isomorphic to a submodule of $M_{\beta}$. This is denoted by $\beta'\subset \beta$.
\item A positive real Schur root $\beta''$ will be called  a \und{quotient root} of $\beta$ if the exceptional module $M_{\beta''}$ is isomorphic to a quotient  of $M_{\beta}$.  This is denoted by $\beta\twoheadrightarrow \beta''$.
\end{enumerate}
\end{notation} 

 \begin{defn}\label{def: lateral ordering}
By a \und{lateral ordering} $\le $ on a set of real Schur roots $\cS$ we mean a total ordering on $\cS$ satisfying the following for any $\alpha,\beta\in\cS$.
\begin{enumerate}
\item If $hom(\alpha,\beta)\neq0$ then $\alpha\le\beta$, where $hom(\alpha,\beta)=\dim_K\Hom_\Lambda(M_\alpha,M_\beta)$. 
\item If $ext(\alpha,\beta)\neq0$ then $\alpha>\beta$, where $ext(\alpha,\beta)=\dim_K\Ext_\Lambda(M_\alpha,M_\beta)$.
\end{enumerate}
\end{defn}

\begin{rem}  We now state several basic facts and some examples 
 of $\cS$ with lateral ordering and some $\cS$ which do not admit such ordering.
 \begin{enumerate}
\item If $\cS$  has lateral ordering then  for all $\alpha, \beta \in \cS$, either $hom(\alpha,\beta)=0$ or $ext(\alpha,\beta)=0$.
\item The left-to-right order of preprojective roots as they occur in the Auslander-Reiten quiver, together with any ordering on the summands of the middle term of each almost split sequence,  is a lateral ordering. (This is any total ordering on this set of roots so that $\alpha<\beta$ whenever there is a irreducible map $M_\alpha\to M_\beta$.)
\item The set af all regular roots does not admit a lateral ordering.
\item The simple roots can always be laterally ordered by taking $\alpha_i<\alpha_j$ whenever there is an arrow $j\to i$ in the quiver.\item If $\omega$ is a rightmost root in $\cS$ in lateral order then $ext(\beta,\omega)=0$ for all $\beta\in\cS$ and $hom(\omega,\beta')=0$ for all $\beta'\neq\omega\in\cS$.
\item Any subset of a laterally ordered set of roots is laterally ordered with the same ordering.
 \end{enumerate}
 \end{rem}

 \begin{defn}\label{def: vertical ordering}
 A sequence of real Schur roots $\cS=(\beta_1,\cdots,\beta_m)$ is said to be \und{vertically ordered} if the following two conditions are satisfied for each $\beta_k\in \cS$.
 \begin{enumerate}
 \item  Let $\beta'\subset \beta_k$ be any (positive real Schur) subroot of $\beta_k$. Then $\beta'=\beta_j$ for some $j\le k$.
 \item Let $\beta''\!\!\twoheadleftarrow \!\!\beta_k$ be a (positive real Schur) quotient-root of $\beta_k$. Then $\beta''\!=\!\beta_j$ for some $j\le k$. 
 \end{enumerate}
The sequence $\cS$ is \und{weakly vertically ordered} if,  for each $\beta_k\in \cS$, at least one of the above two conditions is satisfied.
   \end{defn}

\begin{rem}A finite set of positive real Schur roots which is closed under subroots and quotient roots can be vertically ordered if the roots are ordered by length, and the roots of the same length are ordered in any way. 
\end{rem}

\begin{defn}\label{def: admissible sequence}
Let $\cS$ be a finite set of positive real Schur roots. 
  \begin{enumerate}
  \item $\cS$ is called \und{admissible} if it has a lateral and a vertical ordering
   \item $\cS$ is called \und{weakly admissible} if it has a lateral ordering and a weakly vertical ordering
  \item an \und{admissible sequence} is an admissible set listed in its vertical ordering
    \item a \und{weakly admissible sequence} is a  weakly admissible set listed in its weakly vertical ordering
   \end{enumerate}
 \end{defn}

\begin{rem}\label{rem: properties of admissible sequences}
 Let $(\beta_1, \dots,\beta_m)$ be an admissible sequence (of positive real Schur roots). The following will be crucial for the induction steps in the proofs.
  \begin{enumerate}
\item If one of the $\beta_i$'s is removed, the sequence 
$(\beta_1, \dots\widehat{\beta_i}\dots,\beta_m)$ is a weakly admissible sequence (not necessarily admissible sequence). 
\item If  the last element $\beta_m$ is removed then the resulting sequence will still be an admissible sequence.
  \end{enumerate} 
\end{rem}

\noi The notion of ``picture groups" for hereditary artin algebras of finite representation type was introduced in \cite{IOTW4} as groups associated to the ``semi-invariant pictures". These groups were defined using all positive roots (the algebras were of finite representation type). We now give a more general definition of ``picture groups", using admissible subsets of positive real Schur roots for all finite dimensional hereditary algebras.

\begin{defn} Let $\cS$ be an admissible set of (positive real Schur) roots.
We will call a subset $\cR\subseteq\cS$ \underline{relatively closed} if $\cR$ is closed under extensions in $\cS$.
\end{defn}

\noi Relatively closed subsets $\cR\subset \cS$ of admissible sets, have ``picture groups'', which we now define.

\begin{defn}\label{def: picture group G(S)}
For any relatively closed subset $\cR\subseteq \cS$ of an admissible set of roots $\cS$, we define the \und{picture group} $G(\cR)$ as follows. There is one generator $x(\beta)$ for each $\beta\in\cR$. There is the following relation for each pair $\beta_i,\beta_j$ of hom-orthogonal roots with $ext(\beta_i,\beta_j)=0$:
\begin{equation}\label{eq: commutator relation}
	x(\beta_i)x(\beta_j)=\prod x(\gamma_k)
\end{equation}
where $\gamma_k$ runs over all roots in $\cR$ which are linear combinations $\gamma_k=a_k\beta_i+b_k\beta_j, a_k,b_k\in\ZZ_{\ge0}$ in increasing order of the ratio $a_k/b_k$ (going from $0/1$ where $\gamma_1=\beta_j$ to $1/0$ where $\gamma_k=\beta_i$). In particular, $x(\alpha), x(\beta)$ commute when $\alpha,\beta$ are both hom-orthogonal and ext-orthogonal. For any $g\in G(\cR)$, we define a \und{positive expression} for $g$ to be any word in the generators $x(\beta)$ (with no $x(\beta)^{-1}$ terms) whose product is $g$.
\end{defn}

\begin{rem}\label{rem: G(R) is a functor}
(a) Note that $G(\cR)$ is independent of the choice of $\cS$. However, the existence of an admissible $\cS$ containing $\cR$ is important. Also, by the well-known Theorem \ref{thm: wide subcategory gen by hom-orthogonal roots}, each $\gamma_k=a_k\beta_i+b_k\beta_j$ has $\beta_i$ as a subroot and $\beta_j$ as a quotient root if $\beta_i,\beta_j$ are hom-orthogonal with $ext(\beta_i,\beta_j)=0$.

(b) Whenever $\cR\subset \cR'$ are relatively closed subsets of an admissible set $\cS$ we get a homomorphism of groups $G(\cR)\to G(\cR')$ since any relation among the generators of $G(\cR)$ is also a relation among the corresponding generators of $G(\cR')$.

(c) Definition \ref{def: picture group G(S)} is a generalization of the notion of ``picture groups'' for hereditary artin algebras of finite representation type as defined in \cite{IOTW4}. Indeed, the \underline{picture group} $G(\Lambda)$ for such an algebra is, by definition, equal to the picture group $G(\Phi^+(\Lambda))$ for the set $\Phi^+(\Lambda)$ of all positive roots of $\Lambda$. These roots are vertically ordered by dimension and laterally ordered by their position in the Auslander-Reiten quiver of $\Lambda$, i.e., there exists a lateral ordering so that, for any irreducible map $M_\alpha\to M_\beta$, $\alpha<\beta$ in lateral order.
\end{rem}

\begin{defn}
Given $\cS$ admissible, we define the \und{Coxeter element} $c_\cS$ of $G(\cS)$ to be the product of the generators $x(\alpha_i)$ for all simple roots $\alpha_i\in\cS$ in lateral order, i.e., so that $\alpha_i<\alpha_j$ whenever there is an arrow $i\leftarrow j$ in the quiver of the algebra. 
\end{defn}

\begin{rem}
As an element of the picture group $G(\cS)$, this product $c_\cS=\prod x(\alpha_i)$ is independent of the choice of the lateral ordering. This is because one can pass from any lateral ordering to any other by transposing consecutive generators $x(\alpha_i),x(\alpha_j)$ when there is no arrow between them in the quiver. But in that case, $x(\alpha_i),x(\alpha_j)$ commute. So, the product remains invariant.
\end{rem}

\begin{eg}\label{eg: A3 example}
Consider the quiver of type $A_3$ with straight orientation: $1\leftarrow 2\leftarrow 3$. The Auslander-Reiten quiver, with modules on the left and corresponding roots on the right is:
\[
\xymatrixrowsep{10pt}\xymatrixcolsep{10pt}
\xymatrix{%begin xy matrix
&& P_3\ar[dr] &&&&&&& \alpha_6\ar[dr] \\
&P_2\ar[dr]\ar[ur]\ar@{--}[rr] && I_2\ar[dr] &&&&&\alpha_4\ar[ur]\ar[dr]\ar@{--}[rr] && \alpha_5\ar[dr]
\\
S_1 \ar[ur]\ar@{--}[rr]&&S_2\ar[ur]\ar@{--}[rr] && S_3 &&&  
	\alpha_1 \ar[ur]\ar@{--}[rr]&&\alpha_2 \ar[ur]\ar@{--}[rr]&&
	\alpha_3
	}%end xy matrix
\]
The set $\cS=(\alpha_1,\alpha_2,\alpha_4,\alpha_3)$ is vertically ordered since the subroot $\alpha_1$ and quotient root $\alpha_2$ of $\alpha_4$ come before it. The set $\cS$ is admissible since it also has a lateral ordering $\alpha_1<\alpha_4<\alpha_2<\alpha_3$. The subsequence $\cS'=(\alpha_1, \alpha_4,\alpha_3)$ is weakly admissible. Also, $\cS'$ is relatively closed in $\cS$ since the missing element is simple.

The picture group $G(\cS)$ has four generators $x(\alpha_1), x(\alpha_2), x(\alpha_3), x(\alpha_4)$ and four relations given by the four pairs of hom-orthogonal roots: 
\begin{enumerate}
\item $x(\alpha_1)x(\alpha_2)=x(\alpha_2)x(\alpha_4)x(\alpha_1)$ from the extension $\alpha_1\cof \alpha_4\onto \alpha_2$.
\item $x(\alpha_2)x(\alpha_3)=x(\alpha_3)x(\alpha_2)$ since the extension $\alpha_5$ of $\alpha_2$ by $\alpha_3$ is not in $\cS$.
\item $x(\alpha_1)x(\alpha_3)=x(\alpha_3)x(\alpha_1)$ since $\alpha_1,\alpha_3$ do not extend each other.
\item $x(\alpha_4)x(\alpha_3)=x(\alpha_3)x(\alpha_4)$ since $\alpha_6\notin\cS$.
\end{enumerate}
Thus $x(\alpha_3)$ is central. (This follows from the fact that $\alpha_3$ is last in both vertical and lateral orderings.) The picture group $G(\cS')$ has generators $x(\alpha_1), x(\alpha_3), x(\alpha_4)$ modulo the relation that $x(\alpha_3)$ is central. The Coxeter element of $G(\cS)$ is
\[
	c_\cS=x(\alpha_1)x(\alpha_2)x(\alpha_3).
\]
\end{eg}

\begin{rem}\label{rem: commutator in lateral order}
If $\beta_i,\beta_j$ are hom-orthogonal and $\beta_i<\beta_j$ in lateral order then 
\[
	x(\beta_i)x(\beta_j)=x(\beta_j)w
\]
where $w$ is a positive expression in letters $\gamma$ where $\beta_i\le\gamma<\beta_j$ in lateral order since $hom(\beta_i,\gamma)\neq0$ and $hom(\gamma,\beta_j)\neq 0$ when $\gamma\neq\beta_i$.
\end{rem}

An important case is when $j=m$, the size of $\cS$. For $\beta_i,\beta_m$ hom-orthogonal we get
\[
	x(\beta_i)x(\beta_m)=x(\beta_m)x(\beta_i).
\]
since the other roots $\gamma_k$ in the formula above would come after $\beta_m$ so do not lie in $\cS$.

We recall that, for all roots $\beta$, there is a unique exceptional module $M_\beta$ with dimension vector $\beta$. The subset $D(\beta)\subseteq \RR^n$ is given by
\[
	D(\beta)=\{x\in\RR^n: \brk{ x,\beta} =0\text{ and } \brk{ x,\beta'} \le0\ \forall \beta'\subset \beta\}
\]
where $\beta'\subset\beta$ means that $M_\beta$ contains an exceptional submodule isomorphic to $M_{\beta'}$. The inner product $\brk{x,\beta}$ is the weighted dot product $\brk{x,\beta}=\sum x_ib_if_i$ where $x_i,b_i$ are the $i$th coordinates of $x,\beta$ and $f_i=\dim_K\End(S_i)$ where $S_i$ is the $i$th simple module. So, $D(\beta)$ does not contain points in $\RR^n$ all of whose coordinates are positive (or negative). For more details see Appendix \ref{ss4}.

Theorem \ref{thm: equivalent definitions of D(b)} in the Appendix proves that $D(\beta)$ has the following equivalent description.
\[
	D(\beta)=\{x\in\RR^n: \brk{ x,\beta} =0\text{ and } \brk{ x,\beta''} \ge0\ \forall \beta''\twoheadleftarrow \beta\}
\]
where $\beta''\twoheadleftarrow \beta$ means that $M_\beta$ has an exceptional quotient module isomorphic to $M_{\beta''}$.

Given $\cS$ a weakly admissible sequence of roots, let $CL(\cS)\subset\RR^n$ denote the union of $D(\beta)$ for all $\beta\in\cS$. Since this set is invariant under scaling in the sense that $\lambda CL(\cS)=CL(\cS)$ for all $\lambda>0$, we usually consider just the intersection $L(\cS):=CL(\cS)\cap S^{n-1}$. The \und{semi-invariant picture} for $G(\cS)$ is defined to be this set $L(\cS)\subset S^{n-1}$ together with the labels of its walls by positive roots and the normal orientation of each wall $D(\beta)$ telling on which side the vector $\beta$ lies. When $n=3$, we draw the stereographic projection of this set onto the plane. (Projecting away from the negative octant. See Figures \ref{Figure99} and \ref{Figure weakly}.)

\begin{defn} Given 
$\cS=(\beta_1,\cdots,\beta_m)$ weakly admissible and $\epsilon=(\epsilon_1,\cdots,\epsilon_m)\in \{0,+,-\}^m$. 
\begin{enumerate}
\item Define $\cU_\epsilon$ to be the convex open set given by
\[
	\cU_\epsilon=\{x\in\RR^n: \brk{ x,\beta_i} >0\text{ if }\epsilon_i=+\text{ and }\brk{ x,\beta_j} <0\text{ if }\epsilon_j=-\}.
\]
\item $\epsilon$ will be called \und{admissible} (with respect to $\cS$) if for all $1\le k\le m$ we have:
\[
	\epsilon_k=0 \iff D(\beta_k)\cap\cU_{\epsilon_1,\cdots,\epsilon_{k-1}}=\emptyset.
\]
\item When $\epsilon$ is admissible the open set $\cU_\epsilon$ will be called an \und{$\cS$-compartment}. See Fig. \ref{Figure99}, \ref{Figure weakly}.
\end{enumerate}
\end{defn}

In Proposition \ref{prop: weakly admissible S gives convex compartments} below we show that, for $\cS$ weakly admissible, each compartment $\cU_\epsilon$ is open and convex and these regions form the components of the complement of $CL(\cS)$ in $\RR^n$.

\begin{figure}[htbp]
\begin{center}
\begin{tikzpicture}[scale=.75]
%\draw[help lines=1,thick] (-5,-5) grid (15,4);
%\foreach \x in {-8,-6,...,8}\draw (\x,0) node{\x};\foreach \y in {-6,-4,...,8}\draw (0,\y) node{\y};
%\draw[very thick, color=blue] (5,-5)--(5,4);
%\draw[very thick, color=blue] (9,-5)--(9,4);
\begin{scope}
\coordinate (A) at (-2,0);
\coordinate (B) at (1,0);
\coordinate (C) at (-.8,0);
\coordinate (D1) at (-.8,-3);
\coordinate (D2) at (3,3);
\coordinate (X1) at (-3,0);
\coordinate (X2) at (-1,0);
\coordinate (X3) at (1,0);
\coordinate (X4) at (3,0);
\coordinate (X5) at (-4,-3);
\draw (-3.5,2.5) node{$D(\alpha_1)$};
\draw (3,2.5) node{$D(\alpha_2)$};
\draw (2.2,1.3) node{$D(\alpha_4)$};
\draw (X1) node{$\cU_{+-0}$};
\draw (X2) node{$\cU_{++0}$};
\draw (X3) node{$\cU_{-++}$};
\draw (X4) node{$\cU_{-+-}$};
\draw (X5) node{$\cU_{--0}$};
\draw[very thick] (A) ellipse[x radius=2cm, y radius=2.5cm];
\draw[very thick] (B) ellipse[x radius=3cm, y radius=2.5cm];
\clip (D1) rectangle (D2);
\draw[very thick] (C) ellipse[x radius=2.9cm, y radius=2cm];
\end{scope}
\begin{scope}[xshift=10cm]
\coordinate (A) at (-2,0);
\coordinate (B) at (1,0);
\coordinate (C) at (-.8,0);
\coordinate (D1) at (-.8,-3);
\coordinate (D2) at (3,3);
\coordinate (E) at (0,-2);
\draw (3.5,-3.7) node{$D(\alpha_3)$};
\begin{scope}[yshift=5mm,xshift=0mm]
\coordinate (X1) at (-3,0);
\coordinate (X2) at (-1,0);
\coordinate (X3) at (1,0);
\coordinate (X4) at (3,0);
\coordinate (X5) at (0,-3.5);
\draw (X1) node{$\cU_{+-0-}$};
\draw (X2) node{$\cU_{++0-}$};
\draw (X3) node{$\cU_{-++-}$};
\draw (X4) node{$\cU_{-+--}$};
\draw (X5) node{$\cU_{--0+}$};
\end{scope}
\begin{scope}[yshift=-6mm,xshift=0mm]
\coordinate (X1) at (-2.5,-.7);
\coordinate (X2) at (-1,0);
\coordinate (X3) at (1,0);
\coordinate (X4) at (2.5,-.7);
\coordinate (X5) at (-4,-3);
\draw (X1) node{$\cU_{+-0+}$};
\draw (X2) node{$\cU_{++0+}$};
\draw (X3) node{$\cU_{-+++}$};
\draw (X4) node{$\cU_{-+-+}$};
\draw (X5) node{$\cU_{--0-}$};
\end{scope}
\draw[very thick] (A) ellipse[x radius=2cm, y radius=2.5cm];
\draw[very thick] (B) ellipse[x radius=3cm, y radius=2.5cm];
\draw[very thick] (E) ellipse[x radius=4cm, y radius=2cm];
\clip (D1) rectangle (D2);
\draw[very thick] (C) ellipse[x radius=2.9cm, y radius=2cm];
\end{scope}
\end{tikzpicture}
\caption{On the left is the semi-invariant picture $L(\cS_0)$ for the admissible subsequences $\cS_0=(\alpha_1,\alpha_2,\alpha_4)$ of $\cS$ from Example \ref{eg: A3 example}. $L(\cS_0)$ is a subset of $S^2\subset\RR^3$. Thus, e.g., $D(\alpha_i)$ are actually coordinate hyperplanes. The $\cS_0$-compartments are the components of the complement of $L(\cS_0)$. For example, $\cU_{++0}=\cU_{++}$ is the region on the positive side of the two hyperplanes $D(\alpha_1),D(\alpha_2)$. $\cU_{-++}$ is the set of point in $\cU_{-+}$ on the positive side of $D(\alpha_4)$.
On the right, the wall $D(\alpha_3)$ cuts all five $\cS_0$-compartments in half giving the semi-invariant picture for $\cS=(\alpha_1,\alpha_2,\alpha_4,\alpha_3)$ with ten compartments.
}
\label{Figure99}
\end{center}
\end{figure}

\begin{defn}\label{def: MGS}
For any weakly admissible $\cS$, we define a \und{maximal $\cS$-green sequence} (of length $s$) to be a sequence of $\cS$-compartments $\cU_{\epsilon(0)},\cdots,\cU_{\epsilon(s)}$ satisfying the following.
\begin{enumerate}
\item Every pair of consecutive compartments $\cU_{\epsilon(i-1)},\cU_{\epsilon(i)}$ is separated by a wall $D(\beta_{k_i})$ so that $\epsilon(i-1)_{k_i}=-$ and $\epsilon(i)_{k_i}=+$ and $\epsilon(i-1)_j=\epsilon(i)_j$ for all $j< k_i$. 
\item $\cU_{\epsilon(0)}$ is the compartment containing vectors all of whose coordinates are negative.
\item $\cU_{\epsilon(s)}$ is the compartment containing vectors all of whose coordinates are positive.
\end{enumerate} We say that $(\cU_\epsilon)$ is an \und{$\cS$-green sequence} if only the first condition is satisfied.
We define an \und{$\cS$-green path} representing the $\cS$-green sequence $(\cU_\epsilon)$ to be a continuous path, $\gamma:\RR\to\RR^n$, so that, for some $t_1<t_2<\cdots<t_s$ we have the following
\begin{enumerate}
\item $\gamma(t)\in \cU_{\epsilon(0)}$ when $t<t_1$
\item $\gamma(t)\in \cU_{\epsilon(s)}$ when $t>t_s$
\item $\gamma(t)\in \cU_{\epsilon(i)}$ for $0<i<s$ whenever $t_i<t<t_{i+1}$
\item For $1\le i\le s$, $\gamma(t)$ goes from the negative side to the positive side of $D(\beta_{k_i})$ for some $\beta_{k_i}\in\cS$ when $t$ crosses the value $t_i$.
\end{enumerate}
The word ``maximal'' may be misleading. (See Figure \ref{Figure weakly}.)
\end{defn}

\begin{figure}[htbp]
\begin{center}
\begin{tikzpicture}[scale=.75]
\begin{scope}%[xshift=10cm]
\coordinate (A) at (-2,0);
\coordinate (B) at (1,0);
\coordinate (C) at (-.8,0);
\coordinate (D1) at (-.8,-3);
\coordinate (D2) at (3,3);
\coordinate (E) at (0,-2);
\draw (-4.7,0.5) node{$D(\alpha_1)$};
\draw (1.5,1.8) node{$D(\alpha_4)$};
\draw (3.5,-3.7) node{$D(\alpha_3)$};
\begin{scope}[yshift=5mm,xshift=0mm]
\coordinate (X1) at (-3,0);
\coordinate (X2) at (-2,0);
\coordinate (X3) at (1,0);
\coordinate (X4) at (-4,2);
\draw (X2) node{$\cU_{+0-}$};
\draw (X3) node{$\cU_{-+-}$};
\draw (X4) node{$\cU_{---}$};
\end{scope}
\begin{scope}[yshift=-6mm,xshift=0mm]
\coordinate (X1) at (-2.5,-.7);
\coordinate (X2) at (-1.8,-.5);
\coordinate (X3) at (1,0);
\coordinate (X4) at (2.5,-.7);
\draw (X2) node{$\cU_{+0+}$};
\draw (X3) node{$\cU_{-++}$};
\draw (X4) node{$\cU_{--+}$};
\end{scope}
\draw[very thick] (A) ellipse[x radius=2cm, y radius=2.5cm];
\draw[very thick] (E) ellipse[x radius=4cm, y radius=2cm];
\draw[thick,color=green!60!black,->] (-2,3)--(-1,-.5);
\clip (D1) rectangle (D2);
\draw[very thick] (C) ellipse[x radius=2.9cm, y radius=2cm];
\draw[ thick,color=green,dashed,->] (1,3)..controls (-.7,2) and (-.6,0)..(-.5,-.9);
\end{scope}
\end{tikzpicture}
\caption{Semi-invariant picture $L(\cS')$ for the weakly admissible sequence $\cS'=(\alpha_1,\alpha_4,\alpha_3)$ from Example \ref{eg: A3 example}. The  solid green arrow indicates an $\cS'$-green path giving the maximal $\cS'$-green sequence $\cU_{---},\cU_{+0-},\cU_{+0+}$. Note that the dashed green arrow indicates another $\cS'$-green path giving the maximal $\cS'$-green sequence 
$\cU_{---}, \cU_{-+-}$, $\cU_{+0-}, \cU_{+0+}$. So, ``maximal'' is a misnomer when $\cS'$ is only weakly admissible. Also, $L(\cS')$ is not a ``planar picture'' for $G(\cS')$ as defined in Section \ref{sec3: Pictures} since $\cS'$ is not admissible.
}
\label{Figure weakly}
\end{center}
\end{figure}

\begin{rem}\label{rem: Coxeter path}
The green arrow in Figure \ref{Figure weakly} is an example of a ``{Coxeter path}'' which is given more generally as follows. Let $\cS$ be an admissible set of roots. Let $\alpha_1,\cdots,\alpha_k$ be the simple roots in $\cS$ in any lateral order. In other words, any arrow between the corresponding vertices in the quiver go from $\alpha_j$ to $\alpha_i$ only when $i<j$. Also, all roots in $\cS$ have support at these vertices by definition of an admissible set. The corresponding \und{Coxeter path} is defined to be the linear path $\gamma: \RR \to \RR^n$
\[
	\gamma(t)=(t,t,\cdots,t)-\sum j\alpha_j.
\]
This path crosses the hyperplanes $D(\alpha_i)$ at time $t=i$ (in the order $\alpha_1,\alpha_2,\cdots$) and it passes from the negative to the positive side of each hyperplane. 

Also, this path is disjoint from all other walls $D(\beta)$ for $\beta\in \cS$ not simple. To see this, suppose $\gamma(t_0)\in D(\beta)$. Then
\[
	\brk{\gamma(t_0),\beta}=0=\sum (t_0-j) f_jb_j.
\]
So, some of the coefficients $t_0-j$ are positive and some are negative with the positive ones coming first, say $t_0-1,t_0-2,\cdots,t_0-p$ positive and the rest negative. In that case $\beta'=b_1\alpha_1+b_2\alpha_2+\cdots+b_p\alpha_p$ is a sum of subroots of $\beta$, but $\brk{\gamma(t_0),\beta'}>0$ which contradicts the assumption that $\gamma(t_0)\in D(\beta)$. So, the Coxeter path does not meet any $D(\beta)$ for $\beta\in \cS$ not simple. Thus, the Coxeter path is an $\cS$-green path. Since the coordinates of $\gamma(t)$ are all negative for $t<<0$ and all positive for $t>>0$, this green path gives a maximal $\cS$-green sequence which we call the \ul{Coxeter green sequence}. The product of the group labels on the walls crossed by this green sequence form the Coxeter element $c_\cS=x(\alpha_1)x(\alpha_2)\cdots x(\alpha_k)\in G(\cS)$.
\end{rem}

\subsection{Statement of the main results}\label{ss1.2: statement of results}

The main property of $\cU_\epsilon\subset \RR^n$ is that it is convex and nonempty when $\cS$ is weakly admissible and $\epsilon$ is admissible with respect to $\cS$ (Proposition \ref{prop: weakly admissible S gives convex compartments}). Furthermore, when $\cS$ is admissible, the complement of the union of these regions forms a ``picture'' for the picture group $G(\cS)$. The precise statement is as follows.

\begin{thm}\label{thm: L(S) is picture for picture group G(S)}
When $\cS$ is admissible, each $\cS$-compartment $\cU_\epsilon$ can be labelled with an element of the picture group $g(\epsilon)\in G(\cS)$ so that, if $\cU_\epsilon$ and $\cU_{\epsilon'}$ are separated by a wall $D(\beta)$, $\beta\in \cS$, with $\cU_{\epsilon'}$ on the positive side of $D(\beta)$, then
\begin{equation}\label{eq: labels for wall crossing}
	g(\epsilon)x(\beta)=g(\epsilon').
\end{equation}
\end{thm}

Note that, given any system of compartment labels $g(\epsilon)$ satisfying \eqref{eq: labels for wall crossing}, left multiplication of all labels by a fixed element of $G(\cS)$ will preserve the condition. Therefore, we may, without loss of generality, assume that $g(\epsilon)=1$ on the negative $\cS$-compartment $\cU_\epsilon$ where all $\epsilon_i$ are negative or zero. Theorem \ref{thm: L(S) is picture for picture group G(S)} follows from the following lemma.

\begin{lem}\label{lem: all max green sequences start at neg L and go to pos L}
For $\cS$ weakly admissible, every $\cS$-compartment $\cU_\epsilon$ lies in a maximal $\cS$-green sequence given by an $\cS$-green path.
\end{lem}

\begin{proof}
Given any $\cS$-compartment $\cU_\epsilon$, choose a general point $v\in \cU_\epsilon$ and consider the straight line $f(t)=v+(t,t,\cdots,t)$, $t\in \RR$. This line passes though walls $D(\beta)$ only in the positive direction since
\[
	\brk{ (1,1,\cdots,1),\beta} >0
\]
for all positive roots $\beta$. Thus $f(t)$ is an $\cS$-green path giving an $\cS$-green sequence. For $t>>0$, the coordinates of $f(t)$ are all positive. For $t<<0$, they are all negative. Therefore $f(t)$ gives a maximal $\cS$-green sequence passing through the $\cS$-compartment $\cU_\epsilon$ at $t=0$.
\end{proof}

\begin{proof}[Proof of Theorem \ref{thm: L(S) is picture for picture group G(S)}]
Given an $\cS$-compartment $\cU_\epsilon$ choose an $\cS$-green path through $\cU_\epsilon$ as in the lemma above. Let $g(\epsilon)$ be the product of labels $x(\beta_i)$ for the walls crossed by this path on the way to $\cU_\epsilon$. Condition \eqref{eq: labels for wall crossing} will be satisfied. We only need to show that $g(\epsilon)$ is well defined. To do this suppose we have two $\cS$-green paths $\gamma,\gamma'$ from the negative compartment to $\cU_\epsilon$. Since $\RR^n$ is contractible these paths are homotopic. The homotopy gives a mapping of $h:[0,1]^2\to \RR^n$. Make this a smooth mapping transverse to $CL(\cS)$.

Since each wall $D(\beta)$ is contained in the hyperplane $H(\beta)$, $CL(\cS)$ is contained in the union of these hyperplanes. The intersection of two hyperplanes has codimension 2. Since $\cS$ is finite, there are only finitely many such subspaces. We ignore the other intersections which have higher codimension. By transversality, the homotopy $h$ will only meet these codimension 2 subspaces at a finite number of points. Let $x_0\in \RR^n$ be one these points and let $\cB$ be the set of all $\beta\in\cS$ so that $x_0\in D(\beta)$. Let $\cA$ be the set of minimal elements of $\cB$, i.e., the set of all $\alpha\in\cB$ so that no subroot of $\alpha$ lies in $\cB$.

Then $\cA$ has at most two elements since, otherwise, by Proposition \ref{prop: exceptional sequences are lin indep}, the intersection of $D(\alpha)$ for $\alpha\in\cA$ has codimension $\ge3$. If $\cA$ has only one element then $\cA=\cB$. In that case, the wall crossing sequence is unchanged when the path is deformed past $x_0$. The remaining case is when $\cA$ has two elements: $\cA=\{\alpha_1,\alpha_2\}$. By Corollary \ref{cor: W(x0) cap S}, the other elements of $\cB$ are positive linear combinations $\beta=x\alpha_1+y\alpha_2$ and $D(\beta)$ lies on the negative side of $D(\alpha_1)$ and the positive side of $D(\alpha_2)$ since $\alpha_1\subset\beta$ implies $\brk{v,\alpha_1}\le 0$ for all $v\in D(\beta)$. This means that, on one side of $x_0$, the $\cS$-green path goes through $D(\alpha_1)$ followed by $D(\alpha_2)$ and on the other side, it goes through $D(\alpha_2)$, then, being on the positive side of $D(\alpha_2)$ and on the negative side of $D(\alpha_1)$ it goes through $D(\beta)$ for $\beta\in\cB$. (See Figure \ref{Fig: wall crossing homotopy}.) This sequence of wall crossings gives the same element of the picture group. So, the group label $g(\epsilon)$ is independent of the path. This proves the theorem.
\end{proof}

\begin{figure}[htbp]
\begin{center}
\begin{tikzpicture}[scale=.85]
\coordinate (A) at (0,0);
\coordinate (B) at (3.5,0);
\coordinate (Bp) at (4,3);
\coordinate (X) at (1,3);
\coordinate (A1) at (3,-1.5);
\coordinate (A1p) at (3.5,1.75);
\coordinate (A2) at (3,1.5);
\coordinate (A2p) at (3.5,4.25);
\coordinate (DA2) at (2.5,3.25);
\coordinate (A1m) at (-3,1.5);
\coordinate (A1mp) at (-1.5,4.25);
\coordinate (A2m) at (-3,-1.5);
\coordinate (A2mp) at (-2,1);
\draw[very thick] (A1)--(A1p)--(A1mp);
\draw[fill,color=white] (X)--(A)--(B)--(Bp)--(A2p)--(X);
\draw[very thick] (B)--(Bp)--(X);
\draw[fill,color=white] (X)--(A)--(A2)--(A2p)--(X);
\draw[very thick] (A2m)--(A2mp);
\draw[very thick] (A1)--(A1m)--(A1mp) (X)--(A2p)--(A2)--(A2m) (A)--(B);
\draw[thick] (A)--(X) (-1.8,-.3) node{$D(\alpha_2)$};
\draw[thick] (-2,1.7) node{$D(\alpha_1)$};
\draw[thick] (1.3,1.3) node{$D(\alpha_2)$};
\draw[thick] (2.4,-.6) node{$D(\alpha_1)$};
\draw[thick] (2.7,.4) node{$D(\beta)$};
\draw[fill] (.5,1.5) circle[radius=1mm] node[left]{$x_0$};
\draw[ thick,color=red] (1,3.9) --(-.5,1.6);
\draw[ thick,color=red] (1,3.9) --(0,1.6);
\draw[ thick,color=red] (1,3.9) --(0.5,1.6);
\draw[ thick,color=red] (1,3.9) --(1,1.6);
\draw[ thick,color=red] (1,3.9) --(1.5,1.6);
\draw[ thick,color=red] (1.5,2.5) node{$h$};
\draw[ thick,color=red] (-.5,-.25)--(0,-1.4) --(0,0);
\draw[ thick,color=red] (.5,-.25)--(0,-1.4) --(1,-.5);
\draw[very thick,color=green,->] (1,4)..controls (1,3.5) and (2,3)..(2,1.5); 
\draw[very thick,color=green,->] (2,1)--(1.9,.5); 
\draw[very thick,color=green,->] (1.8,0)--(1.6,-.5); 
\draw[very thick,color=green,->] (1.5,-.75)..controls (1.3,-1) and (0,-1.2)..(0,-1.5); 
\draw[very thick,color=green,->] (1,4)..controls (1,3.5) and (-1,2)..(-1,1.5); 
\draw[very thick,color=green,->] (-1,.5)--(-1,0); 
\draw[color=green!70!black] (-.5,2.6) node{$\gamma$} (2.1,2.6) node{$\gamma'$}; 
\draw[very thick,color=green,->]  (-.9,-.45)..controls (-.7,-1)and (0,-1.2)..(0,-1.5); 
\end{tikzpicture}
\caption{A typical intersection of two walls $D(\alpha_1)$ and $D(\alpha_2)$ producing walls $D(\beta_i)$. In this drawing there is only $\beta=\alpha_1+\alpha_2$. The green path $\gamma$ crosses $D(\alpha_1),D(\alpha_2)$ and $\gamma'$ crosses $D(\alpha_2),D(\beta),D(\alpha_1)$. The homotopy $h:\gamma\simeq\gamma'$ passes through $x_0$.}
\label{Fig: wall crossing homotopy}
\end{center}
\end{figure}

\begin{lem}\label{lem: green sequences give cS}
Take any maximal $\cS$-green sequence for $\cS$ admissible and consider the sequence of walls $D(\beta_{k_1})$, $\cdots,D(\beta_{k_s})$ which are crossed by the sequence. Then the product of the corresponding generators  $x(\beta_{k_i})\in G(\cS)$ is equal to the Coxeter element $c_\cS\in G(\cS)$:
\[
	x(\beta_{k_1})\cdots x(\beta_{k_s})=\prod x(\alpha_i).
\]
\end{lem}

\begin{proof} Use Theorem \ref{thm: L(S) is picture for picture group G(S)} with the group element $g(\epsilon)=1$ on the negative $\cS$-compartment. Then the group label on the positive $\cS$-compartment is equal to the product of the positive expression associated to any maximal $\cS$-green sequence. By Remark \ref{rem: Coxeter path}, any Coxeter path gives the Coxeter element. Therefore, every maximal $\cS$-green sequence gives a positive expression for the Coxeter element of $G(\cS)$.
\end{proof}

Lemma \ref{lem: green sequences give cS} can be rephrased as follows. Any maximal $\cS$-green sequence gives a positive expression for $c_\cS$ by reading the labels of the walls which are crossed by the sequence. The main theorem of this paper is the following theorem and its corollary. 

\begin{customthm}{A}\label{thm A}
Suppose that $\cS$ is an admissible set of roots. Then, the operation described above gives a bijection:
\[
	\left\{
	\text{maximal $\cS$-green sequences}
	\right\}\xrightarrow{\cong}
	\left\{
	\text{positive expressions for $c_\cS$ in $G(\cS)$}
	\right\}
\]
\end{customthm}

It is clear that distinct maximal $\cS$-green sequences give distinct positive expressions. Therefore, it suffices to show that every positive expression for $c_\cS$ can be realized as a maximal $\cS$-green sequence.

Recall from Remark \ref{rem: G(R) is a functor}(c) that, for $\Lambda$ a hereditary artin algebra of finite representation type, the set $\Phi^+(\Lambda)$ of positive roots of $\Lambda$ forms an admissible set and that the picture group of $\Lambda$ is equal to the picture group of $\cS=\Phi^+(\Lambda)$. This leads to the following corollary.

\begin{customcor}{B}\label{cor B}
For $\Lambda$ any hereditary artin algebra of finite representation type, there is a bijection between the set of maximal green sequences for $\Lambda$ and the set of positive expressions for the Coxeter element $c_Q=x(\alpha_1)\cdots x(\alpha_n)$ in $ G(\Lambda)=G(\Phi^+(\Lambda))$.\qed
\end{customcor}

\subsection{Outline of proof of Theorem A}\label{ss1.3: outline of proof of main theorem}

The proof is by induction on $m$, the size of the finite set $\cS$. If $m=1$, the root $\beta_1$ must be simple. So, the group $G(\cS)$ is infinite cyclic with generator $x(\beta_1)$ which is equal to $c_\cS$. There are two compartments $\cU_1,\cU_{-1}$ separated by the single hyperplane $D(\beta_1)=H(\beta_1)$. And $\cU_-,\cU_+$ is the unique $\cS$-green sequence. The associated positive expression is $x(\beta_1)$ which is the unique positive expression for $c_\cS$. So, the result holds for $m=1$. Thus, we may assume that $m\ge 2$ and the theorem holds for the admissible sequence of roots $\cS_0=(\beta_1,\cdots,\beta_{m-1})$.

\begin{rem}\label{rem: bm commutes with b iff hom-orthog}
One key property of the last element $\beta_m$ in an admissible sequence $\cS$ is that, for $\beta\neq \beta_m$ in $\cS$, $x(\beta)$ commutes with $x(\beta_m)$ if and only if $\beta$ is hom-orthogonal to $\beta_m$. The reason is that there is a formula for the commutator of two roots if and only if they are hom-orthogonal and, in that case, the commutator is a product of extensions of these roots. But any extension comes afterwards in admissible (vertical) order, so any extension of $\beta_m$ will not be in the set $\cS$. 
\end{rem}

\begin{lem}\label{lem: pi: G(S) to G(S0)}
There is a surjective group homomorphism
\[
	\pi:G(\cS)\onto G(\cS_0)
\]
given by sending each $x(\beta_i)\in G(\cS)$ to the generator in $G(\cS_0)$ with the same name when $i<m$ and sending $x(\beta_m)$ to 1. 
\end{lem}

\begin{proof}
By Remark \ref{rem: bm commutes with b iff hom-orthog} there are only two kinds of relations in $G(\cS)$ involving $x(\beta_m)$:
\begin{enumerate}
\item Commutation relations: $[x(\beta_m),x(\beta_j)]=1$ when $\beta_m,\beta_j$ are hom-orthogonal.
\item Relations in which $x(\beta_m)$ occurs only once:
\[
	x(\beta_i)x(\beta_j)=x(\beta_j)\cdots x(\beta_m)\cdots x(\beta_i).
\]\end{enumerate}
In both cases, when $x(\beta_m)$ is deleted, the relation in $G(\cS)$ reduces to a relation in $G(\cS_0)$ (or to the trivial relation $x(\beta_j)=x(\beta_j)$ in Case 1). Thus, $G(\cS_0)$ is given by $G(\cS)$ modulo the relation $x(\beta_m)=1$.
\end{proof}

Suppose $m\ge 2$ and $\beta_m$ is simple, say $\beta_m=\alpha_k$ the $k$th simple root. Then, since $\cS$ is admissible, all previous roots $\beta_j,j<m$ have support disjoint from $\alpha_k$. Then $x(\beta_m)$ is central and $G(\cS)$ is the product $G(\cS)=G(\cS_0)\times \ZZ$ where the $\ZZ$ factor is generated by $x(\beta_m)$. Thus a positive expression for $c_\cS$ is given by any positive expression for $x_{\cS_0}$ with the letter $x(\beta_m)$ inserted at any point. 

Each $\cS_0$-compartment $\cU_\epsilon$ is the inverse image in $\RR^n$ of a compartment for $\cS_0$ in $\RR^{n-1}$. Thus, any $\cS$-maximal green sequence will pass through these walls giving a maximal $\cS_0$-green sequence and must, at some point, pass from the negative side of the hyperplane $D(\beta_m)$ to its positive side. (See Figure \ref{Figure99} for an example.) By induction on $m$, this $\cS_0$-maximal green sequence is any positive word for $c_{\cS_0}$ and the crossing of $D(\beta_m)$ inserts $x(\beta_m)$ at any point. This describes all words for $c_\cS$. So, the theorem holds in this case.

Now suppose $\beta_m$ is not simple. Then $\cS,\cS_0$ have the same set of simple roots. So, $\pi(c_\cS)=c_{\cS_0}$. Suppose that $w$ is a positive expression for $c_\cS$ in $G(\cS)$. Let $\pi(w)=w_0$ be the positive expression for $c_{\cS_0}$ in $G(\cS_0)$ given by deleting every instance of the generator $x(\beta_m)$ from $w$. By induction on $m$, there exists a unique maximal $\cS_0$-green sequence $\cU_{\epsilon(0)},\cdots,\cU_{\epsilon(s)}$ which realizes the positive expression $w_0$. These fall into two classes.\vs2

\underline{Class 1.} Each $\cS_0$-compartment $\cU_{\epsilon(i)}$ in the maximal $\cS_0$-green sequence is disjoint from $D(\beta_m)$.

For maximal $\cS_0$-green sequences in this class, each $\cU_{\epsilon(i)}=\cU_{\epsilon'(i)}$ where $\epsilon'(i)=(\epsilon_1,\cdots,\epsilon_{m})$ with $\epsilon_m=0$ and $\epsilon(i)=(\epsilon_1,\cdots,\epsilon_{m-1})$. Therefore, the maximal $\cS_0$-green sequence $\cU_{\epsilon(i)}$ is also a maximal $\cS$-green sequence and $w_0=w$ by the following lemma proved in subsection \ref{ss3.5: proof of lemmas C,E}. So, the positive expression $w$ is realized by a maximal $\cS$-green sequence.

\begin{customlem}{C}\label{lem C: beta m occurs a fixed number of times}
Let $w,w'$ be two positive expressions for the same element of the group $G(\cS)$. Suppose $\pi(w)=\pi(w')$, i.e., the two expressions are identical modulo the generator $x(\beta_m)$. Then $x(\beta_m)$ occurs the same number of times in $w,w'$. In particular, $x(\beta_m)\neq1$ in $G(\cS)$.
\end{customlem}

In the case at hand, $w'=w_0$ does not contain the letter $x(\beta_m)$. So, neither does $w$ and we must have $w=w_0$ as claimed. So, by Lemma \ref{lem C: beta m occurs a fixed number of times}, the theorem hold when $w_0=\pi(w)$ corresponds to a maximal $\cS_0$-green sequence of Class 1.

\underline{Class 2.} At least one $\cS_0$-compartment in the $\cS_0$-green sequence meets $D(\beta_m)$.

For green sequences in this class, the $\cS_0$-compartments which intersect $D(\beta_m)$ are consecutive:

\begin{customlem}{D}\label{lem D: properties of compartments used in proof}
Let $\cU_{\epsilon(0)},\cdots,\cU_{\epsilon(s)}$ be a maximal $\cS_0$-green sequence. Then
\begin{enumerate}
\item The $\cS_0$-compartments $\cU_{\epsilon(i)}$ which meet $D(\beta_m)$ are consecutive, say $\cU_{\epsilon(p)},\cdots,\cU_{\epsilon(q)}$.
\item Let $D(\beta_{k_i})$ be the wall between $\cU_{\epsilon(i-1)}$ and $\cU_{\epsilon(i)}$ so that $w_0=x(\beta_{k_1})\cdots x(\beta_{k_s})$. Then $\beta_m$ is hom-orthogonal to $\beta_{k_i}$ for $p<i\le q$ but not hom-orthogonal to $\beta_{k_p}$, $ \beta_{k_{q+1}}$.
\item For $p<i\le q$ and $\delta\in\{+,-\}$, $D(\beta_{k_i})$ is also the wall separating $\cU_{\epsilon(i-1),\delta}$ and $\cU_{\epsilon(i),\delta}$.
\end{enumerate}
\end{customlem}

Lemma \ref{lem D: properties of compartments used in proof} tells us: (1) The $\cS_0$-compartments $\cU_{\epsilon(r)}$ for $p\le r\le q$ are divided into two $\cS$-compartments by the wall $D(\beta_m)$. (3) The wall separating consecutive $\cS_0$-compartments $\cU_{\epsilon(r)}, \cU_{\epsilon(r+1)}$ for $p\le r<q$ also separate the pairs of $\cS$-compartments $\cU_{\epsilon(r),-},\cU_{\epsilon(r+1),-}$ and $\cU_{\epsilon(r),+},\cU_{\epsilon(r+1),+}$. (See Figure \ref{Fig: Inescapable regions}.)

So, we can refine the maximal $\cS_0$-green sequence to a maximal $\cS$-green sequence, by staying on the negative side of $D(\beta_m)$ until we reach the $\cS$-compartment $\cU_{\epsilon(r),-}$ for some $p\le r\le q$, then cross through $D(\beta_m)$ into $\cU_{\epsilon(r),+}$ and continue in the given $\cS_0$-compartments but on the positive side of $D(\beta_m)$. This gives the maximal $\cS$-green sequence
\[
	\cU_{\epsilon(0),0},\cdots,\cU_{\epsilon(p-1),0},\cU_{\epsilon(p),-},\cdots,\cU_{\epsilon(r),-},\cU_{\epsilon(r),+}, \cdots, \cU_{\epsilon(q),+}, \cU_{\epsilon(q+1),0}, \cdots, \cU_{\epsilon(s),0} 
\]
of length $s+1$ giving the positive expression
\[
	w_r=x(\beta_{k_1})\cdots x(\beta_{k_p}) \cdots x(\beta_{k_r})x(\beta_m)x(\beta_{k_{r+1}})\cdots x(\beta_{k_q})\cdots x(\beta_{k_s}).
\]
By the defining relations in the group $G(\cS)$, the generators $x(\beta)$ and $x(\beta_m)$ commute if $\beta$ is hom-orthogonal to $\beta_m$. By (3) in the lemma this implies that $w_r$ is a positive expression for $c_\cS$ if $p\le r\le q$. We have just shown that each such $w_r$ is realizable by a maximal $\cS$-green sequence. So, it remains to show that the positive expression $w$ that we started with is equal to one of these $w_r$. 

By Lemma \ref{lem C: beta m occurs a fixed number of times}, $x(\beta_m)$ occurs exactly once in the expression $w$. We need to show that, if the generator $x(\beta_m)$ occurs in the ``wrong place'' then $w$ is not a positive expression for $c_\cS$, in other words, the product of the elements of $w$ is not equal to $c_\cS$. This follows from the following lemma proved in subsection \ref{ss3.5: proof of lemmas C,E}.

\begin{customlem}{E}\label{lem E: beta m only commutes with hom orthogonal roots}
Let
\[
	\cR(\beta_m)=\{\beta_i\in\cS_0: hom(\beta_i,\beta_m)=0=hom(\beta_m,\beta_i)\}.
\]
Let $\beta_{j_1},\cdots,\beta_{j_s}$ be elements of $\cS_0$ which do not all lie in $\cR(\beta_m)$. Then $x(\beta_m)$, $\prod x(\beta_{j_i})$ do not commute in the group $G(\cS)$.
\end{customlem}

By part (2) of Lemma \ref{lem D: properties of compartments used in proof}, $\beta_{k_r}\in\cR(\beta_m)$ if $p< r\le q$ and $\beta_{k_p},\beta_{k_{q+1}}\notin\cR(\beta_m)$. So, this lemma implies that $w_r$ is a positive expression for $c_\cS$ if and only if $p\le r\le q$. So, we must have $w=w_r$ for one such $r$ and $w$ is realizable. This concludes the outline of the proof of the main theorem. It remains only to prove the three lemmas \ref{lem C: beta m occurs a fixed number of times}, \ref{lem D: properties of compartments used in proof}, \ref{lem E: beta m only commutes with hom orthogonal roots} invoked in the proof.

\begin{rem}\label{rem: only finitely many S-MGS}
The number of times a maximal $\cS$-green sequences crosses $D(\beta_m)$ is at most one. In Class 1, the number is zero by definition. In Class 2, the number is one as explained in detail above assuming Lemmas \ref{lem C: beta m occurs a fixed number of times}, \ref{lem D: properties of compartments used in proof}, \ref{lem E: beta m only commutes with hom orthogonal roots}. It follows that any maximal $\cS$-green sequence crosses any wall $D(\beta_k)$ for $\beta_k\in\cS$ at most once since $\beta_k$ will be the last element of the subsequence $\cR=(\beta_1,\cdots,\beta_k)$ of $\cS$ (which is admissible by Remark \ref{rem: properties of admissible sequences}) and any maximal $\cS$-green sequence gives a maximal $\cR$-green sequence which crosses $D(\beta_k)$ the same number of times. It follows that any maximal $\cS$-green sequence has length at most equal to the size of $\cS$. Since $\cS$ is finite, it follows that there are only finitely maximal $\cS$-green sequences.
\end{rem}

\section{Properties of compartments $\cU_\epsilon$}\label{sec2: Compartments}

We derive the basic properties of the compartments $\cU_\epsilon$ and prove Lemma \ref{lem D: properties of compartments used in proof}. The basic property is the following.

\begin{prop}\label{prop: weakly admissible S gives convex compartments}
For all weakly admissible $\cS$ and all admissible $\epsilon$ the $\cS$-compartment $\cU_\epsilon$ is convex and nonempty. When $\epsilon_m\neq0$, or equivalently, when $D(\beta_m)\cap\cU_{\epsilon_1,\cdots,\epsilon_{m-1}}$ is nonempty, the boundary of $D(\beta_m)$ does not meet $\cU_{\epsilon_1,\cdots,\epsilon_{m-1}}$. Equivalently,
\[
D(\beta_m)\cap\cU_{\epsilon_1,\cdots,\epsilon_{m-1}}=H(\beta_m)\cap \cU_{\epsilon_1,\cdots,\epsilon_{m-1}}.
\]
Consequently, the $\cS$-compartments form the components of the complement of $CL(\cS)$ in $\RR^n$.
\end{prop}

\begin{proof}
When $m=1$, $\beta_1$ is simple and $D(\beta_1)=H(\beta_1)$ is a hyperplane whose complement has two convex components $\cU_+,\cU_-$. So, the proposition holds for $m=1$. Now, suppose $m\ge2$ and all statements hold for $m-1$. Let $\cS_0=\cS\backslash \beta_m$. This a weakly admissible sequence of roots. So, the $\cS_0$-components $\cU_\epsilon$ are convex and open and their union is the complement of $CL(\cS_0)$.

Since $\cS$ is weakly admissible, it either contains all subroots of $\beta_m$ or it contains all quotient roots of $\beta_m$. By symmetry we assume the first condition. Let $\epsilon=(\epsilon_1,\cdots,\epsilon_{m-1})$ be admissible of length $m-1$. If $(\epsilon, 0)$ is admissible for $\cS$ then $\cU_{\epsilon,0}=\cU_\epsilon$. Otherwise, $\cU_\epsilon$ meets $D(\beta_m)$. In that case $\cU_\epsilon\cap \partial D(\beta_m)$ must be empty since any element $x_0\in \partial D(\beta_m)$ must be an element of $D(\beta')$ for some proper subroot $\beta'\subsetneq \beta_m$. By assumption, $\beta'\in\cS_0$. So, $x_0\in CL(\cS_0)$. This gives a contradiction since $\cU_\epsilon$ is disjoint from $CL(\cS_0)$ by induction on $m$. Therefore, $\cU_\epsilon$ is divided into two convex open sets $\cU_{\epsilon,+}$ and $\cU_{\epsilon,-}$ separated by $D(\beta_m)$. So the $\cS$-compartments fill up the complement of $CL(\cS_0)\cup D(\beta_m)=CL(\cS)$.
\end{proof}

\subsection{Inescapable regions}\label{ss2.1: inescapable regions}

For $\cS_0$ weakly admissible, let $\cV$ be the closure of the union of some set of $\cS_0$-compartments $\cU_\epsilon$. Then $\cV$ has internal and external walls. The \und{internal walls} of $\cV$ are the ones between two of the compartments $\cU_\epsilon,\cU_{\epsilon'}$ in $\cV$. $\cV$ has points on both sides of the internal walls. The \und{external walls} of $\cV$ are the ones which separate $\cV$ from its complement. The region $\cV$ will be called \und{inescapable} if it is on the positive side of all of its external walls. I.e., they are all red on the inside. Once an $\cS_0$-green sequence enters such a region, it can never leave. Since $\cV$ is closed, it contains all of its internal and external walls. We also consider open regions $\cW$ which are inescapable regions minus their external walls. Then $\cW$ is the complement of the closure of the union of all compartments not in $\cW$.

Given an admissible sequence $\cS$ with last object $\beta_m$ which we assume to be nonsimple, let $\cS_0=(\beta_1,\cdots,\beta_{m-1})$. Recall that this is also admissible. We will construct two inescapable regions $\cW(\beta_m),\cV(\beta_m)$ where the first is open and the second is closed. All maximal $\cS_0$-green sequences start outside both regions, end inside both regions and fall into two classes: those that enter $\cW(\beta_m)$ before they enter $\cV(\beta_m)$ and those that enter $\cV(\beta_m)$ before they enter $\cW(\beta_m)$. And these coincide with the two classes of maximal $\cS_0$-green sequences discussed in the outline of the main theorem (Corollary \ref{cor: two classes of green paths} below).

The first inescapable region is the open set
\[
	\cW(\beta_m):=\{x\in\RR^n: \brk{x,\alpha}>0\text{ for some }\alpha\subsetneq\beta_m\}.
\]
For example, on the left side of Figure \ref{Figure99}, $m=3$ and $\cW(\beta_3)$ is the interior of $D(\alpha_1)$.
\begin{prop} The complement of $\cW(\beta_m)$ in $\RR^n$ is closed and convex. Furthermore:
\begin{equation}\label{eq: V cap H is the complement of D}
	 \cW(\beta_m)\cap H(\beta_m)=H(\beta_m)-D(\beta_m).
\end{equation}
\end{prop}

\begin{proof}
The complement of $\cW(\beta_m)$ is
\[
	\RR^n\backslash\cW(\beta_m)=\{x\in\RR^n: \brk{x,\alpha}\le0\text{ for all }\alpha\subsetneq\beta_m\}
\]
which is closed and convex since it is given by closed convex conditions $\brk{x,\alpha}\le0$.

For the second statement, suppose that $v\in H(\beta_m)$. Then $\brk{v,\beta_m}=0$. By the stability conditions which we are using to define $D(\beta_m)$, $v\in D(\beta_m)$ if and only if $\brk{v,\alpha}\le0$ for all $\alpha\subset\beta_m$, in other words,
\[
	D(\beta_m)=H(\beta_m)\cap(\RR^n\backslash\cW(\beta_m))
\]
which is equivalent to \eqref{eq: V cap H is the complement of D}.
\end{proof}

\begin{prop}\label{prop: walls to enter W}
The region $\cW(\beta_m)$ is inescapable. I.e., all external walls are red. Furthermore, each external walls of $\cW(\beta_m)$ has the form $D(\alpha)$ for some $\alpha\subsetneq \beta_m$. Consequently, every $\cS_0$-compartment is contained either in $\cW(\beta_m)$ or in its complement.
\end{prop}

\begin{proof}
Take any external wall $D(\alpha)$ of $\cW(\beta_m)$. Let $v_t$ be a continuous path which goes through that wall from inside to outside. In other words, $v_t\in\cW(\beta_m)$ for $t<0$ and $v_t\notin \cW(\beta_m)$ for $t\ge0$. By definition of $\cW(\beta_m)$ this means that there is some $\alpha'\subsetneq \beta_m$ so that $\brk{v_t,\beta}$ changes sign from positive to nonpositive at $t$ goes from negative to nonnegative.

By choosing $v_t$ in general position, $v_0$ will not lie in $H(\alpha')$ for any $\alpha'\neq\alpha$. So, we must have $\alpha\subsetneq\beta_m$. And $\brk{v_t,\alpha}>0$ for $t<0$ and $\brk{v_t,\alpha}<0$ for $t>0$. Therefore, $\cW(\beta_m)$ is on the positive (red) side of the external wall $D(\alpha)$. So, $\cW(\beta_m)$ is inescapable.

Since each part of the boundary lies in $D(\alpha)$ for some $\alpha\in \cS_0$, the boundary of $\cW(\beta_m)$ is contained in the union of the boundaries of the $\cS_0$-compartments. So, all such compartments are either entirely insider or entirely outside $\cW(\beta_m)$.
\end{proof}

The second inescapable region is the closed set
\[
	\cV(\beta_m)=\{y\in\RR^n: \brk{y,\gamma}\ge 0\text{ for all quotient roots $\gamma$ of $\beta_m$}\}.
\]
For example, on the left side of Figure \ref{Figure99}, $m=3$ and $\cV(\beta_3)$ is the closure of the interior of $D(\alpha_2)$. In Figure \ref{Fig: Inescapable regions}, $\cV(\beta_m)$ is the region enclosed by the large oval. By arguments analogous to the ones above, we get the following.
\begin{prop}\label{prop: walls to enter V}
$\cV(\beta_m)$ is a closed convex inescapable region whose external walls all have the form $D(\gamma)$ where $\gamma$ is a quotient root of $\beta_m$. So, every $\cS_0$-compartment is contained in $\cV(\beta_m)$ or its complement. Furthermore, 
\[
	\cV(\beta_m)\cap H(\beta_m) = D(\beta_m).
\]
\end{prop}

\begin{figure}[htbp]
\begin{center}
\begin{tikzpicture}[scale=.8]
\coordinate (V) at (4.3,-1);
\coordinate (V0) at (1,2.3);
\coordinate (W) at (-6.8,.8);
\coordinate (H) at (6.8,0);
\coordinate (A) at (-2.8,2);
\coordinate (B) at (2.8,-2);
\coordinate (C) at (4,1.5);
\coordinate (G) at (-1.7,2.3);
\coordinate (Upm) at (-2,.5);
\coordinate (Upp) at (-2,-.5);
\coordinate (Urm) at (0,.5);
\coordinate (Urp) at (0,-.5);
\coordinate (Uqm) at (2,.5);
\coordinate (Uqp) at (2.,-.5);
\draw (Upm) node{$\cU_{\epsilon(p)-}$};
\draw (Upp) node{$\cU_{\epsilon(p)+}$};
\draw (Urm) node{$\cU_{\epsilon(r)-}$};
\draw (Urp) node{$\cU_{\epsilon(r)+}$};
\draw (Uqm) node{$\cU_{\epsilon(q)-}$};
\draw (Uqp) node{$\cU_{\epsilon(q)+}$};
\draw[very thick] (0,-1.5)..controls (4,-1.5) and (5.8,1.5)..(7.5,1.5);
\draw[very thick] (0,-1.5)..controls (-4,-1.5) and (-5.8,1.5)..(-7.5,1.5);
\draw[thick,color=red] (-7.5,0)--(7.5,0);
\draw[very thick] (0,-1) ellipse [x radius=5cm, y radius= 2.5cm];
\draw[color=red] (H) node[above]{$H(\beta_m)$};
\draw (V) node{$\cV(\beta_m)$};
\draw (W) node{$\cW(\beta_m)$};
\draw (V0) node[right]{$\cV_0$};
\draw[thick,->] (V0)--(0.2,1.1);
\draw[very thick, color=green,->] (-5.5,2)..controls (-6,0) and (-5,-1.5)..(-3.5,-1.7); 
\draw[color=green] (-5.5,2) node[right]{$\gamma_1$};
\draw[very thick, color=green,->] (G)..controls (-2,-1) and (1.5,1)..(1.5,-2); 
\draw[color=green] (G) node[right]{$\gamma_2$};
\draw[color=blue] (A) node[above]{$D(\beta_{k_p})$} (B) node[below]{$D(\beta_{k_{q+1}})$} (C) node[right]{$D(\beta_m)$};
\draw[thick, color=blue,->] (A) --(-2.5,1.3);
\draw[thick, color=blue,->] (B) --(2.4,-1.2);
\draw[thick, color=blue,->] (C) --(3.7,.1);
\begin{scope}
\clip (-3.3,-2)--(-2.7,2)--(-1.3,2)--(-.7,-2);
\draw[very thick,color=blue] (0,-1) ellipse [x radius=5cm, y radius= 2.5cm];
\end{scope}
\begin{scope}
\clip (.7,-2)--(1.3,2)--(3.3,2)--(2.7,-2);
\draw[very thick, color=blue] (0,-1.5)..controls (4,-1.5) and (5.8,1.5)..(7.5,1.5);
\end{scope}
\clip (0,-1) ellipse [x radius=5cm, y radius= 2.5cm];
\clip (0,-1.5)..controls (4,-1.5) and (5.8,1.5)..(7.5,1.5)--(-7.5,1.5)..controls (-5.8,1.5) and (-4,-1.5)..(0,-1.5);
\draw[very thick,color=blue] (-5,0)--(5,0);
\draw[thick] (-3.3,-2)--(-2.7,2) (-.7,-2)--(-1.3,2) (.7,-2)--(1.3,2) (2.7,-2)--(3.3,2) ;
\end{tikzpicture}
\caption{The green path $\gamma_1$ is in Class 1 since it is disjoint from $D(\beta_m)$. The green path $\gamma_2$ is in Class 2 and passes through three $\cS_0$-compartments $\cU_{\epsilon(p)},\cU_{\epsilon(r)},\cU_{\epsilon(q)}$ in $\cV_0=int(\cV(\beta_m)\backslash\cW(\beta_m))$. Each of these is divided into two $\cS$-compartments by the wall $D(\beta_m)$ and $\gamma_2$ passes through four of these $\cS$-compartments in $\cV_0$. $D(\beta_m)$ is the part of the hyperplane $H(\beta_m)$ inside the oval region $\cV(\beta_m)$ and outside of $\cW(\beta_m)$.}
\label{Fig: Inescapable regions}
\end{center}
\end{figure}

\subsection{Class 1 and Class 2 maximal $\cS_0$-green sequences}\label{ss2.2: Class 1 and 2}

Recall that a maximal $\cS_0$-green sequence with $\cS_0=(\beta_1,\cdots,\beta_{m-1})$ is in:
\begin{enumerate}
\item \und{Class 1} if each $\cS_0$-compartment $\cU_{\epsilon(i)}$ in the green sequence is disjoint from $D(\beta_m)$.
\item \und{Class 2} if at least one $\cS_0$-compartment, say $\cU_{\epsilon(j)}$, in the $\cS_0$-green sequence meets $D(\beta_m)$. So, $\cU_{\epsilon(j)}$ is divided into two $\cS$-compartments $\cU_{\epsilon(j),-}$ and $\cU_{\epsilon(j),+}$. See Figure \ref{Fig: Inescapable regions}.
\end{enumerate}

\begin{cor}\label{cor: two classes of green paths}
A maximal $\cS_0$-green sequence is in Class 1 if and only if it passes through $\cW(\beta_m)\backslash \cV(\beta_m)$. It is in Class 2 if and only if it contains a compartment in 
\[
\cV(\beta_m)\backslash \cW(\beta_m)=\{x\in \RR^n: \brk{x,\alpha}\le 0\text{ for all }\alpha\subsetneq \beta_m\text{ and }\brk{x,\gamma}\ge0\text{ for all }\beta_m\onto\gamma\}.
\]
\end{cor}

\begin{proof}
Every maximal green sequence starts on the negative side of the hyperplane $H(\beta_m)$ and ends on its positive side. Therefore the maximal $\cS_0$-green sequence must cross the hyperplane at some point. Since $\beta_m\notin \cS_0$, none of the $\cS_0$-compartments has $H(\beta_m)$ as a wall. So, there must be one compartment in the $\cS_0$-green sequences which meets the hyperplane $H(\beta_m)$. Let $\cU_\epsilon$ be the first such compartment. Then, either $\cU_\epsilon\cap D(\beta_m)$ is empty or nonempty. In the first case, $\cU_\epsilon$ is in $\cW(\beta_m)$ and it is outside $\cV(\beta_m)$. Since $\cW(\beta_m)$ is inescapable and does not meet $D(\beta_m)$, the green sequence is in Class 1. In the second case, $\cU_\epsilon$ is in $\cV(\beta_m)$ and not in $\cW(\beta_m)$ and the green sequence is in Class 2. So, these two cases correspond to Class 1 and Class 2 proving the corollary.
\end{proof}

Recall that $\cR(\beta_m)$ is the set of all $\alpha\in\cS_0$ which are hom-orthogonal to $\beta_m$. Let $\cV_0$ be the interior of the closed region $\cV(\beta_m)\backslash \cW(\beta_m)$. Thus
\[
	\cV_0:=\interior(\cV(\beta_m)\backslash \cW(\beta_m))\]
\[=\{x\in \RR^n: \forall\alpha\subsetneq \beta_m\brk{x,\alpha}< 0\text{ and }\brk{x,\gamma}>0\,\forall\beta_m\onto\gamma, \gamma\neq\beta_m\}.
\]

\begin{prop}\label{prop: D(a) meets V0 iff a lies in R(b)}
For all $\alpha\in\cS_0$, $\alpha\in\cR(\beta_m)$ if and only if $D(\alpha)\cap \cV_0\neq\emptyset$.
\end{prop}

\begin{proof}
Suppose that $x\in D(\alpha)\cap \cV_0$ and $hom(\beta_m,\alpha)\neq0$. Then there is a subroot $\alpha'$ of $\alpha$ which is also a quotient root of $\beta_m$: $\beta_m\onto \alpha'\subset \alpha$. Since $\alpha\in\cS_0$ we cannot have $\beta_m\subset \alpha$. Therefore $\alpha'$ is a proper quotient of $\beta_m$. Then $\brk{x,\alpha'}>0$ since $x\in\cV_0$ and $\brk{x,\alpha'}\le0$ since $x\in D(\alpha)$ and $\alpha'\subset\alpha$. This is a contradiction. So, $hom(\beta_m,\alpha)=0$. A similar argument shows that $hom(\alpha,\beta_m)=0$. So, $\alpha\in\cR(\beta_m)$.

Conversely, if $\alpha\in\cR(\beta_m)$ then $\alpha,\beta_m$ span a rank 2 wide subcategory $\cA(\alpha,\beta_m)$ of $mod\text-\Lambda$. Choose any tilting object $T$ in the left perpendicular category $^\perp\cA(\alpha,\beta_m)$ (for example the sum of the projective objects). Then the $g$-vector $g(\undim T)$ lies in the interior of both $D(\alpha)$ and $D(\beta_m)$ by Proposition \ref{prop: x0 in interior of D(a)} since $M_\alpha,M_{\beta_m}$ are the minimal objects in $T^\perp=\cA(\alpha,\beta_m)$. So, $g(\undim T)\in \cV_0$. So, $D(\alpha)$ meets $\cV_0$.
\end{proof}

\begin{cor}
The open region $\cV_0$ contains no vertices of the semi-invariant picture $L(\cS_0)$.
\end{cor}

\begin{proof}
Suppose that $x_0\in \cV_0$ is a vertex of $L(\cS_0)$. By Theorem \ref{thm: W(x0) is wide}, we have a wide subcategory $\cW(x_0)$ of all modules $V$ so that $x_0\in D(V)$. Since $x_0$ is a vertex of $L(\cS_0)$, the wide subcategory $\cW(x_0)$ must have rank $n-1$ and its minimal objects must lie in $\cS_0$, i.e., $\cW(x_0)=\cA(\alpha_1,\cdots,\alpha_{n-1})$ where $\alpha_i\in\cS_0$. 

By Proposition \ref{prop: D(a) meets V0 iff a lies in R(b)}, each $\alpha_i$ is hom-orthogonal to $\beta_m$. This implies that $\alpha_1,\cdots,\alpha_{n-1}$ together with $\beta_m$ form the minimal roots of a wide subcategory of rank $n$. By Theorem \ref{thm: rank n wide subcategory} this must be all of $mod\text-\Lambda$. So, $\beta_m$ must be a simple root contrary to our initial assumption. Therefore $\cV_0$ contains no vertices of $L(\cS_0)$.
\end{proof}

\begin{cor}\label{cor: corner sets meet D(bm)}
Let $\alpha_1,\cdots,\alpha_k$ be pairwise hom-orthogonal elements of $\cR(\beta_m)$ then the intersection $D(\alpha_1)\cap \cdots\cap D(\alpha_k)\cap D(\beta_m)\cap\cV_0$ is nonempty.
\end{cor}

\begin{proof}
More precisely, let $\cA(\alpha_1,\cdots,\alpha_k,\beta_m)$ be the rank $k+1$ wide subcategory of $mod\text-\Lambda$ with simple objects $M_{\alpha_i},M_{\beta_m}$. Let $T=T_1\oplus \cdots\oplus T_{n-k-1}$ be any cluster tilting object of the cluster category of $^\perp\cA(\alpha_1,\cdots,\alpha_k,\beta_m)$. Then the $g$-vector $g(\undim T)$ is a point in $D(\alpha_1)\cap\cdots\cap D(\alpha_k)\cap D(\beta_m)$ which lies in the interior of $D(\beta_m)$. This can be proved by induction on $k$ using the argument in the proof of Proposition \ref{prop: D(a) meets V0 iff a lies in R(b)}.
\end{proof}

\subsection{Proof of Lemma \ref{lem D: properties of compartments used in proof}}\label{ss2.3: proof of lemma D}

We will show that maximal $\cS_0$-green sequences satisfy the three properties listed in Lemma \ref{lem D: properties of compartments used in proof}.

\begin{prop}\label{prop: all compartments in V0 meet Dbm}
An $\cS_0$-compartment $\cU_\epsilon$ meets $D(\beta_m)$ if and only if $\cU_\epsilon\subseteq\cV_0$.
\end{prop}

Before proving this we show that this implies the first property in Lemma \ref{lem D: properties of compartments used in proof}. Recall that this states:\vs2

D(1) \emph{In every maximal $\cS_0$-green sequence in Class 2, the compartments which meet $D(\beta_m)$ are consecutive.}\vs2

\begin{proof}[Proof of D(1)]
Let $\cU_{\epsilon(i)}$ be a maximal $\cS_0$-green sequence. Let $p,q$ be minimal so that $\cU_{\epsilon(p)}\subseteq\cV(\beta_m)$ and $\cU_{\epsilon(q)}\subseteq\cW(\beta_m)$. When the green sequence is in Class 2, $p<q$. Since $\cV(\beta_m)$ is inescapable, $\cU_{\epsilon(i)}\subseteq\cV(\beta_m)$ iff $p\le i$. Since $\cW(\beta_m)$ is inescapable, $\cU_{\epsilon(i)}\subseteq\cV_0$ iff $p\le i<q$. So, the compartments of the green sequence which lie in $\cV_0$ are consecutive. By the proposition these are the compartments which meet $D(\beta_m)$.
\end{proof}

\begin{proof}[Proof of Proposition \ref{prop: all compartments in V0 meet Dbm}] Let $\cU_\epsilon$ be an $\cS_0$-compartment in $\cV_0$. Let $x\in \cU_\epsilon$. If $\brk{x,\beta_m}=0$ then $x\in H(\beta_m)\cap \cV_0\subset D(\beta_m)$ and we are done. So, suppose $\brk{x,\beta_m}\neq0$. Pick a point $y\in D(\beta_m)\cap \cV_0$ and take the straight line from $x$ to $y$. Since $\cV_0$ is convex, this line is entirely contained in $\cV_0$. If the line is not in $\cU_\epsilon$ then it must meet an internal wall $D(\alpha)$ on the boundary of $\cU_\epsilon$. By Proposition \ref{prop: D(a) meets V0 iff a lies in R(b)}, $\alpha\in\cR(\beta_m)$.

Let $k$ be maximal so that the closure of $\cU_\epsilon$ contains a point $z\in D(\alpha_\ast)=D(\alpha_1)\cap\cdots\cap D(\alpha_k)$ where $\alpha_1,\cdots,\alpha_k\in \cR(\beta_m)$ are pairwise hom-orthogonal. Then, by Corollary \ref{cor: corner sets meet D(bm)}, $D(\alpha_\ast)\cap D(\beta_m)\cap \cV_0$ is nonempty. Let $w$ be an element. Since $D(\alpha_\ast)$ and $\cV_0$ are both convex, $D(\alpha_\ast)\cap \cV_0$ contains the straight line $\gamma(t)=(1-t)z+tw$, $0\le t\le 1$.

Let $\delta$ be a very small vector so that $\brk{\delta,\beta_m}=0$ and $z+\delta\in\cU_\epsilon$. Consider the line $\gamma(t)+\delta$. This is in $\cU_\epsilon$ for $t=0$ and lies in $D(\beta_m)$ when $t=1$. This proves the proposition if $\gamma(t)+\delta\in \cU_\epsilon$ for all $0\le t\le1$. So, suppose not. Let $t_0$ be minimal so that this open condition fails. Then the line $\gamma(t)$ meets another wall at $t=t_0$ and $\gamma(t_0)$ will be a point in the closure of $\cU_\epsilon$ which meets a codimension $k+1$ set $D(\alpha_0)\cap D(\alpha_1)\cap\cdots\cap D(\alpha_k)$ where $\alpha_0\in \cS_0$ is hom-orthogonal to the other roots $\alpha_i$. (Take $\alpha_0$ of minimal length among the new roots so that $\gamma(t_0)\in D(\alpha_0)$.) This contradicts the maximality of $k$. So, there is no point $t_0$ and $\gamma(1)+\delta\in \cU_\epsilon\cap D(\beta_m)$ as claimed.
\end{proof}

We have already shown property (2) in Lemma \ref{lem D: properties of compartments used in proof}: Any maximal $\cS_0$-green sequence of Class 2 crosses a wall $D(\gamma)$ at some point to enter region $\cV_0$, passes through several internal walls of $\cV_0$, then exists $\cV_0$ by a wall $D(\alpha)$ of $\cW(\beta_m)$. By Propositions \ref{prop: walls to enter V}, \ref{prop: walls to enter W} $\gamma$ is a quotient root of $\beta_m$ and $\alpha$ is a subroot of $\beta_m$, both not hom-orthogonal to $\beta_m$. By Proposition \ref{prop: D(a) meets V0 iff a lies in R(b)} the internal walls of $\cV_0$ are $D(\beta)$ where $\beta\in\cR(\beta_m)$. So, property (2) in Lemma \ref{lem D: properties of compartments used in proof} holds.

The last property we need to verify in Lemma \ref{lem D: properties of compartments used in proof} is the following.\vs2

D(3)\emph{
Suppose that the two $\cS_0$-compartments $\cU_{\epsilon(1)}$ and $\cU_{\epsilon(2)}$ meet along a common internal wall $D(\beta_j)$. Then the $\cS$-compartments $\cU_{\epsilon(1),+},\cU_{\epsilon(2),+}$ meet along the common internal wall $D(\beta_j)$ and the $\cS$-compartments $\cU_{\epsilon(1),-},\cU_{\epsilon(2),-}$ also meet along $D(\beta_j)$.
}\vs2

\begin{proof}
Let $\cS'$, $\cS_0'$ be $\cS,\cS_0$ with $\beta_j$ deleted. Then $\cS',\cS_0'$ are weakly admissible. Since $\beta_j\notin\cS_0'$, the two $\cS_0$-compartments $\cU_{\epsilon(1)}$ and $\cU_{\epsilon(2)}$ merge to form one $\cS_0'$-compartment $\cU_\epsilon$. This compartment meets $D(\beta_m)$ so it breaks up into two $\cS'$-compartments $\cU_{\epsilon,+}$ and $\cU_{\epsilon,-}$. We know that $D(\beta_j)$ must divide these two $\cS'$-compartments into four $\cS$-compartments since $\cU_{\epsilon(1)}$, $\cU_{\epsilon(2)}$ are both divided into two parts by $D(\beta_m)$. Since $\cS'$-compartments are convex by Proposition \ref{prop: weakly admissible S gives convex compartments}, this can happen only if $D(\beta_j)$ meets both $\cS'$-compartments and forms the common wall separating the two halves of each.
\end{proof}

\section{Planar pictures and group theory}\label{sec3: Pictures}

In this section we will use planar pictures to prove the two properties of the group $G(\cS)$ that we are using: Lemmas \ref{lem C: beta m occurs a fixed number of times} and \ref{lem E: beta m only commutes with hom orthogonal roots}. The key tool will be the ``sliding lemma'' (Lemma \ref{lem: sliding lemma}) which comes from the first author's PhD thesis \cite{I:thesis}. Unless otherwise stated, all pictures in this section will be planar. We begin with a review of the topological definition of a (planar) picture with special language coming from the fact that all relations in our group $G(\cS)$ are commutator relations. Since this section uses only planar diagrams, we feel that theorems can be proven using diagrams and topological arguments. Algebraic versions of these arguments using HNN extensions, geometric realizations of categories and cubical CAT(0) categories can be found in other papers which prove similar results for pictures of arbitrary dimension (\cite{IT13}, \cite{Cat0}, \cite{IOTW4}). The first example of a picture group for a picture group of a Dynkin quiver appears in a paper by Jean-Louis Loday \cite{Loday} where the picture group for the quiver of type $A_n$ with straight orientation is constructed and the picture space (the $K(\pi,1)$ for the picture group) is also constructed.

\subsection{Planar pictures}\label{ss3.1: pictures in general}

Suppose that the group $G$ has a presentation $G=\left<\cX\,|\,\cY\right>$. This means there is an exact sequence \[
	R_\cY\into F_\cX\onto G
\]
where $F_\cX$ is the free group generated by the set $\cX$ and $R_\cY\subseteq F_\cX$ is the normal subgroup generated by the subset $\cY\subseteq F_\cX$. Then $G$ is the fundamental group of a 2-dimensional CW-complex $X^2$ given as follows. Let $X^1$ denote the 1-dimensional CW-complex having a single 0-cell $e^0$, one 1-cell $e^1(x)$ for every generator $x\in\cX$ attached on $e^0$. Then $\pi_1X^1=F_\cX$ and any $f\in F_\cX$ gives a continuous mapping $\eta_f:S^1\to X^1$ given by composing the loops corresponding to each letter in the unique reduced expression for $f$. Here $S^1=\{z\in\CC:||z||=1\}$, $1\in S^1$ is the basepoint and $S^1$ is oriented counterclockwise.

Let $X^2$ denote the 2-dimensional CW-complex given by attaching one 2-cell $e^2(r)$ for every relations $r\in\cY$ using an attaching map
\[
	\eta_r:S^1\to X^1
\]
homotopic to the one described above. We choose each mapping $\eta_r$ so that it is transverse to the centers of the 1-cells of $X^1$. So, the inverse images of these center points are fixed finite subsets of $S^1$. The relation $r$ is given by the union of these finite sets, call it $E_r\subset S^1$, together with a mapping $\lambda:E_r\to \cX\cup \cX^{-1}$ indicating which 1-cell the point goes to and in which direction the image of $\eta_r$ traverses that 1-cell. Then we have:
\[
	r=\prod_{x\in E_r}\lambda(x)\in F_\cX.
\]
The circle $S^1$ is the boundary of the unit disk $D^2=\{x\in \CC: ||x||\le 1\}$. Let $CE_r\subset D^2$ denote the cone of the set $E_r$:
\[
	CE_r:=\bigcup_{x\in E_r}\{ax\in D^2: 0\le a\le 1\}.
\]
This is the union of the straight lines from all $x\in E_r$ to $0\in D^2$.
\begin{figure}[htbp]
\begin{center}
\begin{tikzpicture}
\clip (-6.4,-1) rectangle (6.5,1.2);
	\draw (0,0) circle[radius=.7];
	\draw[thick] (1.5,-.5) -- (0,0) -- (-1,1);
	\draw[thick] (1,1) -- (0,0) -- (-1,-1);
	\draw (.85,.1) node{*} (1.2,-.4) node[below]{$z$} (.7,1) node{$x$};
	\draw (-1,-.7) node{$x$} (-.7,1) node{$y$};
	\draw (-5,.5) node{Example:};
	\draw (-5,0) node{$r=x^{-1}yxz$};
\end{tikzpicture}
\caption{The cone of $E_r$ in $D^2$ is the part inside the circle $S^1$. The asterisks $\ast$ indicates the position of the basepoint $1\in S^1$. The labels are drawn on the negative side of each edge.}
\label{figCEr}
\end{center}
\end{figure}

A picture is a geometric representation of a continuous pointed mapping $\theta:S^2\to X^2$ where pointed means preserving the base point. A (pointed) deformation of a picture represents a homotopy of such a mapping. Deformation classes of pictures form a module over the group ring $\ZZ G$.

\begin{defn}\label{def: planar picture}
Given a group $G$ with presentation $G=\left<\cX\,|\,\cY\right>$ and fixed choices of $E_r\subset S^1$, $\lambda:E^4\to \cX\cup\cX^{-1}$, a \und{picture} for $G$ is defined to be a graph $L$ embedded in the plane $\RR^2$ with circular edges allowed, together with: 
\begin{enumerate}
\item a label $x\in \cX$ for every edge in $L$,
\item a normal orientation for each edge in $L$,
\item a label $r\in \cY\cup \cY^{-1}$ for each vertex in $L$,
\item for each vertex $v$, a smooth ($C^\infty$) embedding $\theta_v:D^2\to \RR^2$ sending 0 to $v$
\end{enumerate} 
satisfying the following where $E(x)$ denotes the union of edges labeled $x$.
\begin{enumerate}
\item[(a)] Each $E(x)$ is a smoothly embedded 1-manifold in $\RR^2$ except possibly at the vertices.
\item[(b)] For each vertex $v\in L$, $\theta_v^{-1}(E(x))\subseteq CE_r$ is equal to the cone of $\lambda^{-1}(\{x,x^{-1}\})\subset E_r$.
\end{enumerate}
The image of $1\in S^1$ under $\theta_v:D^2\to \RR^2$ will be called the \und{basepoint direction} of $v$ and will be indicated with $\ast$ when necessary.
\end{defn} 

The embedding $\theta_v$ has positive, negative orientation when $r\in\cY$, $r\in\cY^{-1}$, respectively.

One easy consequence of this definition is the following.

\begin{prop}
Given a picture $L$ for $G$, there is a unique label $g(U)\in G$ for each component $U$ of the complement of $L$ in $\RR^2$ having the following properties.
\begin{enumerate}
\item $g(U_\infty)=1$ for the unique unbounded component $U_\infty$.
\item $g(V)=g(U)x$ if the regions $U,V$ are separated by an edge labeled $x$ and oriented towards $V$.
\end{enumerate}
\end{prop}

\begin{proof}
For any region $U$, choose a smooth path from $\infty$ to any point in $U$. Make the path transverse to all edge sets. Then let $g(U)=x_1^{\epsilon_1}\cdots x_m^{\epsilon_m}$ if the path crosses $m$ edges labeled $x_1,\cdots,x_m$ with orientations given by $\epsilon_i$. This is well defined since any deformation of the path which fixes the endpoints and which pushes it through a vertex will not change the product $g(U)$ since the paths on either side of the vertex have edge labels giving a relation in the group and therefore give the same product of labels in the group $G$.
\end{proof}

\begin{rem} Any particular smooth path $\gamma$ from $U_\infty$ to $U$ gives a lifting $f_\gamma(U)$ of $g(U)$ to the free group $F_\cX$.
\end{rem}

It is well-known that the set of deformation classes of pictures for any group $G$ is a $\ZZ G$-module $P(G)$. (See Theorem \ref{thm: PG is kernel of d2} and Corollary \ref{cor: long exact sequence for PG} below.)

The action of the group $G$ is very easy to describe. Given any picture $L$ and any generator $x\in \cX$, the pictures $xL, x^{-1}L$ are given by enclosing the set $L$ with a large circle, labeling the circle with $x$ and orienting it inward or outward, respectively. Addition of pictures is given by disjoint union of translates of the pictures.

To define the equivalence relation which we call ``deformation equivalence'' of pictures, it is helpful to associate to each picture $L$ an element $\psi(L)\in \ZZ G\!\left<\cY\right>$ where $\ZZ G\!\brk{\cY}$ is the free $\ZZ G$ module generated by the set of relations $\cY$. This is given by
\[
	\psi(L)=\sum_{v_i} g(v_i) \brk{r_i}
\]
where the sum is over all vertices $v_i$ of $L$, $r_i\in \cY\cup\cY^{-1}$ is the relation at $v_i$, $g(v_i)\in G$ is the group label at the basepoint direction of $v_i$ and $\brk{r^{-1}}=-\brk r$ by definition.

\begin{defn}
A \und{deformation} $L_0\simeq L_1$ of pictures for $G$ is defined to be a sequence of allowable moves given as follows.
\begin{enumerate}
\item Isotopy. $L_0\simeq L_1$ if there is an orientation preserving diffeomorphism $\varphi:\RR^2\to \RR^2$ so that $L_1=\varphi(L_0)$ with corresponding labels. By isotopy we can make the images of the embeddings $\theta_v:D^2\to\RR^2$ disjoint and arbitrarily small.
\item Smooth concordance of edge sets. There are two concordance moves:
\begin{enumerate}
	\item If $L_0$ contains a circular edge $E$ with no vertices and $L_0$ does not have any point in the region enclosed by $E$ then $L_0\simeq L_1$ where $L_1$ is obtained from $L_0$ by deleting $E$.
	\item If $U$ is a connected component of $\RR^2-L_0$ and two of the walls of $U$ have the same label $x$ oriented in the same way (inward towards $U$ or outward) then, choose a path $\gamma$ in $U$ connecting points on these two edges then perform the following modification of $L_0$ in a neighborhood of $\gamma$ to obtain $L_1\simeq L_0$.
\begin{figure}[htbp]
\begin{center}
{
\setlength{\unitlength}{1cm}
%\centerline
{\mbox{
\begin{picture}(7,2)
      \thicklines
%    \thinlines
  %
%    \put(-.1,0){\circle*{.05}}
    \put(0,0){
\qbezier(0,0)(1,1)(0,2)    
    \put(2,0){
\qbezier(0,0)(-1,1)(0,2)   }
\put(0,1.4){$x$}
\put(2,0.4){$x$}
\put(-1,1.8){$L_0:$}
%      \thicklines
   \thinlines
\put(.5,1){\line(1,0)1}
      \thicklines
%    \thinlines
\put(.9,1.1){$\gamma$}
}
\put(3.3,1){$\then$}
    \put(5.5,0){
\qbezier(0,0)(1,1)(2,0)    
    \put(0,0){
\qbezier(0,2)(1,1)(2,2)   }
\put(0,1.4){$x$}
\put(2,0.4){$x$}
\put(-1,1.8){$L_1:$}
}
\end{picture}}
}}
\end{center}
\end{figure}
\end{enumerate}
\item Cancellation of vertices. Suppose that two vertices $v_0,v_1$ of $L_0$ have inverse labels $r,r^{-1}$. Suppose that there is a path $\gamma$ disjoint from $L_0$ connecting the basepoint directions of $v_0,v_1$. Let $V$ be the union of the $\theta_{v_0}(D^2),\theta_{v_1}(D^2)$ and a small neighborhood of the path $\gamma$. We can choose $V$ to be diffeomorphic to $D^2$. Then $L_0\simeq L_1$ if $L_0,L_1$ are identical outside of the region $V$ and $L_1$ has no vertices in $V$. (The two vertices in $V\cap L_0$ cancel.)
\end{enumerate}
\begin{center}
\begin{tikzpicture}
	\draw (0,0) node{$\then$};
	\draw (-6,0) circle[radius=.7]  (-2.5,0) circle[radius=.7] (2.5,0) circle[radius=.7] (6,0) circle[radius=.7];
	\draw (-5.7,0)node{$\ast$}  -- (-2.8,0) node{$\ast$};
	\draw[thick] (-6,0)  +(-1,.8)-- +(0,0) +(3.5,0) -- +(4.5,.8);
	\draw[thick] (-6,0) +(-1,-1)-- +(0,0) +(3.5,0) -- +(4.5,-1);
	\draw[thick] (-6,0) +(-1.3,-.2)-- +(0,0)  +(3.5,0) -- +(4.8,-.2);
	\draw (-6,0) +(1.75,0) node[above]{$\gamma$};
	\draw (-6,0) +(-1,.8) node[above]{$x$} +(-1,-1) node[above]{$z$} +(-1.3,-.2) node[above]{$y$} ;
	\draw (-2.5,0) +(1,.8) node[above]{$x$} +(1,-1) node[above]{$z$} +(1.3,-.2) node[above]{$y$} ;
	\draw (-6.2,1.5) node{$L_0$};
	\draw (2,1.5) node{$L_1$};
	\draw[thick] (2.5,0)  +(-1,.8).. controls (2.5,0) and (5,0) .. +(4.5,.8);
	\draw[thick] (2.5,0) +(-1,-1).. controls (2.5,0) and (5,0) .. +(4.5,-1);
	\draw[thick] (2.5,0) +(-1.3,-.2).. controls (2.5,0) and (5,0) ..  +(4.8,-.2);
	\draw (2.5,0) +(-1,.8) node[above]{$x$} +(-1,-1) node[above]{$z$} +(-1.3,-.2) node[above]{$y$} ;
	\draw (6,0) +(1,.8) node[above]{$x$} +(1,-1) node[above]{$z$} +(1.3,-.2) node[above]{$y$} ;
\end{tikzpicture}
\end{center}
\end{defn}

Concordance means $L_0,L_1$ have the same vertex sets and are equal in a neighborhood of each vertex and that $f_{\gamma_i}\in F_\cX$ are equal for $L_0$, $L_1$ for some (and thus every) choice of paths $\gamma_i$ disjoint from vertices from $\infty$ to the basepoint direction of each vertex of $L_0$. The same paths work for $L_1$ since $L_0,L_1$ have the same vertex set.

\begin{thm}\cite{JW}\cite[Prop 7.4]{I:dilog1}\label{thm: PG is kernel of d2}
$L_0,L_1$ are deformation equivalent if and only if $\psi(L_0)=\psi(L_1)$. Furthermore, the set of possible values of $\psi(L)$ for all pictures $L$ is equal to the kernel of the mapping
\[
	\ZZ G\!\brk \cY\xrightarrow{d_2} \ZZ G\!\brk\cX
\]
where $d_2\brk r=\sum \partial_x r\!\brk x$, where $\partial_x$ is the Fox derivative of $r$ with respect to $x$. %and $\brk{x^{-1}}:=-x^{-1}\brk x$.
\end{thm}

The \und{Fox derivative} of $w\in F_\cX$ is given recursively on the reduced length of $w$ by
\begin{enumerate}
\item $\partial_x(x)=1$, $\partial_x(x^{-1})=-x^{-1}$.
\item $\partial_x(y)=0$ if $y\in\cX\cup \cX^{-1}$ is not equal to $x,x^{-1}$.
\item $\partial_x(ab)=\partial_x a+a\partial_xb$ for any $a,b\in F_\cX$.
\end{enumerate}

\begin{defn} The \und{group of pictures} $P(G)$ is defined to be the group of deformation classes of pictures for $G$. 
\end{defn}

\begin{cor}\label{cor: long exact sequence for PG}
There is an exact sequence of $\ZZ G$-modules
\[
	0\to P(G)\to \ZZ G\!\brk \cY\xrightarrow{d_2} \ZZ G\!\brk\cX \xrightarrow{d_1} \ZZ G\xrarrow\epsilon \ZZ\to 0
\]
where $d_1\sum a_i\brk{x_i}=\sum a_i(x_i-1)$, $\epsilon:\ZZ G\to \ZZ$ is the augmentation map and $d_2$ is as above.
\end{cor}

\begin{rem}
The chain complex $\ZZ G\!\brk \cY\xrightarrow{d_2} \ZZ G\!\brk\cX \xrightarrow{d_1} \ZZ G$ is the cellular chain complex of the universal covering $\widetilde{X^2}$ of the 2-dimensional CW complex $X^2$ constructed above. Since $\widetilde{X^2}$ is simply connected, we have 
\[
P(G)=H_2(\widetilde{X^2})=\pi_2(\widetilde{X^2})=\pi_2(X^2).
\]
Therefore, $P(G)=\pi_2(X^2)$ as claimed at the beginning of this subsection.
\end{rem}

We also use ``partial pictures''. These are given by cutting a picture in half using a straight line transverse to the picture.

\begin{defn} Let $w$ be a word in $\cX\cup \cX^{-1}$ given by a finite subset $W$ of the $x$-axis in $\RR^2$ together with a mapping $W\to \cX\cup\cX^{-1}$.
A \und{partial picture} with boundary $\partial L=w$ is defined to be a closed subset $L$ of the upper half plane so that the intersection of $L$ with the $x$-axis is equal to $W$ together with labels on $L$ so that the union of $L$ and its mirror image $L_-$ in the lower half plane is a picture for $G$ and so that the labels on the edges which cross the $x$-axis agree with the given mapping $W\to \cX\cup\cX^{-1}$. We call $L\cup L_-$ the \und{double} of $L$.
\end{defn}

\begin{figure}[htbp]
\begin{center}
\begin{tikzpicture}%[scale=3]
%\draw[help lines=1,thick] (-5,-5) grid (5,4);
%\foreach \x in {-8,-6,...,8}\draw (\x,0) node{\x};\foreach \y in {-6,-4,...,8}\draw (0,\y) node{\y};
%
\begin{scope}
\coordinate (A0) at (-2.5,.2); % begin first L
\coordinate (A1) at (-2.5,0);
\coordinate (Aa) at (-2.5,2);
\coordinate (Ab) at (-1.5,2.2);
\coordinate (A2) at (-1,2.2);
\coordinate (A2a) at (-.5,2.2);
\coordinate (A2b) at (0,2);
\coordinate (A3) at (0,0);
\coordinate (B0) at (-3,0.2); % first L
\coordinate (B1) at (-3,0);
\coordinate (Ba) at (-3,2.4);
\coordinate (Bb) at (-1.5,2.7);
\coordinate (B2) at (-1,2.7);
\coordinate (B2a) at (0,2.7);
\coordinate (B2b) at (0.5,2);
\coordinate (B3) at (0.5,0);
\coordinate (C) at (-1.2,1.5); % r1, first L
\coordinate (C1a) at (-1.4,1.2);
\coordinate (C1b) at (-1.6,1);
\coordinate (C1) at (-1.9,0);
\coordinate (C2) at (-1.3,0);
\coordinate (C3a) at (-1,1.3);
\coordinate (C3b) at (-.7,1);
\coordinate (C3) at (-.5,0);
\coordinate (C10) at (-1.9,0.2);
\coordinate (C20) at (-1.3,0.2);
\coordinate (C30) at (-.5,0.2);
\coordinate (C10b) at (-1.9,-0.2);
\coordinate (C20b) at (-1.2,-0.2);
\coordinate (C30b) at (-.5,-0.2);
\coordinate (D0) at (1.3,0.2); % first L
\coordinate (D1) at (1.3,0);
\coordinate (Da) at (1.3,2);
\coordinate (Db) at (1.3,2.2);
\coordinate (D2) at (2.3,2.2);
\coordinate (D2a) at (2.5,2.2);
\coordinate (D2b) at (2.8,2);
\coordinate (D3) at (2.8,0);
\coordinate (E) at (2,1.5); % r2, first L
\coordinate (E1) at (1.7,0);
\coordinate (E2) at (2.3,0);
\coordinate (E10b) at (1.7,-.2);
\coordinate (E20b) at (2.3,-.2);
\coordinate (E15b) at (2,-.2);

\coordinate (L) at (-3,3);

\draw (L) node{$L_1:$};

\coordinate (Cp) at (-1.2,1.9);
\coordinate (Cpp) at (-1.2,4);
\draw[thick,dotted,<-] (Cp)--(Cpp); % dotted r1
\draw (Cpp) node[right]{$\ell_1$};

\coordinate (Ep) at (2,1.9);
\coordinate (Epp) at (2,4);
\draw[thick,dotted,<-] (Ep)--(Epp); % dotted r2
\draw (Epp) node[right]{$\ell_2$};

\draw (A0) node[left]{$b$}; % first L
\draw (B0) node[left]{$a$};
\draw (D0) node[left]{$c$};
\draw[blue] (C10) node[right]{$x$}; % r1, first L
\draw[blue] (C20) node[right]{$y$};
\draw[blue] (C30) node[left]{$z$};
\draw[blue,thick] 
(C20b) node[below]{$r_1$};
\draw[blue,thick,->] (C10b)--(C30b);

\draw[->] (-3.5,0)--(3,0); % first L
\draw[thick] (A1)..controls (Aa) and (Ab) ..(A2)..controls (A2a) and (A2b)..(A3);
\draw[thick] (B1)..controls (Ba) and (Bb) ..(B2)..controls (B2a) and (B2b)..(B3);
\draw[blue] (C) node[above]{$\ast$}
(C) node[right]{$v_1$}; to % r1, first L
\draw[blue,thick] (C)..controls (C1a) and (C1b)..(C1) 
(C)--(C2)
(C)..controls (C3a) and (C3b)..(C3);
\draw[thick] (D1)..controls (Da) and (Db) ..(D2)..controls (D2a) and (D2b)..(D3);
\draw[thick,red] (E1)--(E)--(E2) (E)node[above]{$\ast$}
 (E15b) node[below]{$r_2$}
 (E) node[right]{$v_2$};
\draw[thick,red,->] 
(E10b)--(E20b) ; % r2, first L
\end{scope} % end first L
\draw (3.5,1.2) node{$\then$};
\begin{scope}[xshift=9cm] % begin second L'
\coordinate (A1) at (-2.5,0);
\coordinate (Aa) at (-2.5,2);
\coordinate (Ab) at (-1.5,2.2);
\coordinate (A2) at (-1,2.2);
\coordinate (A2a) at (-.5,2.2);
\coordinate (A2b) at (0,2);
\coordinate (A3) at (0,0);
\coordinate (B1) at (-3,0); % second L'
\coordinate (Ba) at (-3,2.4);
\coordinate (Bb) at (-1.5,2.7);
\coordinate (B2) at (-1,2.7);
\coordinate (B2a) at (0,2.7);
\coordinate (B2b) at (0.5,2);
\coordinate (B3) at (0.5,0);
\coordinate (C) at (-1.2,1.5); % r1, second L'
\coordinate (Cp) at (-1.2,1.9);
\coordinate (Cpp) at (-1.2,4.4);
\draw[thick,dotted,<-] (Cp)--(Cpp); % dotted r1, L'
\draw (Cpp) node[right]{$\ell_1'$};

\coordinate (C1a) at (-1.4,1.2); % r1, second L'
\coordinate (C1b) at (-1.6,1);
\coordinate (C1) at (-1.9,0);
\coordinate (C2) at (-1.3,0);
\coordinate (C3a) at (-1,1.3);
\coordinate (C3b) at (-.7,1);
\coordinate (C3) at (-.5,0);
\coordinate (C10b) at (-1.9,-0.2);
\coordinate (C20b) at (-1.2,-0.2);
\coordinate (C30b) at (-.5,-0.2);

\coordinate (D0) at (1.3,0.2); % second L'
\coordinate (D1) at (1.3,0);
\coordinate (Da) at (1.3,2);
\coordinate (Db) at (1.3,3);
\coordinate (D2) at (-1,3);
\coordinate (D2a) at (-3,3);
\coordinate (D2b) at (-3.5,2.5);
\coordinate (D2c) at (-3.8,0);
\coordinate (D2c0) at (-3.8,0.2);

\coordinate (Dx0) at (-4.3,0.2); % second L'
\coordinate (Dx) at (-4.3,0);
\coordinate (Dxa) at (-4,3);
\coordinate (Dxb) at (-3,4);
\coordinate (Dxc) at (-1,4);
%\coordinate (Dx2) at (2,3);
\coordinate (Dx2a) at (2.3,4);
\coordinate (Dx2b) at (2.8,3);
\coordinate (D3) at (2.8,0);
\coordinate (D30) at (2.8,0.2);

\coordinate (L) at (-4.3,3.3);

\draw (L) node{$L_2:$}; % second L'

%\draw[blue,thick] (C10b)--(C30b)
%(C20b) node[below]{$r_1$};
\draw[blue,thick] 
(C20b) node[below]{$r_1$};
\draw[blue,thick,->] (C10b)--(C30b);

\draw (D2c0) node[right]{$c$}; % second L'
\draw (D30) node[right]{$c$};
\draw (Dx0) node[left]{$c$};
\draw (D0) node[left]{$c$};
\coordinate (E) at (-2.3,3.3); % r2, second L'
\coordinate (E1a) at (1.3,3.3);
\coordinate (E1b) at (1.7,3);
\coordinate (E1) at (1.7,0);

\coordinate (Ep) at (-2.5,3.45);
\coordinate (Epp) at (-2.5,4.4);
\draw[thick,dotted,<-] (Ep)--(Epp); % dotted r2, L'

\coordinate (E2a) at (1.3,4);
\coordinate (E2b) at (2.3,3);
\coordinate (E2) at (2.3,0);

\coordinate (E10b) at (1.7,-.2);
\coordinate (E20b) at (2.3,-.2);
\coordinate (E15b) at (2,-.2);

\draw[thick,red] (E15b) node[below]{$r_2$};
\draw[thick,red,->] %(E1)--(E)--(E2) (E)node[above]{$\ast$}
(E10b)--(E20b) ; % second L'

\draw[->] (-4.5,0)--(3,0);
\draw[thick] (A1)..controls (Aa) and (Ab) ..(A2)..controls (A2a) and (A2b)..(A3);
\draw[thick] (B1)..controls (Ba) and (Bb) ..(B2)..controls (B2a) and (B2b)..(B3);
\draw[blue] (C) node[above]{$\ast$}; % second L'
\draw[thick,blue] (C)..controls (C1a) and (C1b)..(C1) 
(C)--(C2)
(C)..controls (C3a) and (C3b)..(C3); % r1, second L'
\draw[thick] (D1)..controls (Da) and (Db) ..(D2) ..controls (D2a) and (D2b)..(D2c) 
(Dxc)..controls (Dx2a) and (Dx2b)..(D3) 
(Dx)..controls (Dxa) and (Dxb)..(Dxc);
\draw[thick,red] 
(E)..controls (E1a) and (E1b)..(E1)
(E)..controls (E2a) and (E2b)..(E2)
(Ep)node[below]{$\ast$}; % r2, second L'
\end{scope} % end second L'
%
%\coordinate (A) at (0,0);
%\draw[thick,color=blue] (A) ellipse [x radius=2.8cm,y radius=2.1cm];
\end{tikzpicture}
\caption{On the left, $L_1$ is a partial picture with $\partial L_1=abr_1b^{-1}a^{-1}cr_2c^{-1}$ where $\color{blue}r_1=x^{-1}y^{-1}z$ and $\color{red}r_2$ are relations (or inverse relations). $L_1$ is the ``standard partial picture'' for $q(L_1)=(ab,r_1)(c,r_2)\in Q(G)$. On the right is $L_2$, a deformation of $L_1$ with $\partial L_2=cc^{-1}\partial L_1$. $q(L_2)=(c,r_2)(cr_2^{-1}c^{-1}ab,r_1)$ since the vertex for $r_2$ is on the left and $cr_2^{-1}c^{-1}ab$ is given by reading the labels on the dotted path $\ell_1'$. Then $q(L_1)=q(L_2)$ by \eqref{eq: relation for Q(G)}.}
\label{Fig: partial picture example}
\end{center}
\end{figure}

A \und{deformation} of a partial picture $L$ is defined to be any deformation of its double in which vertices do not cross the $x$-axis and which are transverse to the $x$-axis at the beginning and end of the deformation. (See Figure \ref{Fig: partial picture example} for an example where the deformation pushes the $c$ curve through the $x$-axis producing the cancelling pair $cc^{-1}$ in the word for $\partial L_2$.) It is clear that deformation of partial pictures preserves its boundary $\partial L=w$ as an element of the free group $F_\cX$ and that $w$ lies in the relation group $R_\cY\subseteq F_\cX$. The main theorem about partial pictures is the following.

\begin{thm}\label{thm: Peiffer presentation}
The set of deformation classes of partial pictures forms a (nonabelian) group $Q(G)$ given by generators and relations as follows.
\begin{enumerate}
\item The generators of $Q(G)$ are pairs $(f,r)$ where $f\in F_\cX$ and $r\in\cY$.
\item The relations in $Q(G)$ are given by
\[
	(f,r)(f',r')(f,r)^{-1}=(frf^{-1}f',r')
\]
\end{enumerate}
\end{thm}

Note that there is a well defined group homomorphism
\[
	\varphi:Q(G)\to F_\cX
\]
given by $\varphi(f,r)=frf^{-1}$. Then relation (2) can be written as
\begin{equation}\label{eq: relation for Q(G)}
	(f,r)(f',r')=(\varphi(f,r)f',r')(f,r).
\end{equation}
The image of $\varphi$ is $R_\cY$, the normal subgroup generated by all $r\in\cY$. We use the notation 
\begin{equation}\label{eq: Peiffer transformation 2}
	(f,r^{-1}):= (f,r)^{-1}.
\end{equation}
This is compatible with the relations in $Q(G)$ and with the homomorphism $\varphi$ since 
\[
	(f,r^{-1})(f',r')(f,r)=(fr^{-1}f^{-1}f',r')
\]
and $\varphi(f,r^{-1})=fr^{-1}f^{-1}=\varphi(f,r)^{-1}.$ We assume that the relations are irredundant. So, $\cY$, $\cY^{-1}$ will be disjoint.

The generators and relations for $Q(G)$ first appeared in a paper by Ren\'ee Peiffer \cite{Peiffer}. For this reason, the relations \eqref{eq: relation for Q(G)} and \eqref{eq: Peiffer transformation 2}, or rather the equivalent equation $(f,r)(f,r^{-1})=1$ are called \emph{Peiffer transformations of the first and second kind}, respectively \cite{LS}. 

\begin{proof} Given a partial picture $L$ for $G=\left<\cX|\cY\right>$, e.g. $L_1$ in Figure \ref{Fig: partial picture example}, the corresponding element $q(L_1)\in Q(G)$ is given as follows. 

First, by a small deformation of the partial picture, we may assume that the $x$-coordinates of the vertices of $L$ are all distinct. Label the vertices $v_1,\cdots,v_n$ from left to right (in order of increasing $x$-coordinates). From the basepoint direction of vertex $v_i$, draw a lines $\ell_i$ straight up. By a small deformation we can make the green lines transverse to $L$. (These are dotted arrow $\ell_1,\ell_2$ in Figure \ref{Fig: partial picture example}.) At each vertex $v_i$ left $r_i\in \cY\coprod \cY^{-1}$ be the relation at $v_i$ and let $f_i\in F$ be given by reading the labels of the edges in $L$ traversed by $\ell_i$ oriented towards $v_i$. The resulting element of $Q(G)$ is
\[
	q(L)=(f_1,r_1)\cdots(f_n,v_n).
\]
In Figure \ref{Fig: partial picture example}, on the left, we have $r_1=x^{-1}y^{-1}z$ ($r_2$ is not given) $f_1=ab$, $f_2=c$. This gives
\[
	q(L_1)=(ab, x^{-1}y^{-1}z)(c,r_2)
\]
On the right the vertices are in the reverse order. So, $(c,r_2)$ comes first. The new line $\ell_1'$ traverses six edges of $L_2$ giving $(cr_2^{-1}c^{-1}ab,r_1)$. So, the element of $Q(G)$ associated to $L_2$ is
\[
	q(L_2)=(c,r_2)(cr_2^{-1}c^{-1}ab,r_1).
\]
By \eqref{eq: relation for Q(G)} we see that $q(L_1)=q(L_2)\in Q(G)$.\vs2

\noi\ul{Claim 1}: $q(L)\in Q(G)$ is invariant under deformations of $L$ and therefore well-defined.\vs2

\noi\emph{Proof:} First, consider deformations which keep the vertices $v_1,\cdots,v_n$ in the same order. Then the lines $\ell_1,\cdots,\ell_n$ will cross edges whose labels give the same elements $f_1,\cdots,f_n\in F$ since the only changes will be to add or delete cancelling pairs of edges labeled $x,x^{-1}$. So, $q(L)$ remains the same.

Next, consider deformations in which the order of the vertices changes. This happens when, at some point in the deformation, one vertex, say $v_i$, passes above  the next, $v_{i+1}$ (or the previous one $v_{i-1}$ as in Figure \ref{Fig: partial picture example}). In that case, the line $\ell_i$ will cross the same edges as before, but the line $\ell_{i+1}$ will cross edges $f_{i+1}$ before and $\varphi(f_i,r_i)f_{i+1}$ after the deformation changing $(f_i,r_i)(f_{i+1},r_{i+1})$ to $(\varphi(f_i,r_i)f_{i+1},r_{i+1})(f_i,r_i)$ which is \eqref{eq: relation for Q(G)}. Thus $q(L)\in Q(G)$ is unchanged.

Finally, consider a deformation in which two vertices are cancelled. In that case, they must be consecutive, say $v_i,v_{i+1}$, the relations $r_i,r_{i+1}$ must be inverse to each other and the paths $\ell_i,\ell_{i+1}$ must cross the same edges making $f_i=f_{i+1}$, since otherwise, the vertices are not allowed to cancel. So, $(f_{i+1},r_{i+1})=(f_i,r_i^{-1})$ which cancels $(f_i,r_i)$ in $Q(G)$. So, $q(L)$ is unchanged in all deformations. 

\vs2

Conversely, let $Q=(f_1, r_1) \cdots(f_n,r_n)\in  Q(G)$.
We will construct the ``standard
partial picture'' $L_Q$ satisfying $q(L_Q)=Q$.
An example of a standard picture is $L_1$
in Figure \ref{Fig: partial picture example}.

\begin{enumerate}
	\item Let $w_i$ be the unique reduced word
in the letters $\cX\coprod \cX^{-1}$ represents $f_i$. Each $r_i$ is already given as a (cyclically) reduced
word. Let $w(Q) = w_1r_1w_1^{-1}\cdots w_nr_n w_n^{-1}$.
	\item Along the $x$-axis choose a sequence of points one for each letter in $w(Q)$ and label these points with the letters of $w(Q)$.
	\item Connect the points labeled with the
letters in $r_i$ to a point $v_i$ above
these points with line segments
labeled with the letters of $r_i$.
Place a base point direction $\ast$ above $v_i$.
Then the word given by reading the
edge labels counterclockwise around $v_i$ staring at $\ast$ will be $r_i$.
	\item From the points labeled with
the letters in $w_i$, $w_i^{-1}$ draw
vertical lines going up labeled with
the letters of $w_i$, $w_i^{-1}$. Above vertex $v_i$
connect the lines from $w_i$ to ones from $w_i^{-1}$ with semicircles centered at $v_i$. Since all the loose edges in the upper half-plane have been closed off,
this gives a partial picture $L$.
We denote this $L_Q$ and call it the \emph{standard partial picture} corresponding to $Q$.
\end{enumerate}
\vs2

\noi\ul{Claim 2}: $q(L_Q)=Q$.

\noi\ul{Claim 3}: $L_{q(L)}\simeq L$.
\vs2

These two claims imply that $Q\mapsto L_Q$, $L\mapsto q(L)$ give a 1-1 correspondence between deformation classes of partial pictures and the elements of
$Q(G)$.\vs2

\noi\emph{Proof of Claim 2}. This follows directly
from the construction of $L_Q$. The straight
line going up from each vertex $v_i$ will cross the
picture through semicircular edges labelled
with the letters of $w_i$. The relation at $v_i$ is $r_i$
by construction. So $q(L_Q)=(w_1,r_1)\cdots(w_n,r_n)=Q$.\vs2

\noi\emph{Proof of Claim 3}. Given a partial picture
$L$ with $q(L) = Q$, for example $L_2$
in Figure \ref{Fig: partial picture example},
a deformation of $L$ to the standard picture $L_Q$
is given by ``pushing down'' to the $x$-axis all edges outside a small nbd of the lines $\ell_i$.
Since there are no vertices of $L$ outside these neighborhoods, this deformation is allowed.
The result is a standard partial picture
for $Q=q(L)$. See Figure \ref{Fig: proof of Claim 3}.%%%%%%%

Thus $Q\leftrightarrow L_Q$ is a bijection as claimed.
\end{proof}

\begin{figure}[htbp]
\begin{center}
\begin{tikzpicture}[scale=.7]
%\draw[help lines=1,thick] (-5,-1) grid (16.5,6);
%\foreach \x in {-8,-6,...,8}\draw (\x,0) node{\x};\foreach \y in {-6,-4,...,8}\draw (0,\y) node{\y};
%
%
\clip (-5,-1) rectangle (16.5,6);

\begin{scope} %[xshift=9cm] % begin first L'
\coordinate (A1) at (-2.5,0);
\coordinate (Aa) at (-2.5,2);
\coordinate (Ab) at (-1.5,2.2);
\coordinate (A2) at (-1,2.2);
\coordinate (A2a) at (-.5,2.2);
\coordinate (A2b) at (0,2);
\coordinate (A3) at (0,0);
\coordinate (B1) at (-3,0); % first L'
\coordinate (Ba) at (-3,2.4);
\coordinate (Bb) at (-1.5,2.7);
\coordinate (B2) at (-1,2.7);
\coordinate (B2a) at (0,2.7);
\coordinate (B2b) at (0.5,2);
\coordinate (B3) at (0.5,0);

\coordinate (Cxx) at (-1.6,1.5);
\coordinate (Cxxx) at (-1.6,4.4);
\coordinate (Cyy) at (-.8,1.5);
\coordinate (Cyyy) at (-.8,4.4);

\coordinate (Czz) at (-1.2,1.5);
\draw[thick,dashed] (Czz) circle[radius=4mm];
\draw[fill,white] (Cxx) rectangle (Cyyy);

\coordinate (C) at (-1.2,1.5); % r2, first L'
\coordinate (C-) at (-1.2,1.4); % r2, first L'
\coordinate (Cp) at (-1.2,1.9);
\coordinate (Cpp) at (-1.2,4.4);

\draw[thick,dashed] (Cxx)--(Cxxx);
\draw[thick,dashed] (Cyy)--(Cyyy);

\draw[thick,dotted,<-] (Cp)--(Cpp); % dotted r2, L'
\draw (Cpp) node[above]{$\ell_2$};

\coordinate (C1a) at (-1.4,1.2); % r2, first L'
\coordinate (C1b) at (-1.6,1);
\coordinate (C1) at (-1.9,0);
\coordinate (C2) at (-1.3,0);
\coordinate (C3a) at (-1,1.3);
\coordinate (C3b) at (-.7,1);
\coordinate (C3) at (-.5,0);
\coordinate (C10b) at (-1.9,-0.2);
\coordinate (C20b) at (-1.2,-0.2);
\coordinate (C30b) at (-.5,-0.2);

\coordinate (D0) at (1.3,0.2); % first L'
\coordinate (D1) at (1.3,0);
\coordinate (Da) at (1.3,2);
\coordinate (Db) at (1.3,3);
\coordinate (D2) at (-1,3);
\coordinate (D2a) at (-3,3);
\coordinate (D2b) at (-3.5,1.5);
\coordinate (D2c) at (-3.8,0);
\coordinate (D2c0) at (-3.8,0.2);

\coordinate (Dx0) at (-4.3,0.2); % first L'
\coordinate (Dx) at (-4.3,0);
\coordinate (Dxa) at (-4,3);
\coordinate (Dxb) at (-3,4);
\coordinate (Dxc) at (-1,4);
%\coordinate (Dx2) at (2,3);
\coordinate (Dx2a) at (2.3,4);
\coordinate (Dx2b) at (2.8,3);
\coordinate (D3) at (2.8,0);
\coordinate (D30) at (2.8,0.2);

\coordinate (L) at (-4.3,4);

\draw (L) node{$L_2:$}; % first L'

%\draw[blue,thick] (C10b)--(C30b)
%(C20b) node[below]{$r_1$};
\draw[blue,thick] 
(C20b) node[below]{$r_2$};
\draw[blue,thick,->] (C10b)--(C30b);

\draw (D2c0) node[right]{$c$}; % first L'
\draw (D30) node[right]{$c$};
\draw (Dx0) node[left]{$c$};
\draw (D0) node[left]{$c$};
\coordinate (E) at (-2.3,3.3); % r1, first L'
\coordinate (E1a) at (1.3,3.3);
\coordinate (E1b) at (1.7,3);
\coordinate (E1) at (1.7,0);

\coordinate (Ep) at (-2.5,3.6);
\coordinate (Ep-) at (-2.5,3.5);
\coordinate (Epp) at (-2.5,4.4);

\coordinate (Exx) at (-2.9,3.3);
\coordinate (Exxx) at (-2.9,4.4);
\coordinate (Eyy) at (-2.1,3.3);
\coordinate (Eyyy) at (-2.1,4.4);
\coordinate (Ezz) at (-2.5,3.3);
\draw[thick,dashed] (Ezz) circle[radius=4mm]; % dashed r1, L'
\draw[fill,white] (Exx) rectangle (Eyyy);

\draw[thick,dotted,<-] (Ep-)--(Epp); % dotted r1, L'
\draw (Epp) node[above]{$\ell_1$};

\draw[thick,dashed] (Exx)--(Exxx); % dashed r1, L'
\draw[thick,dashed] (Eyy)--(Eyyy); % dashed r1, L'

\coordinate (E2a) at (1.3,4);
\coordinate (E2b) at (2.3,3);
\coordinate (E2) at (2.3,0);

\coordinate (E10b) at (1.7,-.2);
\coordinate (E20b) at (2.3,-.2);
\coordinate (E15b) at (2,-.2);

\draw[thick,red] (E15b) node[below]{$r_1$};
\draw[thick,red,->] 
(E10b)--(E20b) ; % first L'

\draw[->] (-5,0)--(3.5,0); % x-axis
\draw[thick] (A1)..controls (Aa) and (Ab) ..(A2)..controls (A2a) and (A2b)..(A3);
\draw[thick] (B1)..controls (Ba) and (Bb) ..(B2)..controls (B2a) and (B2b)..(B3);
\draw[blue] (C-) node[above]{$\ast$}; % first L' %%%%%%%%%%%%%%
\draw[thick,blue] (C)..controls (C1a) and (C1b)..(C1) 
(C)--(C2)
(C)..controls (C3a) and (C3b)..(C3); % r2, first L'
\draw[thick] (D1)..controls (Da) and (Db) ..(D2) ..controls (D2a) and (D2b)..(D2c) 
(Dxc)..controls (Dx2a) and (Dx2b)..(D3) 
(Dx)..controls (Dxa) and (Dxb)..(Dxc);
\draw[thick,red] 
(E)..controls (E1a) and (E1b)..(E1)
(E)..controls (E2a) and (E2b)..(E2)
(Ep)node[below]{$\ast$}; % r2, first L'
\end{scope} % end first L'
\draw (3.5,1.9) node{$\then$};
\begin{scope}[xshift=13cm] % BEGIN second L', Fig 7

\foreach \x in {12,16,20,32}
\draw[thick] (-1,1.6) circle[radius=\x mm];
\foreach \x in {24,28}
\draw[thick,red] (-1,1.6) circle[radius=\x mm];
\draw[thick] (-5.75,3.1) circle[radius=7.5mm];
\draw[fill,white] (-4.2,1.5) rectangle (2.2,-2)
(-6.5,3)rectangle(-5,2);

\foreach \x in {0.2,.6,1,2.2}
\draw[thick] (\x,0)--(\x,1.5);
\foreach \x in {1.4,1.8}
\draw[thick,red] (\x,0)--(\x,1.5);

\draw (-6.5,.2) node[left]{$c$};

\foreach \x in {-5,-6.5}
\draw[thick] (\x,0)--(\x,3);
\foreach \x in {-3.4,-3.8}
\draw[thick,red] (\x,0)--(\x,1.5);

\coordinate (X) at (-5.75,3);
\coordinate (X1) at (-6,0);
\coordinate (X2) at (-5.5,0);
\coordinate (X1p) at (-6,-.2);
\coordinate (X2p) at (-5.5,-.2);
\coordinate (Xp) at (-5.75,-.2);

\coordinate (Ep) at (-5.75,3.5);
\coordinate (Epp) at (-5.75,4.4);

\draw[thick,red,->] (X1p)--(X2p);
\draw[red] (Xp) node[below]{$r_1$};

\draw[thick,dotted,<-] (Ep)--(Epp); % dotted r1, L', Fig 7
\draw (Epp) node[above]{$\ell_1$};

\foreach \x in {-2.2,-2.6,-3,-4.2}
\draw[thick] (\x,0)--(\x,1.5);
\draw[thick,red] (X1)--(X)--(X2)  (Ep) node[below]{$\ast$};

\coordinate (L) at (-8,4.3);
\draw (L) node{$L_{q(L_2)}:$}; % second L', Fig 7
\coordinate (C10b) at (-1.7,-0.2);
\coordinate (C20b) at (-1,-0.2);
\coordinate (C30b) at (-.3,-0.2);
\draw[blue,thick] 
(C20b) node[below]{$r_2$};
\draw[blue,thick,->] (C10b)--(C30b);
\draw[->] (-7.2,0)--(3,0); % x-axis, second L, Fig 7

\coordinate (C1a) at (-1.2,1.2); % r1, second L', Fig 7
\coordinate (C1b) at (-1.4,1);
\coordinate (C1) at (-1.7,0);
\coordinate (C2) at (-1.1,0);
\coordinate (C3a) at (-.8,1.3);
\coordinate (C3b) at (-.5,1);
\coordinate (C3) at (-.3,0);

\coordinate (C) at (-1,1.5); % r1, second L', Fig 7
\coordinate (Cp) at (-1,1.9);
\coordinate (Cpp) at (-1,5);
\draw[thick,dotted,<-] (Cp)--(Cpp); % dotted r1, L', Fig 7
\draw (Cpp) node[above]{$\ell_2$};
\draw[blue] (C) node[above]{$\ast$}; % second L', Fig 7
\draw[thick,blue] (C)..controls (C1a) and (C1b)..(C1) 
(C)--(C2)
(C)..controls (C3a) and (C3b)..(C3); % r1, second L', Fig 7
\end{scope} % end second L', Fig 7
%
%\coordinate (A) at (0,0);
%\draw[thick,color=blue] (A) ellipse [x radius=2.8cm,y radius=2.1cm];
\end{tikzpicture} %, Fig 7
\caption{Dotted lines $\ell_1,\ell_2$ are given by definition of $q(L_2)$. Take dashed lines parallel to $\ell_1,\ell_2$ and connected with small semicircles below vertices $v_1,v_2$. Push the dashed line down to the $x$-axis.
This gives an admissible deformation of $L_2$ (on the left) to $L_{q(L_2)}$ (on the night). The dotted lines $\ell_1,\ell_2$ cross the same edges in both partial pictures.
}
\label{Fig: proof of Claim 3}
\end{center}
\end{figure}

\subsection{Pictures with good commutator relations}\label{ss3.2: good commutator relations}

If the same letter, say $x$, occurs more than twice in a relation $r$, then, at the vertex $v$, the edge set $E(x)$ cannot be a manifold. (For example, if $G=\left<x\,|\, x^3\right>$ then $E(x)$ will not be a manifold.) However, this does not happen in our case because our relations are ``good''. 

We define a \und{good commutator relation} to be a relation of the form
\[
	r(a,b):=ab(bc_1,\cdots,c_ka)^{-1}
\]
where $a,b,c_1,\cdots,c_k$ are distinct elements of $\cX$ and $k\ge0$. The letters $a,b$ will be called \und{X-letters} and the letters $c_j$ will be called \und{Y-letters} in the relation. In the picture, the two X-letters in any commutator relation form the shape of the letter ``X'' since the lines labeled with these letters go all the way through the vertex. Call theses \und{X-edges} at the vertex. The edges labeled with the Y-letters go only half way and stop at the vertex. Call these \und{Y-edges} at the vertex. (See Figure \ref{fig02}.) In the definition of a picture we can choose the sets $E_r\subset S^1$ so that the points labeled $a,a^{-1}$ (and $b,b^{-1}$) are negatives of each other. Then the edge sets $E(a),E(b)$ will be manifolds. (Since $a,b,c_j$ are all distinct there are no other coincidences of labels at the vertices.)

\begin{figure}[htbp]
\begin{center}
{
\setlength{\unitlength}{1cm}
{\mbox{
\begin{picture}(10,2)
      \thicklines
%    \thinlines
\put(-1,1.6){Example:}
\put(-1.7,1){$r=r({a,b})=ab(bc_1c_2c_3a)^{-1}$}
\put(3,0){
\put(1,0){
      \qbezier(0,0)(1,1.2)(2.2,2)
      \put(1.8,2){$b$} % negative side
      \qbezier(-.2,2)(1,1.2)(2,0)
      \put(2,0.1){$a$} % negative side right
      }
 \put(1,2){$a$} % negative side left
 \put(.9,0.2){$b$}  % negative side
    \put(1.9,.97){$\bullet$}
    \put(2,1.07){
    \qbezier(0,0)(1,-.2)(2,-.9)\put(1.8,-.7){$c_3$}
    \qbezier(0,0)(1,0.2)(2.4,-.2) \put(2.2,0){$c_2$}
    \qbezier(0,0)(1,.4)(2.4,.5)\put(2.2,.7){$c_1$}
    }
    \put(1.9,1.6){$\ast$}
}
\end{picture}}
}}
\caption{The X-letters $a,b$ have edge sets which are smooth at the vertex. The basepoint direction is on the negative side of both X-edges $E(a),E(b)$.}
\label{fig02}
\end{center}
\end{figure}

We have the following trivial observation.

\begin{prop} Suppose that $G=\left<\cX\,|\,\cY\right>$ is a group having only good commutator relations. Then, given any label $x$, the edge set $E(x)$ in $L$ is a disjoint union of smooth simple closed curves and smooth paths. At both endpoints of each path, $x$ occurs as a Y-letter. It occurs as $x$ at one end and $x^{-1}$ at the other.
\end{prop}

\begin{cor}\label{cor: pictures have same number of x, x-inv}
Suppose that $G$ has only good commutator relations. Then, for any picture $L$ for $G$ and any label $x$, the number of vertices of $L$ having $x$ as Y-letter is equal to the number of vertices of $L$ having $x^{-1}$ as Y-letter.
\end{cor}

\subsection{Atoms}\label{ss3.3: atoms}

Let $\cS=(\beta_1,\cdots,\beta_m)$ be an admissible sequence of real Schur roots for a hereditary algebra $\Lambda$. Then $G(\cS)$ has only good commutator relations. We need the Atomic Deformation Theorem which says that every picture in $G(\cS)$ is a linear combination of ``atoms''. In other words atoms generate $P(G)$. The definition comes from \cite{IOTW4} and \cite{IT13} but is based on \cite{IOr} where similar generators of $P(G)$ are constructed for a torsion-free nilpotent group $G$.

Suppose that $\cS$ is admissible and $\alpha_\ast=(\alpha_1,\alpha_2,\alpha_3)$ is a sequence of three hom-orthogonal roots in $\cS$ ordered in such a way that $ext(\alpha_i,\alpha_j)=0$ for $i<j$. Let $\cA(\alpha_\ast)$ be the rank 3 wide subcategory of $mod\text-\Lambda$ with simple objects $\alpha_\ast$. One easy way to describe this category is
\[
	\cA(\alpha_\ast)=(^\perp M_{\alpha_\ast})^\perp
\]
where $M_{\alpha_\ast}=M_{\alpha_1}\oplus M_{\alpha_2}\oplus M_{\alpha_3}$. In other words, $\cA(\alpha_\ast)$ is the full subcategory of $mod\text-\Lambda$ of all modules $X$ having the property that $\Hom(X,Y)=0=\Ext(X,Y)$ for all $Y$ having the property that $\Hom(M_{\alpha_\ast},Y)=0=\Ext(M_{\alpha_\ast},Y)$. The objects of $\cA(\alpha_\ast)$ are modules $M$ having filtrations where the subquotients are $M_{\alpha_i}$. Since $ext(\alpha_i,\alpha_j)=0$ for $i<j$, the modules $M_{\alpha_1}$ occur at the bottom of the filtration and $M_{\alpha_3}$ occurs at the top of the filtration. Let $wide(\alpha_\ast)$ denote the set of all dimension vectors of the objects of $\cA(\alpha_\ast)$. The elements of $wide(\alpha_\ast)$ are all nonnegative integer linear combinations of the roots $\alpha_i$. These are elements of the 3-dimensional vector space $\RR\alpha_\ast$ spanned by the roots $\alpha_\ast$.

 Let $L(\alpha_\ast)\subseteq S^2$ be the semi-invariant picture for the category $\cA(\alpha_\ast)$. We recall (\cite{IOTW4}, \cite{IT13}, \cite{Cat0}) that $L(\alpha_\ast)$ is the intersection with the unit sphere $S^2\subseteq\RR\alpha_\ast\cong \RR^3$ with the union of the 2-dimensional subset $D(\beta)$ of $\RR\alpha_\ast$ where $\beta\in wide(\alpha_\ast)$ given by the stability conditions:
 \[
 	D(\beta):=\left\{x\in\RR\alpha_\ast: \brk{x,\beta}=0, \brk{x,\beta'}\ge 0\text{ for all $\beta'\subset \beta$, $\beta'\in wide(\alpha_\ast)$}\right\}
 \]
When we stereographically project $L(\alpha_\ast)\subset S^2$ into the plane $\RR^2$ we get a planar picture for the group $G(wide(\alpha_\ast))$ according to the definitions in this section.

\begin{defn}
Let $\cS,\alpha_\ast$ be as above. Then the \und{atom} $A_\cS(\alpha_\ast)\subset \RR^2$ is defined to be the picture for $G(\cS)$ given by taking the semi-invariant picture $L(\alpha_\ast)\subset S^2$, stereographically projecting it away from the point $-\sum \undim P_i\in\RR\alpha_\ast$ where $P_i$ are the projective objects of $\cA(\alpha_\ast)$ and deleting all edges having labels $x(\gamma)$ where $\gamma\notin\cS$.
\end{defn}

\begin{center}
\begin{figure}[ht] % begin {fig03}
       \includegraphics[width=4in]{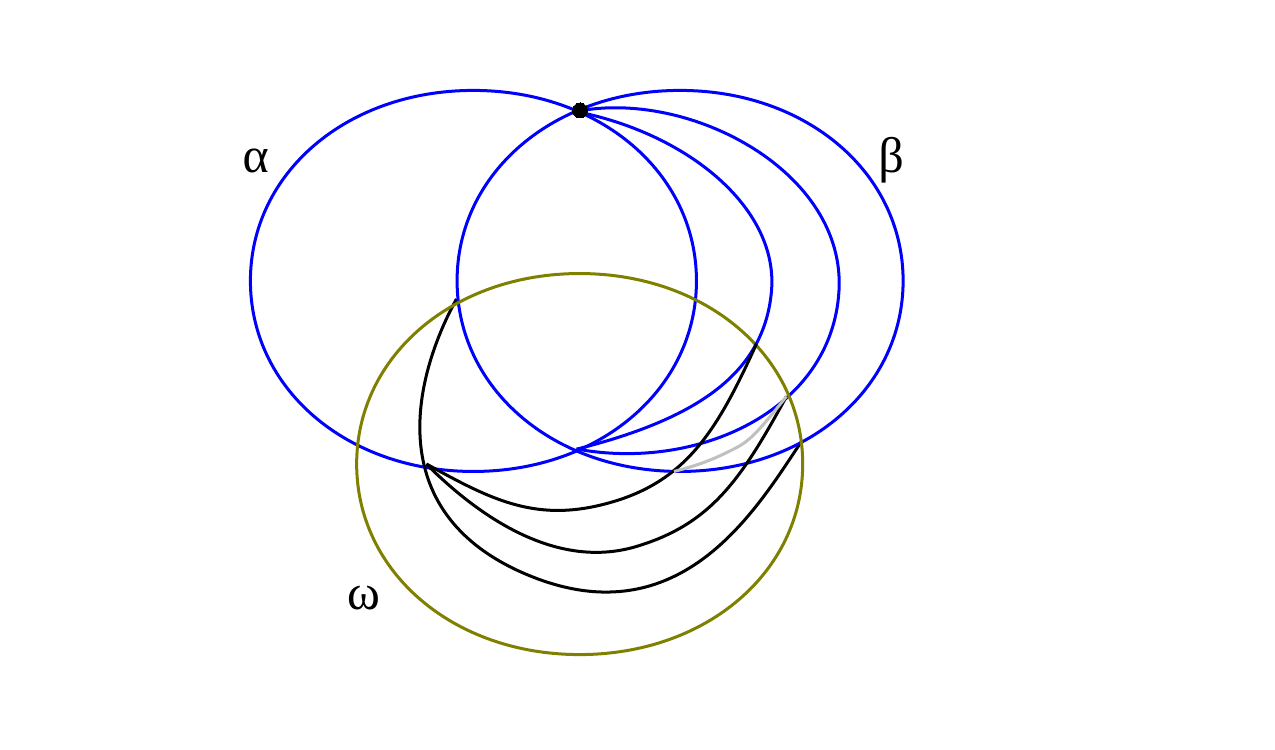}
\caption{The atom $A_\cA(\alpha,\beta,\omega)$. There are three circles labeled $\alpha, \beta, \omega$. There is only one vertex (black dot) outside the brown circle labeled $\omega$. There is only one vertex inside the $\alpha$ circle. The faint gray line is deleted since, in this example, its label is not in the set $\cS$.}
\label{fig03}
\end{figure} % end {fig03}
\end{center}

Figure \ref{fig03} gives an example of an atom. We need to prove that certain aspects of the shape are universal.

\begin{prop}\label{prop: properties of atoms}
Any atom $A_\cS(\alpha_1,\alpha_2,\alpha_3)$ has three circles $E(\alpha_i)=D(\alpha_i)$ with labels $x(\alpha_i)\in G$ and all other edge sets have two endpoints. There is exactly one vertex $v$ outside the $\alpha_3$ circle. This vertex has the relation $r(\alpha_1,\alpha_2)$. Dually, there is exactly one vertex inside the $\alpha_1$ circle with relation $r(\alpha_2,\alpha_3)^{-1}$.
\end{prop}

We use the notation $r(\alpha,\beta)$ for $r(x(\alpha),x(\beta))$
For example, the blue lines in Figure \ref{fig02} meet at two vertices giving the relations 
\[
r(\alpha,\beta)=x(\alpha)x(\beta)\left(x(\beta)x(\gamma_1) x(\gamma_2)x(\alpha)\right)^{-1}
\]
at the top and $r(\alpha,\beta)^{-1}$ in the middle of the brown $x(\omega)$ circle.

\begin{proof}
The only objects of $\cA(\alpha_\ast)$ which do not map onto $M_{\alpha_3}$ are the objects of $\cA(\alpha_1,\alpha_2)$ which are the objects $M_{\alpha_1},M_{\alpha_2}$ and their extensions $M_{\gamma_j}$. These give the terms in the commutator relation $r(\alpha_1,\alpha_2)$ and these lines meet at only two vertices in the atom. All other edges of the atom have at least one abutting edge with a label $\gamma$ where $\gamma\onto \alpha_3$. By the stability condition defining $D(\gamma)$, these points must be inside or on the $\alpha_3$ circle as claimed.
\end{proof}

\subsection{Sliding Lemma and Atomic Deformation Theorem}\label{ss3.4: sliding lemma and Atomic Deformation theorem}

We will prove the Sliding Lemma \ref{lem: sliding lemma} and derive some consequences such as the Atomic Deformation Theorem \ref{Atomic Deformation theorem} which says that every picture for $G(\cS)$ is a linear combination of atoms. First, some terminology. We say that $L'$ is an \und{atomic deformation} of $L$ if $L'$ is a deformation of $L$ plus a linear combination of atoms. Thus the Atomic Deformation Theorem states that every picture has an atomic deformation to the empty picture.

Suppose that $\cS$ is an admissible set of roots with a fixed lateral ordering and let $\omega\in \cS$. Recall that $\cS_-(\omega)$ is the set of all $\beta\le \omega$ in lateral order in $\cS$. In particular, either $\beta=\omega$ or $hom(\omega,\beta)=0$ and $ext(\beta,\omega)=0$. Also, $\cR_-(\omega)$ is the set of all $\beta\in\cS_-(\omega)$ which are hom-orthogonal to $\omega$. Since these are relatively closed subsets of $\cS$, the picture groups $G(\cS_-(\omega))$ and $G(\cR_-(\omega))$ are defined. (See Remark \ref{rem: G(R) is a functor}.) 
 \[\cS_-(\omega):=\{ \beta\in\cS\,:\, \beta\le \omega\text{ in lateral order }\}\]
\[\cR_-(\omega):=\{\beta\in\cS_-(\omega)\,:\, hom(\beta,\omega)=0\}\]for any $\omega\in\cS$.

\begin{lem}[Monomorphism Lemma]\label{monomorphism lemma}
The homomorphism $G(\cR_-(\omega))\to G(\cS_-(\omega))$ induced by the inclusion $\cR_-(\omega)\into \cS_-(\omega)$ has a retraction $\rho$ given on generators by 
\[
	\rho(x(\beta))=\begin{cases} x(\beta) & \text{if } \beta\in \cR_-(\omega)\\
   1 & \text{otherwise}
    \end{cases}
\]
Furthermore, $\rho$ takes pictures and partial pictures $L$ for $G(\cS_-(\omega))$ and gives a picture or partial picture $\rho(L)$ for $G(\cR_-(\omega))$ by simply deleting all edges with labels $x(\beta)$ where $\beta\notin \cR_-(\omega)$.
\end{lem}

Figure \ref{Fig: Atomic Deformation Theorem} gives an example of how this lemma is used. The proof is analogous to the proof of the dual statement which goes as follows. Recall that, for any $\alpha$ in an admissible set of roots $\cS$, $\cS_+(\alpha)$ is the set of all $\beta\ge \alpha$ in $\cS$ and $\cR_+(\alpha)$ is the set of all $\beta\in \cS_+(\alpha)$ which are hom-orthogonal to $\alpha$. As in the case of $\cS_-(\omega),\cR_-(\omega)$ these are relatively closed subsets of $\cS$.

\begin{lem}%[Monomorphism Lemma]
\label{dual monomorphism lemma}
The homomorphism $G(\cR_+(\alpha))\to G(\cS_+(\alpha))$ induced by the inclusion $\cR_+(\alpha)\into \cS_+(\alpha)$ has a retraction $\rho$ given on generators by 
\[
	\rho(x(\beta))=\begin{cases} x(\beta) & \text{if } \beta\in \cR_+(\alpha)\\
   1 & \text{otherwise}
    \end{cases}
\]
Furthermore, $\rho$ takes pictures and partial pictures $L$ for $G(\cS_+(\alpha))$ and gives a picture or partial picture $\rho(L)$ for $G(\cR_+(\alpha))$ by simply deleting all edges with labels $x(\beta)$ where $\beta\notin \cR_+(\alpha)$.
\end{lem}

\begin{proof} The key is that $\cR_+(\alpha)$ is given by a linear condition. Since $ext(\alpha,\beta)=0$ for all $ \beta\in \cS_+(\alpha)$ (and $hom(\beta,\alpha)=0$ for all $\beta\neq \alpha$ in $\cS_+(\alpha)$) we have:
\[
	\cR_+(\alpha)=\{\beta\in\cS_+(\alpha)\,:\, \brk{g(\alpha),\beta}=hom(\alpha,\beta)-ext(\alpha,\beta)=0\}.
\]
Since any two letters in any relation in are linearly independent, if two letters in any relation in $G(\cS_+(\alpha))$ lie in $\cR_+(\alpha)$ then all the letters in the relation lie in $\cR_+(\alpha)$. Thus, if only part of the relation survives under the retraction it must be a single letter. This letter, say $\gamma$, cannot be a Y-letter: If it were and $\gamma_1,\gamma_2$ are the X-letters in that relation then $hom(\alpha,\gamma_1)$ and $hom(\alpha,\gamma_2)$ would both be nonzero. Since one of these is a subroot of $\gamma$, this would also make $hom(\alpha,\gamma)\neq0$ and $\gamma\notin \cR_+(\alpha)$. So, none of the letters in such a relation will lie in $\cR_+(\alpha)$. Therefore, the retraction $\cS_+(\alpha)\to \cR_+(\alpha)$ sends relations to relations and induces a retraction of groups $\rho:G(\cS_+(\alpha))\to G(\cR_+(\alpha))$.

Given any picture or partial picture $L$ for $G(\cS_+(\alpha))$, each vertex has a relation $r$ which has the property that either $\rho(r)=r$ or $\rho(r)$ is an unreduced relation of the form $xx^{-1}$ or $\rho(r)$ is empty. In the second case $\rho(r)=xx^{-1}$ we consider the vertex as part of the smooth curve $E(x)$. Removal of all edges with labels not in $\cR_+(\alpha)$ therefore keeps $L$ looking locally like a picture for $\cR_+(\alpha)$. But pictures and partial pictures are defined by local conditions.
\end{proof}

Using the Monomorphism Lemma \ref{monomorphism lemma}, we can now state and prove the key lemma about pictures for $G(\cS)$. Recall that $E(\omega)$ is the union of the set of edges with label $x(\omega)$ and that, for any root $\beta\in\cR_-(\omega)$, any vertex with relation $r(\beta,\omega)$ or $r(\beta,\omega)^{-1}$ has Y-edges on the positive side of the X-line $E(\omega)$. (For example, in Figure \ref{fig03}, $\alpha,\beta$ and all letters $\gamma_i$ in $r(\alpha,\beta)$ lie in $\cR_-(\omega)$. So the edges corresponding to the commutator relations $r(\gamma_i,\omega)$ for all letters $\gamma_i$ in $r(\alpha,\beta)$ lie in the interior of the brown circle $E(\omega)=D(\omega)$. Since Figure \ref{fig03} is an atom, the edges are curved in the positive direction.) We also note that the base point direction is on the negative side of both X-lines at each crossing.

 \begin{lem}[Sliding Lemma]\label{lem: sliding lemma}
Suppose that $L$ is a picture for $G(\cS)$ so that $E(\omega)$ is a disjoint union of simple closed curves. Let $U$ be one of the components of the complement of $E(\omega)$ and let $\Sigma=\overline U\cap E(\omega)$ be the boundary of the closure $\overline U$ of $U$. Suppose that $U$ is on the negative side of $\Sigma$ and that all edges in $L\cap U$ have labels $x(\beta)$ for $\beta\in\cR_-(\omega)$. Then there is an atomic deformation $L\sim L'$ which alters $L$ only in an arbitrarily small neighborhood $V$ of $\overline U$ so that $L'\cap V$ contains no edges with labels $\ge \omega$ in lateral order.
\end{lem}

\begin{proof} By assumption, every edge which crosses $\Sigma$ has a label $x(\beta)$ where $\beta\in\cR_-(\omega)$. This implies that all Y-edges at all vertices on $\Sigma$ lie outside the region $U$. So, at each vertex of $\Sigma$, only one edge $E(\beta)$ goes into the region $U$. Also, all basepoint directions of all vertices on $\Sigma$ lie inside $U$.

The proof of the lemma is by induction on the number of vertices in the region $V$ containing $\Sigma$. Suppose first that this number is zero. Then $\Sigma$ has no vertices and $L\cap U$ is a union of disjoint simple closed curves which can be eliminated by concordance one at a times starting with the innermost simple closed curve. This includes $\Sigma$. The result has no edges with labels $\ge\omega$.

Suppose next that $L$ has vertices on the set $\Sigma$ but no vertices in the region $U$ enclosed by $\Sigma$. Then every edge of $L$ in $U$ is an arc connecting two vertices on $\Sigma$ and the negative side of each arc has a path connecting the two basepoint directions at these two vertices. So, we can cancel all pair of vertices and we will be left with no vertices in $V$. As before, we can then eliminate all closed curves in $V$ including $\Sigma$ which has now become a union of simple closed curves.

Finally, suppose that $U$ contains a vectex $v$ having relation $r(\alpha,\beta)^\pm$. So, $v$ contributes $\pm g\brk{r(\alpha,\beta)}$ to the algebraic expression for $L$. Then $\alpha,\beta\in\cR_-(\omega)$ by assumption. Now add the atom $\mp A(\alpha,\beta,\omega)$ (which resembles Figure \ref{fig03}) in the region containing the basepoint direction of $v$. (See the left side of Figure \ref{Fig: sliding lemma}.) This adds $\mp gA(\alpha,\beta,\omega)$ to the algebraic expression for $L$. The atom has a circle labeled $x(\omega)$ oriented inward with exactly one vertex outside this circle with relation $r(\alpha,\beta)^\mp$ (the mirror image of the relation at $v$) by Proposition \ref{prop: properties of atoms}. The new vertex cancels the vertex $v$. (See the right side of Figure \ref{Fig: sliding lemma}.) Repeating this process eliminates all vertices in the new region $U'$. 

After that, all edges in $\overline U'$, the closure of $U'$ can be eliminated. This eliminates all edges with label $\omega$ from $V$. However, it also introduces new edge sets (the interior of the $\omega$ oval in the atom). However, these all have labels $<\omega$. So, we are done.

Since the entire process was a sequence of picture deformations and addition of $\ZZ G(\cS)$ multiples of atoms, it is an atomic deformation.
\end{proof}

\begin{figure}[htbp]
\begin{center}
\begin{tikzpicture}[scale=.5]
\begin{scope}[xshift=-8cm] % Left drawing
\coordinate (S1) at (-4.8,4.3);
\coordinate (S2) at (-4.8,1.8);
\coordinate (S3) at (0,3.1);
\coordinate (V) at (0,3.6);
\coordinate (S4) at (0,1.5);
\coordinate (S5) at (4.8,4.3);
\coordinate (S6) at (4.8,1.8);
\coordinate (S7) at (4.8,3.1);
\foreach \x in {S1,S2,S3,S4,S5,S6,S7}
\draw (\x) node{$\ast$};
\draw (V) node{$v$};
\coordinate (A) at (-1,0);
\coordinate (B) at (1,0);
\coordinate (C) at (0,-1.5);
\coordinate (A1) at (-3,1);
\coordinate (B1) at (3,1);
\coordinate (C1) at (-1.8,-2.8);
\coordinate (G) at (2.2,0);
\coordinate (X) at (-1,-.2);
\coordinate (Y) at (2,-1.3);
\coordinate (C2) at (-4.7,-2.8);
\coordinate (C3) at (4.7,-2.8);
\coordinate (D2) at (-4.7,0);
\coordinate (D3) at (4.7,0);
\coordinate (A2) at (-3,2.2);
\coordinate (B2) at (3,2.1);
\coordinate (G1) at (4,3);
\coordinate (U) at (3.5,-2);
\draw[color=blue] (U) node{$U'$};
\draw (A1) node{$\alpha$};
\draw (A2) node{$\alpha$};
\draw (B1) node{$\beta$};
\draw (B2) node{$\beta$};
\draw (G) node{$\gamma$};
\draw (G1) node{$\gamma$};
\draw[thick] (-5,2)--(-6.5,1);
\draw[thick] (5,2)--(6.5,1);
\draw[thick] (-5,4.6)--(-6.5,3.6);
\draw[thick] (5,4.6)--(6.5,3.6);
\draw[thick] (5,3.3)--(6.5,2.5);
\draw[thick] (-6.5,1.6)--(6.5,5);
\draw[thick] (6.5,1.6)--(-6.5,5);
\draw[thick] (0,3.3)--(6.5,3.3);
\draw[very thick, color=blue] (-5,5)--(-5,-3.5);
\draw[very thick, color=blue] (5,5)--(5,-3.5);
\begin{scope}
\clip (0,-2) rectangle (2.2,2);
\draw[ thick] (0,0) ellipse [x radius=2cm, y radius=1.3cm];
\end{scope}
\begin{scope}
\clip (C) ellipse [x radius=2cm, y radius =1.5cm];
\draw[thick] (X) ..controls (-2,-2) and (0,-3)..(Y);
\clip (-1.2,-2) rectangle (2,0);
\draw[thick] (0,-.7) ellipse[x radius=1.69cm, y radius=1.1cm];
\end{scope}
\draw[ thick] (A) ellipse [x radius=2cm, y radius =1.5cm];
\draw[ thick] (B) ellipse [x radius=2cm, y radius =1.5cm];
\draw[very thick, color=blue] (C) ellipse [x radius=2cm, y radius =1.5cm] (C1) node{$\omega$} (D2) node{$\omega$} (D3) node{$\omega$} (C2) node{$\Sigma$}(C3) node{$\Sigma$};
\end{scope} % end of Left drawing
\draw[ thick,->] (-.5,0)--(.5,0);
\begin{scope}[xshift=8cm] % Right drawing
\coordinate (S1) at (-4.8,4.3);
\coordinate (S2) at (-4.8,1.8);
\coordinate (S3) at (0,3.1);
\coordinate (V) at (0,3.6);
\coordinate (S4) at (0,1.5);
\coordinate (S5) at (4.8,4.3);
\coordinate (S6) at (4.8,1.8);
\coordinate (S7) at (4.8,3.1);
\coordinate (U) at (3.5,-2);
\draw[color=blue] (U) node{$U'$};
\foreach \x in {S1,S2,S3,S4,S5,S6,S7}
\draw (\x) node{$\ast$};
\draw (V) node{$v$};
\coordinate (A) at (-1,0);
\coordinate (B) at (1,0);
\coordinate (C) at (0,-1.5);
\coordinate (A1) at (-3,1);
\coordinate (B1) at (3,1);
\coordinate (C1) at (-1.8,-2.8);
\coordinate (G) at (2.2,0);
\coordinate (X) at (-1,-.2);
\coordinate (Y) at (2,-1.3);
\coordinate (C2) at (-4.7,-2.8);
\coordinate (C3) at (4.7,-2.8);
\coordinate (D2) at (-4.7,0);
\coordinate (D3) at (4.7,0);
\coordinate (A2) at (-3,2.2);
\coordinate (B2) at (3,2.1);
\coordinate (G1) at (4,3);
\draw (A1) node{$\alpha$};
\draw (A2) node{$\alpha$};
\draw (B1) node{$\beta$};
\draw (B2) node{$\beta$};
\draw (G) node{$\gamma$};
\draw (G1) node{$\gamma$};
\draw[thick] (-5,2)--(-6.5,1);
\draw[thick] (5,2)--(6.5,1);
\draw[thick] (-5,4.6)--(-6.5,3.6);
\draw[thick] (5,4.6)--(6.5,3.6);
\draw[thick] (5,3.3)--(6.5,2.5);
\draw[thick] (-6.5,1.6)--(6.5,5);
\draw[thick] (6.5,1.6)--(-6.5,5);
\draw[thick] (0,3.3)--(6.5,3.3);
\draw[very thick, color=blue] (-5,5)--(-5,-3.5);
\draw[very thick, color=blue] (5,5)--(5,-3.5);
\begin{scope}
\clip (0,-2) rectangle (2.2,2);
\draw[ thick] (0,0) ellipse [x radius=2cm, y radius=1.3cm];
\end{scope}
\begin{scope}
\clip (C) ellipse [x radius=2cm, y radius =1.5cm];
\draw[thick] (X) ..controls (-2,-2) and (0,-3)..(Y);
\clip (-1.2,-2) rectangle (2,0);
\draw[thick] (0,-.7) ellipse[x radius=1.69cm, y radius=1.1cm];
\end{scope}
\draw[ thick] (A) ellipse [x radius=2cm, y radius =1.5cm];
\draw[ thick] (B) ellipse [x radius=2cm, y radius =1.5cm];
\draw[very thick, color=blue] (C) ellipse [x radius=2cm, y radius =1.5cm] (C1) node{$\omega$} (D2) node{$\omega$} (D3) node{$\omega$} (C2) node{$\Sigma$}(C3) node{$\Sigma$};
\draw[fill,color=white] (-1,.61) rectangle (1,4);
\clip (-1.02,.6) rectangle (1.02,4);
\draw[thick] (1,2.27) ellipse [x radius=3mm, y radius=7.7mm];
\draw[thick] (-1,2.27) ellipse [x radius=3mm, y radius=7.7mm];
\draw[thick] (1.03,2.21) ellipse [x radius=5mm, y radius=11mm];
\draw[thick] (1.03,2.05) ellipse [x radius=7mm, y radius=15.1mm];
\draw[thick] (-1.03,2.05) ellipse [x radius=7mm, y radius=15.1mm];
\end{scope} % end of Right drawing
\end{tikzpicture}
\caption{Illustrating proof of Sliding Lemma \ref{lem: sliding lemma}: $\Sigma$ in blue is a disjoint union of $E(\omega)$ closed curves which encloses a region $\overline U=\Sigma\cup U$. All Y-edges for vertices on $\Sigma$ lie outside $U$. The atom $\cA(\alpha,\beta,\omega)$ in the proof has already been added on the left. The new region $U'$ is the complement of the new $\omega$ oval in $U$. The vertex $v$ has been cancelled with the vertex in the atom on the right.}
\label{Fig: sliding lemma}
\end{center}
\end{figure}

\begin{thm}[Atomic Deformation Theorem]\label{Atomic Deformation theorem}
Suppose that $\cS$ is an admissible set of real Schur roots. Then any picture for $G(\cS)$ has a null atomic deformation. I.e., it is deformation equivalent to a $\ZZ G$ linear combination of atoms. Equivalently, the $\ZZ G(\cS)$-module $P(G(\cS))$ is generated by atoms.
\end{thm}

This theorem follows from the Sliding Lemma and we will see that it implies Lemma \ref{lem C: beta m occurs a fixed number of times}.

\begin{proof} 

Let $\cS=(\beta_1,\cdots,\beta_m)$ be an admissible set of roots. Let $\beta^1,\cdots,\beta^m$ be the same set rearranged in lateral order. Let $\cR^k$ be the set of all elements of $\cS$ which are $\le \beta^k$ in lateral order. Thus, $\cR^k=\cS_-(\beta^k)$. Take $k$ minimal so that the labels which occurs in $L$ all lie in $\cR^k$. If $k=1$ then $L$ has no vertices and is a disjoint union of simple closed curves which are null homotopic. By induction, it suffices to eliminate $\omega=\beta^k$ as a label from the picture $L$ by picture deformations and addition of atoms without introducing labels $\beta^j$ for $j>k$.

Since $\omega$ is a rightmost element in the set $\cR^k$, $x(\omega)$ does not occur as a Y-letter at any vertex of $L$. Therefore the edge set $E(\omega)$ is a disjoint union of simple closed curves. Let $\Sigma$ be innermost such curve and $\Sigma'$ be a curve parallel to $\Sigma$ on the negative side. (See Figure \ref{Fig: Atomic Deformation Theorem}.) Then $\Sigma'$ crosses on those edges $E(\beta)$ where $\beta\in \cS_-(\omega)$ are hom-orthogonal to $\omega$. In other words, $\beta\in \cR_-(\omega)$. 

Let $L_0'$ be the mirror image of $L_0$ through $\Sigma'$. Then $L_0\cup L_0'$ is null deformable, i.e., $L_0+L_0'=0$ in the group of partial pictures $Q(\cR^k)$. Since $\Sigma'$ meets only edges with labels in $\cR_-(\omega)$, we can apply the retraction $\rho$ from the Monomorphism Lemma \ref{monomorphism lemma} to just one side of $\Sigma'$ and still have a well-defined picture. This construction gives us two pictures: $L'=\rho(L_0)\cup L_1$ and $L''=L_0\cup \rho(L_0')$.

\underline{Claim} $L$ is deformation equivalent to $L'\coprod L''$, i.e., $L=L'+L''$ in the group $\cP(\cR^k)$.

Pf: The group of pictures $\cP(\cR^k)$ is a a subgroup of the group of partial pictures $\cQ(\cR^k)$ and in that group we have:
\[
	L=L_0+L_1=L_0+\rho(L_0')+\rho(L_0)+L_1=L''+L'
\]
since $\rho(L_0')+\rho(L_0)=\rho(L_0+L_0')=\rho(0)=0$.

The simple closed curve $\Sigma$ lies either in $L'$ or $L''$. If $\Sigma\subset L'$ then $\Sigma$ can be removed by $L'$ by an atomic deformation by the Sliding Lemma \ref{lem: sliding lemma} since the edges inside $\Sigma$ are in $\cR_-(\omega)$, bing in $\rho(L_0)$. If $\Sigma\subset L''$ (as drawn in Figure \ref{Fig: Atomic Deformation Theorem}), the region outside $\Sigma$ has all labels in $\cR_-(\omega)$. So, it can be removed by Lemma \ref{lem: sliding lemma}. In both cases, the number of $E(\omega)$ components in $L'\coprod L''$ (the same as the number of components in $L$) has been reduced by one by an atomic deformation without introducing any new labels $\ge\omega$. By induction on the number of components of $E(\omega)$, this set can be removed and $k$ can be reduced by one. So, by induction on $k$, we are done. The entire picture can be deformed into nothing by atomic deformation.
\end{proof}

\begin{figure}[htbp]
\begin{center}
\begin{tikzpicture}[scale=.5]
\begin{scope}[xshift=-5cm] % Left drawing
\draw[very thick, color=blue](0,0) ellipse[x radius=2cm, y radius= 4cm];
\draw (1.2,0) node{$L_0$};
\draw (2.8,4) node{$L_1$};
\draw[color=red] (0,0) ellipse[x radius=2.3cm, y radius= 4.3cm];
\draw[color=red] (2.3,0) node[right]{$\Sigma'$};
\draw[color=blue] (-2,0) node[right]{$\Sigma$};
\clip (0,0) ellipse[x radius=2.6cm, y radius= 4.6cm];
\draw[ thick](2,1) ellipse[x radius=2cm, y radius= 3cm];
\draw[ thick](2,-1) ellipse[x radius=2cm, y radius= 3cm];
\clip (0,0) ellipse[x radius=2cm, y radius= 4cm];
\draw[ thick](1,.7) ellipse[x radius=2cm, y radius= 2.85cm];
\draw[ thick](1,-.7) ellipse[x radius=2cm, y radius= 2.85cm];
\clip (1,.7) ellipse[x radius=2cm, y radius= 2.85cm];
\clip (1,-.7) ellipse[x radius=2cm, y radius= 2.85cm];
\draw[thick] (3.1,0) circle[radius=3.5cm];
\end{scope} % end of Left drawing
\draw[ thick,->] (-.5,0)--(.5,0);
\begin{scope}[xshift=5cm] % Middle drawing
\draw (1.2,0) node{$L_0$};
\draw (4.5,3.5) node{$\rho(L_0')$};
\draw[color=red] (0,0) ellipse[x radius=2.3cm, y radius= 4.3cm];
\draw[color=red] (2.3,0) node[right]{$\Sigma'$};
\draw[very thick, color=blue](0,0) ellipse[x radius=2cm, y radius= 4cm];
\draw[ thick](2,1) ellipse[x radius=2cm, y radius= 3cm];
\draw[ thick](2,-1) ellipse[x radius=2cm, y radius= 3cm];
\clip (0,0) ellipse[x radius=2cm, y radius= 4cm];
\draw[ thick](1,.7) ellipse[x radius=2cm, y radius= 2.85cm];
\draw[ thick](1,-.7) ellipse[x radius=2cm, y radius= 2.85cm];
\clip (1,.7) ellipse[x radius=2cm, y radius= 2.85cm];
\clip (1,-.7) ellipse[x radius=2cm, y radius= 2.85cm];
\draw[thick] (3.1,0) circle[radius=3.5cm];
\end{scope} % end of Middle drawing
\begin{scope}[xshift=14cm] % Right drawing
\draw[color=red] (0,0) ellipse[x radius=2.3cm, y radius= 4.3cm];
\draw[color=red] (2.3,0) node[right]{$\Sigma'$};
\draw (-1,0) node{$\rho(L_0)$};
\draw (2.8,4) node{$L_1$};
\clip (0,0) ellipse[x radius=2.6cm, y radius= 4.6cm];
\draw[ thick](2,1) ellipse[x radius=2cm, y radius= 3cm];
\draw[ thick](2,-1) ellipse[x radius=2cm, y radius= 3cm];
\end{scope} % end of right drawing
\end{tikzpicture}
\caption{Illustrating proof of Atomic Deformation Theorem \ref{Atomic Deformation theorem}: $\Sigma'$ (in red) is on the negative side of an innermost $E(\omega)$ curve $\Sigma$ (in blue). The picture $L=L_0\cup L_1$, on the left, is deformation equivalent to the disjoint union of two pictures: $L''=L_0\cup \rho(L_0')$, in the middle, and $L'=\rho(L_0)\cup L_1$ on the right. The $E(\omega)$ component $\Sigma$ lies either in $L'$ or $L''$. (Here it is in $L''$ in the middle.) In either case, it can be removed by the Sliding Lemma \ref{lem: sliding lemma}.}
\label{Fig: Atomic Deformation Theorem}
\end{center}
\end{figure}

\subsection{Proofs of Lemmas \ref{lem C: beta m occurs a fixed number of times}, \ref{lem E: beta m only commutes with hom orthogonal roots}}\label{ss3.5: proof of lemmas C,E}

The proofs of Lemmas \ref{lem C: beta m occurs a fixed number of times} and \ref{lem E: beta m only commutes with hom orthogonal roots} are very similar.

\begin{proof}[Proof of Lemma \ref{lem C: beta m occurs a fixed number of times}]
Suppose that $w,w'$ are expressions for the same element of $G(\cS)$ and $\pi(w)$, $\pi(w')$ are equal as words in the generators of $G(\cS_0)$. This means that $\pi(w^{-1}w')$ reduces to the trivial (empty) word in $G(\cS_0)$.

Let $L$ be a partial picture giving the proof that $w^{-1}w'$ is trivial in $G(\cS)$. Then $\pi(L)$ can be completed to a true picture $L_0$ for the group $G(\cS_0)$ by joining together cancelling letters in $\pi(w^{-1}w')$. By the Atomic Deformation Theorem \ref{Atomic Deformation theorem}, $L_0$ is equivalent to a sum of atoms. However, each atom $A$ for $G(\cS_0)$ can be lifted to an atom $\tilde A$ for $G(\cS)$ by definition of the atoms. Therefore, up to deformation equivalence, $L$ can be lifted to a picture $\tilde L$ for $G(\cS)$. By Corollary \ref{cor: pictures have same number of x, x-inv}, the number of vertices of $\tilde L$ having $x(\beta_m)$ as Y-letter is equal to the number of vertices having $x(\beta_m)^{-1}$ as Y-letter. This implies that the number of vertices in $L_0$ lifting to ones in $\tilde L$ having $x(\beta_m)^{-1}$ as Y-letter is equal to the number of vertices in $L_0$ lifting to ones in $\tilde L$ having $x(\beta_m)^{-1}$ as Y-letter are equal. So, the number of times $x(\beta_m),x(\beta_m)^{-1}$ occur as Y-letters in $L$ are equal. So, the number of times that $x(\beta_m),x(\beta_m)^{-1}$ occur in the word $w^{-1}w'$ are equal. So, $x(\beta_m)$ occurs the same number of times in the words $w,w'$ as claimed.
\end{proof}

\begin{proof}[Proof of Lemma \ref{lem E: beta m only commutes with hom orthogonal roots}] Recall that $\beta_m$ is the last element of an admissible set $\cS$. Lemma \ref{lem E: beta m only commutes with hom orthogonal roots} says that if $w_0$ is a positive expression for some element of $G(\cS)$ which commutes with $x(\beta_m)$ then every letter of $w_0$ commutes with $\beta_m$. To prove this, suppose not and let $w_0$ be a minimal length positive expression in the letters $\cS$ satisfying the following.
\begin{enumerate}
\item As an element of $G(\cS)$, $w_0$ commutes with $x(\beta_m)$. 
\item One of the letters of $w_0$, say $x(\beta)$, does not commute with $x(\beta_m)$. Equivalently, $\beta,\beta_m$ are not hom-orthogonal (Remark \ref{rem: bm commutes with b iff hom-orthog}).
\end{enumerate}
Clearly, $w_0$ has at least 2 letters and the first and last letter of $w_0$ do not commute with $x(\beta_m)$.

In the group $G(\cS)$ we have the relation
\[
	W=w_0 x(\beta_m) w_0^{-1}x(\beta_m)^{-1}=1.
\]
A proof of the relation $W=1$ gives a partial picture $L$ for $G(\cS)$ having the word $W$ as it boundary. Let $\beta^1,\cdots,\beta^m$ be the letters in $\cS$ in lateral order. Then $\beta_m=\beta^k$ for some $k$. Let $\beta^i,\beta^j$ be the letters which occurs in the partial picture $L$ with $i$ minimal and $j$ maximal. Then $i<j$ and $i\le k\le j$. In particular, either $i< k$ or $k<j$. By symmetry we may assume that $k<j$. Then we will use the Monomorphism Lemma \ref{monomorphism lemma} for $\omega=\beta^j\neq \beta_m$. (For $k=j$ the argument is the same using the dual lemma \ref{dual monomorphism lemma} with $\alpha=\beta^i$.)

There are two cases. Either $\lambda=\beta^j$ is a letter in $W$ or not.

\underline{Case 1}. $\lambda$ is not a letter in $W$. Then the edge set $E(\lambda)$ is a disjoint union of simple closed curves. We claim that these can all be eliminated by Lemmas \ref{monomorphism lemma} and \ref{lem: sliding lemma}. Let $\Sigma$ be any component of $E(\lambda)$. Let $\Sigma'$ be a parallel curve on the negatives side of $\Sigma$. Then $\Sigma'$ crosses only edges $E(\beta)$ where $\beta$ is hom-orthogonal to $\lambda$. Therefore, we can apply the retraction $\rho:G(\cS_-(\lambda))\to G(\cR_-(\lambda))$ to the region enclosed by $\Sigma'$ to eliminate all edges in that region which are not hom-orthogonal to $\lambda$. By the Sliding Lemma \ref{lem: sliding lemma} we can then eliminate $\Sigma$ if it is still there. Repeating this process produces a new partial picture $L'$ with boundary $W$ so that the laterally rightmost letter in $L'$ is a letter in $W$, i.e., we are reduced to Case 2.

\underline{Case 2} $\lambda=\beta^j$ is a letter in $W$. Since $j>k$, $\lambda$ is then a letter in $w_0$. The generator $x(\lambda)$ may occur several times in $w_0$ and $x(\lambda)^{-1}$ occurs in $w_0^{-1}$. Taking the first occurrence of $x(\lambda)$ in $w_0$ we can write $w_0=w_1x(\lambda)w_2$ there $x(\lambda)$ is not a letter in $w_1$. Then
\[
	W=w_1x(\lambda)w_2 x(\beta_m)w_2^{-1}x(\lambda)^{-1}w_1^{-1}x(\beta_m)^{-1}
\]
is the boundary of $L$ which is a partial picture for $G(\cS_-(\lambda))$. Since $\lambda$ is rightmost in later order, $x(\lambda)$ does not occur as a Y-letter at any of the vertices of $L$. Therefore, the edge set $E(\lambda)$ is a disjoint union of simple closed curves and disjoint arcs connecting the $x(\lambda)$ in $w_0$ to the $x(\lambda)^{-1}$ in $w_0^{-1}$. Since these arc are disjoint, the outermost such arc $\Sigma$ connects the first occurrence of $x(\lambda)$ in $w_0$ to the last occurrence of $x(\lambda)^{-1}$ in $w_0^{-1}$. Let $\Sigma'$ be an arc parallel to $\Sigma$ on its negative side. Thus $L=L_0\cup L_1$ where $L_0$ is the portion of $L$ enclosed by $\Sigma'$. Since $x(\lambda)$ is to the left of $x(\lambda)^{-1}$, $\Sigma\subset L_0$. (See the left side of Figure \ref{Fig: proof of Lemma E}.)

\begin{figure}[htbp] % begin {Fig: proof of Lemma E}
\begin{center}
\begin{tikzpicture}[scale=.85]
\begin{scope}[yshift= 3mm]
\coordinate (C) at (2,0);
\coordinate (C1) at (2.4,0);
\draw[color=blue,thick] (C) arc[start angle=0,end angle=180, radius=20mm];
\draw[color=red,dashed] (C1) arc[start angle=0,end angle=180, radius=24mm];
\end{scope}
\draw (-3,0) node[right]{$w_1x(\lambda)$};
\draw (0,0) node{$w_2x(\beta_m)w_2^{-1}$};
\draw (1.1,0) node[right]{$x(\lambda)^{-1}w_1^{-1}x(\beta_m)^{-1}$};
\draw[color=blue] (-1.2,1.4) node{$\Sigma$};
\draw[color=red] (-2,2.2) node{$\Sigma'$};
\draw (0,1) node{$L_0$};
\draw (2.5,2.2) node{$L_1$}; % end left drawing
\draw[thick,->] (4.5,1.4)--(4.9,1.4);
\begin{scope}[xshift=9cm] % right drawing
\draw (0,1) node{$\rho(L_0)$};
\draw (2.5,2.2) node{$L_1$};
\begin{scope}[yshift= 3mm]
\coordinate (C) at (2,0);
\coordinate (C1) at (2.4,0);
\draw[color=red,dashed] (C1) arc[start angle=0,end angle=180, radius=24mm];
\end{scope}
\draw (-3,0) node[right]{$w_1$};
\draw (0,0) node{$\rho\left(w_2\right)x(\beta_m)\rho\left(w_2^{-1}\right)$};
\draw (2.2,0) node[right]{$w_1^{-1}x(\beta_m)^{-1}$};
\draw[color=red] (-2,2.2) node{$\Sigma'$};
\end{scope}
\end{tikzpicture}
\caption{(Proof of Lemma \ref{lem E: beta m only commutes with hom orthogonal roots}) The partial picture $L$ for $G(\cS_-(\lambda))$ is divided into two parts $L=L_0\cup L_1$ by $\Sigma'$. Applying $\rho:G(\cS_-(\lambda))\to G(\cR_-(\lambda))$ to $L_0$ eliminates $x(\lambda)$ from the word $w_0=w_1x(\lambda)x_2$ but does not eliminage $x(\beta_m)$. Then $w_1\rho(w_2)$ commutes with $x(\beta_m)$ contradicting the minimality of $w_0$.}
\label{Fig: proof of Lemma E}
\end{center}
\end{figure}

Using the Monomorphism Lemma \ref{monomorphism lemma}, we apply the retraction $\rho$ to $L_0$. This will eliminate $\Sigma$ and all occurrences of the letter $x(\lambda)$ in $W$ giving a new relation:
\[
	w_1\rho(w_2) \rho(x(\beta_m)) \rho(w_2)^{-1}w_1^{-1} x(\beta_m)^{-1} =1
\]
or, equivalently, $w_1\rho(w_2) \rho(x(\beta_m))=x(\beta_m)w_1\rho(w_2) $.
By Lemma \ref{lem C: beta m occurs a fixed number of times} proved above, $x(\beta_m)$ occurs the same number of times in these two expressions. So, $\rho(x(\beta_m))=x(\beta_m)$. In particular, $\lambda$ is hom-orthogonal to $\beta_m$. Equivalently $x(\lambda)$ commutes with $x(\beta_m)$. So, $x(\lambda)$ is not the first letter of $w_0$ which means $w_1$ is a nontrivial word.

This gives a new word $w_0'=w_1 \rho(w_2)$ which is shorter than $w_0$, commutes with $x(\beta_m)$ and  has at least one letter (the first letter of $w_1$) which does not commute with $x(\beta_m)$. This contradicts the minimality of $w_0$ and completes the proof of Lemma \ref{lem E: beta m only commutes with hom orthogonal roots}.\end{proof}

\section{Appendix}\label{ss4}
This Appendix contains basic background material for this paper. Details can be found in \cite{Modulated} and \cite{Linearity1}

\subsection{Exceptional representations of modulated quivers}
We assume throughout the paper that $Q$ is a quiver without loops, oriented cycles or multiple edges $i\to j$ (since multiplicity of edges is included in the valuation). We recall briefly that a \und{valuation} on a quiver $Q$ is given by assigning positive integers $f_i$ to each vertex $i$ and pairs of positive integers $(d_{ij},d_{ji})$ to every arrow $i\to j$ in $Q$ having the property that $f_id_{ij}=f_jd_{ji}$. For example, the Kronecker quiver is $\bullet\xrightarrow{(2,2)} \bullet$. A \und{$K$-modulation} of a valued quiver is given by assigning a division algebra $F_i$ of dimension $f_i$ at each vertex and an $F_i\text-F_j$-bimodule $M_{ij}$ on each arrow $i\to j$ with $\dim_KM_{ij}=f_id_{ij}=f_jd_{ji}$. A representation of a modulated quiver consists of a right $F_i$-vector space $V_i$ at each vertex and an $F_j$-linear map $V_i\otimes M_{ij}\to V_j$ on each arrow $i\to j$. A representation $V$ is called a \und{brick} if its endomorphism ring is a division algebra. An \und{exceptional} module is a brick having no self-extensions. For hereditary algebras of finite type, all bricks are exceptional.

Given any module $X$ we denote by $X^\perp$ the full subcategory of $mod\text-\Lambda$ with all objects $Y$ so that 
\[
	\Hom_\Lambda(X,Y)=0=\Ext_\Lambda(X,Y)
\]
Similarly, $^\perp X$ is the category of all $\Lambda$-modules $Y$ so that $X\in Y^\perp$. An \und{exceptional sequence} of length $k$ is defined to be a sequence of exceptional modules $E_1,E_2,\cdots,E_k$ so that $E_i\in E_j^\perp$ for all $i<j$.

The \und{dimension vector} $\undim V$ of a representation of a modulated quiver is defined to be $(d_1,d_2,\cdots,d_n)$ where $d_i$ is the dimension of $V_i$ as a vector space over $F_i$. A \und{real Schur root} of the valued quiver $Q$ is defined to be the dimension vector of an exceptional module for any modulation of $Q$. This concept is known to be independent of the choice of modulation. See \cite{Modulated} for details. In this paper we assume a modulation is given.

The \und{semi-stability set} $D(V)$ of any module $V$ is defined by
\[
	D(V):=\{x\in\RR^n\,:\, \brk{x,\undim V}=0\text{ and } \brk{x,\undim V'}\le 0\text{ for all submodules } V'\subset V\}
\]where we use the bilinear pairing:
\[
	\brk{x,y}=\sum x_iy_if_i.
\]

For any real Schur root $\beta$ let $D(\beta)=D(M_\beta)$ where $M_\beta$ is the unique exceptional module with dimension vector $\beta$. In this paper we use the following refinement of the definition of $D(\beta)$ which is essentially proved in \cite{Modulated}.

\begin{thm}\label{thm: equivalent definitions of D(b)}
For $\beta$ a real Schur root and $x\in \RR^n$ so that $\left<x,\beta\right>=0$, the following are equivalent.
\begin{enumerate}
\item $\brk{x,\beta'}\le 0$ for all real Schur subroots $\beta'$ of $\beta$.
\item $\brk{x,\undim V'}\le0$ for all submodules $V\subseteq M_\beta$.
\item $\brk{x,\undim V''}\ge0$ for all quotient modules $V''$ of $M_\beta$.
\item $\brk{x,\beta''}\ge0$ for all real Schur quotient roots of $\beta$.
\end{enumerate}
\end{thm}

\begin{proof}
It is shown in \cite{Modulated} that (1) is equivalent to (2) for $x\in \ZZ^n$. This easily implies that (1) and (2) are equivalent for $x\in \QQ^n$. Taking the closure we get that (1) and (2) are equivalent for all $x\in\RR^n$. 

The equivalence $(2)\Leftrightarrow (3)$ is obvious. The equivalence $(3)\Leftrightarrow(4)$ follows from the equivalence $(1)\Leftrightarrow(2)$. Indeed, applying the duality functor $D=\Hom(-,K)$, the exceptional $\Lambda$-module $M_\beta$ and quotient module $M_{\beta''}$ become exceptional $D\Lambda$ modules with the same dimension vectors, but $DM_{\beta''}\subset DM_\beta$. So, $x\in \RR^n$ satisfies (4) for $\Lambda$ if and only if $\brk{-x,\beta''}\le 0$ for $\beta''\subset \beta$ (as $D\Lambda$-roots). Equivalently, $x\in D_\Lambda(\beta)$ using the criteria (1),(2) if and only if $-x\in D_{D\Lambda}(\beta)$ using the quotient root criteria (4),(3) respectively. So $(3)\Leftrightarrow(4)$. 
\end{proof}

Following \cite{Linearity1}, we use $g$-vectors and modified dot product in this paper instead of the Euler product used in \cite{Modulated}. and we define the \und{$g$-vector} of a module $X$ to be
\[
	g(X):=\undim P_0/rad\,P_0-\undim P_1/rad\,P_1
\]
where 
\[
	0\to P_1\to P_0\to X\to 0
\]
is the minimal projective presentation of $X$. Equivalently, $g(X)=C_\Lambda^{-1}\undim X$ where $C_\Lambda$ is the Cartan matrix of $\Lambda$. 

\begin{lem}\label{lem: Euler pairing using g-vectors}
The $g$-vector of $X$ satisfies the following for any representation $V$.
\[
	\brk{g(X),\undim V}=\dim_K\Hom_\Lambda(X,V)-\dim_K\Ext_\Lambda(X,V).
\]
In particular, $\brk{g(X),\undim V}=0$ when $X\in\,^\perp V$.
\end{lem}

\begin{proof}
This follows from the exact sequence:
\[
0\to \Hom_\Lambda(X,V)\to \Hom_\Lambda(P_0,V) \to \Hom_\Lambda(P_1,V)\to \Ext_\Lambda(X,V)\to 0
\]
and the evident fact that $\dim_K\Hom_\Lambda(P,V)=\brk{g(P),\undim V}$.
\end{proof}

This immediately gives the following.

\begin{prop}\label{prop: exceptional sequences are lin indep}
The dimension vectors of modules in an exceptional sequence are linearly independent.
\end{prop}

\begin{proof}
Suppose that $E_1,\cdots,E_k$ is an exceptional sequence. Lemma \ref{lem: Euler pairing using g-vectors} implies
\[
	\brk{g(E_j),\undim E_i}=0
\]
for all $i<j$. But $\brk{g(E_j),\undim E_j}=\dim_K \End_\Lambda(E_j)\neq0$. So, $\undim E_j$ cannot be a linear combination of $\undim E_i$ for $I<j$.
\end{proof}

The $g$-vector of a shifted projective module $P[1]$ is define by $g(P[1]):=-g(P)$.

We have the following ``Virtual Stability Theorem'' from \cite{Modulated}.

\begin{thm}\label{thm: characterization of perp M}
If $X\in\,^\perp M_\beta$ then $g(X)\in D(\beta)$. If $P\in\,^\perp M_\beta$ is projective then $g(P[1])=-g(P)\in D(\beta)$. Conversely, for any $x\in D(\beta)\cap \ZZ^n$ there is a module $X$ and a projective module $P$ so that
\begin{enumerate}
\item $x=g(X\oplus P[1])=g(X)-g(P)$.
\item $X,P\in\,^\perp M_\beta$, i.e., $\Hom(X\oplus P,M_\beta)=0=\Ext(X,M_\beta)$.
\end{enumerate}
\end{thm}

\subsection{Wide subcategories}

Recall that a full subcategory $\cW$ of an abelian category $\cA$ is \und{wide} if it is closed under extension and kernels and cokernels of morphism between objects. This implies in particular that $\cW$ is closed under taking direct summands.

Returning to the case of $mod\text-\Lambda$ for a hereditary algebra $\Lambda$, we note that $X^\perp$ is a wide subcategories for any object $X$. To see this, look at the following six term exact sequence for any short exact sequence $0\to A\to B\to C\to 0$.
\[
	0\to\Hom(X,A)\to \Hom(X,B)\to \Hom(X,C)\to \Ext(X,A)\to\Ext(X,B)\to\Ext(X,C)\to0
\]If $A,C\in X^\perp$ then we see that $B\in X^\perp$. If $B\in X^\perp$ then $\Hom(X,A)=0=\Ext(X,C)$. So, any object which is both a subobject and quotient object of an object of $X^\perp$ is also in $X^\perp$. So, $X^\perp$ is a wide subcategory of $mod\text-\Lambda$. Similarly, $^\perp X$ is a wide subcategory.

Closely related to this example is the following well-known fact. (See \cite{Linearity1} for a short proof.)

\begin{thm}\label{thm: W(x0) is wide}
Let $R$ be any subset of $\RR^n$. Then the set $\cW(R)$ of all representation $V$ so that $R\subset D(V)$ is a wide subcategory of $mod\text-\Lambda$.
\end{thm}

Consider the case when $R=\{x_0\}$ is a single point $x_0\neq0\in \RR^n$. Suppose that $\cS$ is an admissible set of real Schur roots. Recall our notation that $D(\beta)=D(M_{\beta})$ where $M_{\beta}$ is the unique exceptional module with dimension vector $\beta$.

What can we say about the set of $\beta \in \cS$ so that $x_0\in D(\beta)$?

\begin{prop}\label{prop: x0 in interior of D(a)}
Let $\alpha\in \cS$. Then $x_0$ is in the interior of $D(\alpha)$ if and only if $M_{\alpha}$ is a minimal object of the wide subcategory $\cW(x_0)$.
\end{prop}

\begin{proof}
If $x_0$ lies in the interior of $D(\alpha)$, $\brk{x_0,\gamma}<0$ for all subroots $\gamma\subsetneq \alpha$. So, $x_0\notin D(\gamma)$. So, $\alpha$ is minimal. The converse follows in the same way.
\end{proof}

A wide subcategory $\cW\subset mod\text-\Lambda$ has \und{rank $k$} if it is isomorphic to the module category of an hereditary algebra with $k$ simple modules. More concretely, such a wide subcategory contains $k$ $\Hom$-orthogonal exceptional modules forming an exceptional sequence: $X_1,X_2,\cdots,X_k$. In other words, $\Ext(X_j,X_i)=0$ for $j\ge i$. And all other objects of $\cW$ are iterated extensions of the $X_i$ with each other. From this description we see that the $X_i$ are objects of $\cW$ of minimal length, i.e., proper subobjects and proper quotient objects of the $X_i$ do not lie in $\cW$. In particular, the $X_i$ are uniquely determined by $\cW$. In general, not every wide subcategory of $mod\text-\Lambda$ has finite rank. For example, when $\Lambda$ has infinite representation type, the subcategory of regular modules is a wide subcategory of infinite rank since the Auslander-Reiten translation functor $\tau$ is an automorphism on this subcategory.

One special case of a finite rank wide subcategory which we need in this paper is the case $k=n$.

\begin{thm}\label{thm: rank n wide subcategory}
Let $(E_1,\cdots,E_n)$ be an exceptional sequence of $\Hom$-orthogonal objects in $mod\text-\Lambda$. Then all $E_i$ are simple. In particular, $mod\text-\Lambda$ is the only wide subcategory of rank $n$.
\end{thm}

\begin{proof}
This follows from the theory of exceptional sequences. By \cite{CB} and \cite{Ringel}, the action of the braid group on $n$ strands acts transitively on the set of exceptional sequences of length $n$. However, by definition, braid moves keep objects in the same wide subcategory which is the category of all objects which are iterated extensions of the $E_i$ with each other. By the theorem of \cite{CB} and \cite{Ringel}, this includes all exceptional sequences. But the sequence of simple modules of $mod\text-\Lambda$ forms an exceptional sequence. So, every simple $\Lambda$-module is in our wide subcategory. So, the wide subcategory is all of $mod\text-\Lambda$. Since the $E_i$ are minimal objects, they must all be simple.
\end{proof}

Let $\alpha_1,\cdots,\alpha_k$ be real Schur roots so that $(M_{\alpha_1},\cdots,M_{\alpha_k})$ is a sequence of $\Hom$-orthogonal sequence of modules forming an exceptional sequence. Then we denote by $\cA(\alpha_1,\cdots,\alpha_k)$, or $\cA(\alpha_\ast)$ for short, the wide subcategory of $mod\text-\Lambda$ generated by the modules $M_{\alpha_i}$. As remarked above, this is a rank $k$ wide subcategory whose objects have a filtration with subquotients $M_{\alpha_i}$. Another description is:
\[
	\cA(\alpha_1,\cdots,\alpha_k)=\,^\perp\left(  (M_{\alpha_1}\oplus\cdots\oplus M_{\alpha_k})^\perp\right)
\]
In other words, $\cA(\alpha_\ast)=\,^\perp (E_1\oplus\cdots\oplus E_{n-k})$ for any choice of a complete exceptional sequence $(E_1,\cdots,E_{n-k},M_{\alpha_1},\cdots,M_{\alpha_k})$ ending in the $M_{\alpha_i}$.

Here is another well-known fact that we need.

\begin{thm}\label{thm: wide subcategory gen by hom-orthogonal roots}
The wide subcategory $\cW=\cA(\alpha_1,\cdots,\alpha_k)$ described above contains the exceptional module $M_\beta$ if and only if $\beta$ is a nonnegative linear combination of the $\alpha_i$.
\end{thm}

\begin{proof} Necessity of this condition is clear since all objects of $\cW$ are iterated extensions of the modules $M_{\alpha_i}$. For the converse, we choose an extension of this sequence to a complete exceptional sequence 
$(E_1,\cdots,E_{n-k},M_{\alpha_1},\cdots,M_{\alpha_k})$. Then $\cW=\,^\perp(E_1\oplus\cdots\oplus E_{n-k})$. By Theorem \ref{thm: characterization of perp M}, an exceptional module $M_\beta$ lies in $\cW$ if and only if $g(\beta)\in \bigcap_jD(E_j)$. But this is a convex set. Since this condition holds for the roots $\alpha_i$, it holds for any nonnegative linear combination of the $\alpha_i$.
\end{proof}

For an admissible set of roots $\cS$, this theorem and Proposition \ref{prop: x0 in interior of D(a)} imply the following.

\begin{cor}\label{cor: W(x0) cap S}
For $x_0\neq0\in\RR^n$, let $\alpha_1,\cdots,\alpha_k$ be the elements of $\cS$ for which $M_{\alpha_i}$ is minimal in $\cW(x_0)=\{M\,:\,x_0\in D(M)\}$. Then, $\cS\cap \cW(x_0)$ is the set of elements of $\cS$ which are sums of these roots ($\beta=\sum n_i\alpha_i$ for $n_i\ge0$).
\end{cor}

\begin{proof}
Let $\beta\in \cS\cap \cW(x_0)$. So, $x_0\in D(\beta)$. If $\beta$ is not one of the $\alpha_i$ then, by Proposition \ref{prop: x0 in interior of D(a)}, $x_0\in \partial D(\beta)$. This implies that $x_0\in D(\gamma)$ for a subroot $\gamma\subsetneq \beta$. It follows that $x_0\in D(\gamma')$ for all components $\gamma'$ of the quotient root $\beta-\gamma$. These subroots and quotient roots of $\beta$ all lie in $\cS$ since $\cS$ is admissible. By induction on the length of $\beta$ we conclude that each $\gamma,\gamma'$ is a nonnegative linear combination of the $\alpha_i$. So, the same holds for their sum $\beta$.

Conversely, suppose $\beta\in\cS$ has the form $\beta=\sum n_i\alpha_i$ for $n_i\ge0$. Since $\cS$ is admissible, the modules $M_{\alpha_1},\cdots,M_{\alpha_k}$ are $\Hom$-orthogonal and form an exceptional sequence (being in lateral order). By Theorem \ref{thm: wide subcategory gen by hom-orthogonal roots}, $M_\beta$ lies in the wide subcategory $\cW(x_0)$ as claimed.
\end{proof}

%%%%%%%%%%%%%%%%%%%%%%%%%%%%%%%%%%%%%%%%%%%%%%%%
\section*{Acknowledgements} The authors thank Thomas Br\"ustle, Eric Hanson, Steve Hermes, Moses Kim, Kent Orr and Jerzy Weyman for numerous discussions about ``pictures'' and their relation to maximal green sequences. The first author acknowledges support of the Simons Foundation. Both authors are grateful to the referees for numerous very helpful comments and also for their interest in the history of this subject.

\end{document}